\documentclass[11pt,a4paper,twoside,reqno]{amsart}

\usepackage{amsmath,amssymb,amsthm}
\usepackage{cases}
\usepackage{float}
\floatstyle{boxed} 
\usepackage[utf8]{inputenc}
\usepackage[T1]{fontenc}
\usepackage{enumitem}
\usepackage{mathtools}
\usepackage{eurosym, eucal,palatino}
\usepackage{scrextend}
\changefontsizes{11pt}

\usepackage{mathrsfs}
\usepackage{xcolor}
\usepackage[colorlinks=true, linkcolor=red, citecolor=blue, urlcolor=blue,hyperfootnotes=false]{hyperref}
\usepackage{cleveref}

\usepackage{etoolbox}
\usepackage{upgreek}

\def\@begintheorem#1#2{\par\bgroup{\sc #1 \ #2. }  \it \\\ignorespace }
\def\@opargbegintheorem#1#2#3{\par\bgroup{\sc #1\ #2 \ (#3).}  \it  \ignorespace}
\def\@endtheorem{\egroup}

\protected\def\ignorethis#1\endignorethis{}
\let\endignorethis\relax

\theoremstyle{plain}
\newtheorem{theorem}{Theorem}[section]

\theoremstyle{definition}
\newtheorem{definition}[theorem]{Definition}

\theoremstyle{remark}
\newtheorem{remark}[theorem]{Remark}

\theoremstyle{plain}

\theoremstyle{plain}
\newtheorem{lemma}[theorem]{Lemma}

\theoremstyle{plain}
\newtheorem{corollary}[theorem]{Corollary}

\theoremstyle{plain}
\newtheorem{proposition}[theorem]{Proposition}

\numberwithin{equation}{section}

\newcommand{\rmi}{{i}}
\newcommand{\C}{{\mathbb C}}
\newcommand{\N}{{\mathbb N}}
\newcommand{\dom}{\mathrm{dom}\,}
\newcommand{\Z}{{\mathbb Z}}

\newcommand{\R}{{\mathbb R}}
\newcommand{\bI}{\mathbb{I}}
\newcommand{\bV}{\mathbb{V}}
\newcommand{\bT}{\mathbb{T}}
\newcommand{\bU}{\mathbb{U}}
\newcommand{\cD}{{\mathcal{D}}}
\newcommand{\ccD}{{\mathscr{D}}}
\newcommand{\ccE}{{\mathscr{E}}}
\newcommand{\V}{\mathbb{V}}
\newcommand{\cT}{\mathcal{T}}
\newcommand{\bB}{{B}}

\newcommand{\supp}{{\rm supp}}

\DeclareMathOperator{\Realpart}{Re}
\renewcommand{\Re}{\Realpart}

\DeclareMathOperator{\diag}{diag}
\DeclarePairedDelimiter{\abs}{\lvert}{\rvert}
\DeclarePairedDelimiter{\norm}{\lVert}{\rVert}

\DeclareMathOperator{\arctanh}{arctanh}

\newcommand\eg{{\it e.g.}\,}
\newcommand\ie{{\it i.e.}\,}
\newcommand\cf{{\it cf.}\,}

\title[General $\delta$--shell interactions for the two-dimensional Dirac operator]{General \boldmath{$\delta$}--shell interactions for the
  two-dimensional Dirac operator: self-adjointness and approximation}

\author[B.~Cassano]{Biagio Cassano}
\address{B.~Cassano, Department of Mathematics, Universit\`a degli
  Studi di Bari ``A.~Moro'', via Orabona 4, 70125, Bari (Italy)}
\email{biagio.cassano@uniba.it}

\author[V.~Lotoreichik]{Vladimir Lotoreichik}
\address{V. Lotoreichik, Department of Theoretical Physics, Nuclear Physics Institute, 	Czech Academy of Sciences, 25068 \v Re\v z (Czechia)}
\email{lotoreichik@ujf.cas.cz}

\author[A.~Mas]{Albert Mas}
\address{A.~Mas,
Departament de Matem\`atiques,
Universitat Polit\`ecnica de Catalunya,
Campus Diagonal Bes\`os, Edifici A (EEBE), Av. Eduard Maristany 16, 08019
Barcelona (Spain)}
\email{albert.mas.blesa@upc.edu}

\author[M.~Tušek]{Matěj Tušek}
\address{M.~Tušek, Department of Mathematics, Faculty of Nuclear Sciences and Physical Engineering, Czech Technical University in Prague, Trojanova 13, 120 00, Prague (Czechia)}
\email{matej.tusek@fjfi.cvut.cz}

\AtBeginDocument{\addtocontents{toc}{\protect\setlength{\parskip}{0pt}}}

\begin{document}

\begin{abstract}
In this work we consider the two-dimensional Dirac operator with
general local singular interactions supported on a closed curve. A
systematic study of the interaction is performed by decomposing it
into a linear combination of four elementary interactions:
electrostatic, Lorentz scalar, magnetic, and a fourth one which can be
absorbed by using unitary transformations.
We address the self-adjointness and the spectral description of the
underlying Dirac operator, and moreover we describe its approximation by Dirac operators with regular potentials.\end{abstract}
\maketitle
\setcounter{tocdepth}{1}
{
\hypersetup{linkcolor=black}
\tableofcontents
}

\let\thefootnote\relax\footnotetext[0]{\emph{Acknowledgments.}
  B.C.~is member of GNAMPA (INDAM) and he is supported by Fondo Sociale Europeo – Programma
  Operativo Nazio\-nale Ricerca e Innovazione 2014-2020,
  progetto PON: progetto AIM1892920-attivit\`a 2, linea 2.1.
 V.L.~and M.T.~are supported by the grant No.~21-07129S 
 	of the Czech Science Foundation (GA\v{C}R).
A.M.~is partially supported by MINECO grants 
MTM2017-84214-C2-1-P and MTM2017-83499 (Spain), AGAUR grant 2017-SGR-358 (Catalunya), and ERC-2014-ADG project HADE Id.\! 669689 (European Research Council). A.M.~is member of the Barcelona Graduate School of Mathematics, supported by MINECO grant MDM-2014-0445. M.T.~is supported by the project CZ.02.1.01/0.0/0.0/16\_019/0000778 from the European Regional Development Fund.}

\newpage
\section{Introduction}
In the present paper we study the two-dimensional Dirac operator with
a singular interaction supported on a closed curve.
Our main motivation is to treat the most general
local interactions. Besides electrostatic $\delta$-shell interactions and the Lorentz scalar  $\delta$-shell interactions we include 
into the analysis the magnetic $\delta$-shell interactions, which correspond to the magnetic field supported on a curve. The main two questions addressed in the present paper are self-adjointness of the underlying Dirac operator
and its approximation by Dirac operators with regular potentials. 

Recall that the Dirac operator was firstly introduced in relativistic quantum
mechanics to describe the dynamics of spin--$\frac12$ particles (see, \eg the monograph \cite{thaller}), and
was later associated to the evolution of quasi-particles in new
materials, such as the graphene (see, \eg \cite{neto2009electronic}). Dirac operators with singular interactions supported on sets of lower dimensions serve as idealized models for Dirac operators with more realistic (regular) potentials.

Hamiltonians with interactions supported on sets of zero Lebesgue measure have been studied intensively in the mathematical physics. At first the case of Schr\"odinger operators with singular interactions was investigated; see, \eg~\cite{albeverio2012solvable, BLL13, exner2007leaky}. In the recent years the focus partially shifted to the Dirac operators with singular interactions.
While for Schr\"odinger operators quadratic forms are a convenient tool to define the underlying Hamiltonian~\cite{BEKS94}, in the Dirac setting more subtle techniques are necessary due to the lack of semi-boundedness.
The case of one-dimensional Dirac operators with point-interactions is
well understood~\cite{albeverio2012solvable, CMP13, GS87,
  PR14, borrelli2019overview}. Three-dimensional Dirac operators with singular interactions
supported on surfaces are considered in, \eg \cite{amv1, amv2, amv3,
  behrndt2016spectral, behrndt2019dirac, behrndt2019self,
  behrndt2020limiting}. Finally, the two-dimensional case without the
magnetic interaction has recently been analysed
in~\cite{behrndt2019two,pizzichillo2019self}. The interest to include
the magnetic $\delta$-shell interaction stems from applications in
modern physics~\cite{PeNe_09, MaVaPe_09, FiJaTu_18}. The closely related model of magnetic links in three dimensions has been recently considered in~\cite{PSS18a, PSS18b, PSS20}.

The approximation of Dirac operators with singular interactions by Dirac operators with regular interactions provides a justification of the idealized model under consideration. In the one-dimensional setting 
the analysis is performed in~\cite{sebaklein, hughes97, hughes99,
  tusek19}, a generalization to three-dimensions has recently appeared
in \cite{klein, sphericalnote}. In the present manuscript, we modify to our setting some techniques that worked efficiently in the one-dimensional case.

Recall that the action of the two-dimensional free Dirac operator $\ccD_0$ is given by the differential expression 
\begin{equation}\label{eq:Dirac}
     D_0 := - i \sigma \cdot \nabla 
     + m \sigma_3
     =
     -i (\sigma_1 \partial_1 + \sigma_2 \partial_2)
     + m \sigma_3,
\end{equation}
where $m \in \R$ is the mass and  $\sigma_1,\,\sigma_2$, and $\sigma_3$  are the {Pauli matrices}:  
\begin{equation}\label{eq:Pauli}
 \quad{\sigma}_1 =
\begin{pmatrix}
0 & 1\\
1 & 0
\end{pmatrix},\quad {\sigma}_2=
\begin{pmatrix}
0 & -i\\
i & 0
\end{pmatrix},
\quad{\sigma}_3=
\begin{pmatrix}
1 & 0\\
0 & -1
\end{pmatrix}.
\end{equation}
It is self-adjoint on $\dom \ccD_0 := H^1(\R^2;\C^2)\subset L^2(\R^2;\C^2)$ and essentially
self-adjoint on $C_c^{\infty}(\R^2;\C^2)$.
The  spectrum of $\ccD_0$ is purely absolutely continuous  and
\begin{equation*}
  \sigma(\ccD_0)= \sigma_{ac}(\ccD_0) = (-\infty, -\abs{m}] \cup [\abs{m},+\infty).
\end{equation*}

The interaction under consideration will be supported on
the boundary $\Sigma := \partial\Omega$ of
a $C^\infty$-smooth 
bounded simply connected open set $\Omega \subset \R^2$.
The curve $\Sigma$ splits the Euclidean space into disjoint union
$\R^2 = \Omega_+ \cup \Sigma \cup \Omega_-$, where
$\Omega_+ := \Omega$ and $\Omega_-:=\R^2 \setminus \overline{\Omega_+}$.
We will call the curve $\Sigma$ a \emph{shell}.
Let us denote the outer unit normal to
$\Omega_+$ and the unit vector
tangent to the boundary $\Sigma$ in $x \in \Sigma$ by $\mathbf{n}\equiv(n_1,n_2)= \mathbf{n}(x)$ and $\mathbf{t}\equiv (t_1,t_2) =\mathbf{t}(x) $, respectively. For definiteness, we put $t_1=-n_2$ and $t_2 = n_1$.
For any $\C^2$--valued
function $f$ defined on $\R^2$, we set $f_\pm = f \upharpoonright
\Omega_\pm$. When it is defined in a suitable sense, we denote $\mathcal{T}_\pm^D f_\pm$ the Dirichlet trace of
$f_\pm$ at $\Sigma$, and we define the
distribution $\delta_\Sigma f$ by
\begin{equation*}
  \langle \delta_\Sigma f, \varphi \rangle
  :=
  \int_\Sigma \frac12 (\mathcal{T}_+^D f_+ + \mathcal{T}_-^D f_-)
  \cdot \varphi\, ds,
  \quad \text{ for all }\varphi \in C_c^{\infty}(\R^2;\C^2),
\end{equation*}
where $ds$ means integration with respect to the arc-length of $\Sigma$.

We are interested in the Dirac operator 
in  $L^2(\R^2;\C^2)$ given by the formal expression
\begin{equation}\label{eq:Dirac.shells}
  \begin{split}
    - i \sigma \cdot \big( \nabla  + i (\lambda \mathbf{t} + \omega \mathbf{n})
     \delta_\Sigma\big)
     + (\eta \bI_2 + \tau \sigma_3)\delta_\Sigma
     + m \sigma_3&
     \\
      = 
     \,  D_0
     + \big( \eta \bI_2 + \tau
     \sigma_3 + \lambda (\sigma\cdot \mathbf{t}) + \omega (\sigma\cdot \mathbf{n}) \big)\delta_\Sigma&
   \end{split}
 \end{equation}
where $\eta,\,\tau,\,\lambda,\,\omega$ are smooth real-valued functions,  
and where we used the notation
\begin{equation}\label{eq:notation}
\sigma \cdot A := \sigma_1A_1 + \sigma_2A_2
\quad \text{ for } A \in \C^2.
\end{equation}
For any given $x \in \Sigma$ the matrices $\bI_2,\, \sigma_3,\,
\sigma\cdot \mathbf{t}(x)$, and $\sigma\cdot\mathbf{n}(x)$ constitute a
basis of the Hermitian $2\times 2$ matrices, so at every point 
there is the
most general Hermitian matrix as a coefficient of $\delta_\Sigma$. 
The \emph{electrostatic $\delta$-shell interaction} $\eta \bI_2 \delta_\Sigma$ 
and the \emph{Lorentz scalar $\delta$-shell interaction} $\tau \sigma_3
\delta_\Sigma$ describe a distribution of charges and masses on the curve
$\Sigma$, respectively.
The novelty in our treatment is the \emph{magnetic $\delta$-shell interaction}
$\lambda(\sigma\cdot \mathbf{t}) \delta_\Sigma$, which describes a magnetic field supported on $\Sigma$.
We remark that the vector potential associated with the latter interaction is given by ${\bf A}_\Sigma = \lambda(t_1\delta_\Sigma,t_2\delta_\Sigma)$ and we will show in Appendix~\ref{app:magnetic_field} that the underlying magnetic field is given by $B_\Sigma=\lambda\partial_{\mathbf{n}}\delta_\Sigma$, where $\partial_{\mathbf{n}}\delta_\Sigma$ stands for the double layer distribution. 
Finally, we prove that, under some restrictions on parameters $\eta,\,\tau,\,\lambda,\,\omega$,   the interaction term
$\omega (\sigma\cdot
\mathbf{n}) \delta_\Sigma$ can be gauged
away in the spirit of~\cite{mas2017dirac,PANKRASHKIN2006207}; see \Cref{thm:w.gauged} for details. In particular, this term can be always gauged away when all parameters are constant. Due to this observation we may focus on the case when $\omega=0$ and other parameters are smooth real-valued functions.

The self-adjoint operator $ \ccD_{\eta,\tau,\lambda}$ associated with
the formal expression~\eqref{eq:Dirac.shells} with $\omega =0$ is constructed in Section~\ref{sec:core} rigorously as a self-adjoint extension of the symmetric operator with infinite deficiency indices acting in the Hilbert space $L^2(\R^2;\C^2)$ and given by
\[
	Sf = (-i\sigma\cdot\nabla + m\sigma_3)f,\qquad \dom S= H^1_0(\R^2\setminus\Sigma;\C^2).
\]
To this aim we build an ordinary boundary triple for $S^*$, which is a modification of the boundary triple constructed in~\cite{behrndt2019two}. This modification is necessary to treat the magnetic $\delta$-shell interaction. In this construction the operator
$\ccD_{\eta,\tau,\lambda}$ acts as $D_0$ on $\Omega_\pm$ and is subject to the local boundary conditions on $\Sigma$, which involves the parameters $\eta,\tau,\lambda$, the tangential vector ${\bf t}$,
and the normal vector ${\bf n}$.
The construction boils the question of self-adjointness of $ \ccD_{\eta,\tau,\lambda}$ down to self-adjointness of a certain first-order pseudo-differential operator on $\Sigma$.
The latter is shown by conventional techniques under the condition
\begin{equation}\label{eq:non-critical}
	\left(\frac{\eta^2-\tau^2-\lambda^2}{4}-1\right)^2-\lambda^2\ne 0 \quad \text{everywhere on } \Sigma.
\end{equation}
The case when the above expression on the left-hand side vanishes is called critical and needs a special treatment. In the present paper we cover in Section~\ref{sec:critical} the special sub-case of purely magnetic critical shell interaction ($\eta = \tau = 0$ and $\lambda = \pm 2$), in which case $\ccD_{0,0,\pm 2}$ is defined as a self-adjoint operator by a different and more direct method. It is also remarkable that the condition $\eta^2-\tau^2-\lambda^2 = -4$ is necessary and sufficient for the confinement to take place, where by confinement we understand that the operator $\ccD_{\eta,\tau,\lambda}$ can be decomposed into the orthogonal sum with respect to the decomposition $L^2(\R^2;\C^2) = L^2(\Omega_+;\C^2)\oplus L^2(\Omega_-;\C^2)$; {\it cf.} Section~\ref{ssec:confining} for details.
In particular, the choice $\eta = \tau = 0$ and $\lambda = \pm 2$ gives rise to \emph{zig-zag} boundary conditions.

Finally, we will find in Section~\ref{sec:approximation} approximations by regular potentials in the strong
resolvent sense for the Dirac operator with $\delta$-shell potentials
in the non-critical and non-confining case, \ie when $\eta^2-\tau^2-\lambda^2\ne-4$ everywhere on $\Sigma$ and~\eqref{eq:non-critical}  holds true.
To do so, we explicitly construct regular symmetric 
potentials $  \bV_{\eta,\tau,\lambda;\epsilon} \in
L^\infty(\R^2;\C^{2\times 2})$ supported on a
tubular $\epsilon$--neighbourhood
of $\Sigma$ and such that
\begin{equation*}
\bV_{\eta,\tau,\lambda;\epsilon} \xrightarrow[\epsilon \to 0]{}
(\eta\bI_2 + \tau \sigma_3 + \lambda(\sigma\cdot \mathbf{t}))
\delta_\Sigma
\quad \text{ in the sense of distributions},
\end{equation*}
and we investigate the strong resolvent limit of $\ccD_0 +
\bV_{\eta,\tau,\lambda;\epsilon}$ as $\epsilon\to 0$.
It turns out that $\ccD_0 +
\bV_{\eta,\tau,\lambda;\epsilon} \xrightarrow[\epsilon \to 0]{}
\ccD_{\hat{\eta,}\hat{\tau},\hat{\lambda}}$  for appropriate
$\hat{\eta,}\hat{\tau},\hat{\lambda} \in C^\infty(\Sigma;\R)$ that are in general different
from the starting $\eta,\tau,\lambda$, but are expressed  explicitly in terms of  them.
This phenomenon was observed firstly in the one-dimensional case \cite{sebaklein} and then in the three-dimensional setting \cite{klein}
--we say in this situation that a \emph{renormalization of the
  coupling constants} occurs.

We finish this introduction pointing out that when concluding the
preparation of this manuscript we learnt that the three dimensional
analogue of the magnetic $\delta$--shell interaction introduced here
was being considered in the non-published work \cite{benhellal2021}; the reader may see \Cref{sec:higher.dimensions} for more details.

\subsection*{Organization of the paper} 
In Section~\ref{sec:Main.results} we formulate and discuss all the main results of the present paper. Section~\ref{sec:preliminaries} contains preliminary material that is used throughout the paper. In Section~\ref{sec:gauge.w} we obtain spectral relations for the point spectrum of the Dirac operator with $\delta$-shell interactions under special transforms of the interaction strengths
and, moreover, we show how the fourth interaction $\omega(\sigma\cdot\mathbf{n})\delta_\Sigma$ can be eliminated by a properly constructed unitary transform. Further, in Section~\ref{sec:confinement} we provide a condition on the interaction strengths, which
gives the confinement. In Section~\ref{sec:core} we analyse the non-critical case, prove self-adjointness of the underlying Dirac operator and obtain its basic spectral properties.
Self-adjointness and spectral properties of the Dirac operator with purely magnetic critical interaction are investigated in Section~\ref{sec:critical}. In Section~\ref{sec:approximation} we construct strong resolvent approximations of Dirac operators with $\delta$-shell interactions by sequences of Dirac operators with suitably scaled regular potentials. Possible generalization for higher dimensions is briefly discussed in Section~\ref{sec:higher.dimensions}.

The paper is complemented by two appendices. In Appendix~\ref{sec:Appendix} we focus on exponentials of $2\times2$ matrices of a special structure. Finally, in Appendix~\ref{app:magnetic_field} we compute the magnetic field associated with the magnetic $\delta$-shell interaction.

\section{Main results}
\label{sec:Main.results}
We briefly discuss here the main results of this paper, referring to
the various sections below for more detailed results. For an open set
$\Omega \subset \R^2$, we define 
\begin{equation}\label{eq:defn.HsigmaOmega}
  H(\sigma,\Omega) :=
  \{f \in L^2(\Omega;\C^2) \mid
  \sigma \cdot \nabla
  f \in L^2(\Omega;\C^2)\}.
\end{equation}  
For any $f = f_+ \oplus f_- \in   H(\sigma,\Omega_+) \oplus
H(\sigma,\Omega_-) \subset
L^2(\Omega_+;\C^2) \oplus L^2(\Omega_-;\C^2)
\equiv L^2(\R^2;\C^2)$, 
it was shown in \cite{benguria2017self} that $f_\pm$ admit Dirichlet
traces $\mathcal{T}_\pm^D f_\pm$ in $H^{-\frac12}(\Sigma;\C^2)$, see
\Cref{sec:trace} for details.

Given $\eta,\,\tau,\,\lambda,\,\omega \in C^\infty(\Sigma;\R)$ we define
\begin{equation}\label{eq:defn.D.shell.domain_w}
  \begin{split}
    &\dom(\ccD_{\eta,\tau,\lambda,\omega}) := 
    \big\{  f = f_+ \oplus f_- \in H(\sigma,\Omega_+) \oplus
    H(\sigma,\Omega_-)\mid \\
    &i (\sigma\cdot \mathbf{n})(\mathcal{T}_-^D f_- -\mathcal{T}_+^D  f_+) 
     =
    \tfrac12 (\eta \bI_2 + \tau \sigma_3 + \lambda (\sigma\cdot
    \mathbf{t})
    + \omega (\sigma\cdot\mathbf{n}))
    (\mathcal{T}_-^D f_- + \mathcal{T}_+^D
    f_+)
    \big\},
\end{split}
\end{equation}
and 
\begin{equation}\label{eq:defn.D.shell_w}
  \ccD_{\eta,\tau,\lambda,\omega} f
  := D_0 f_+
  \, \oplus \,
  D_0 f_-,
  \quad
  \text{for all }f \in \dom(\ccD_{\eta,\tau,\lambda,\omega}).
\end{equation}
In \eqref{eq:defn.D.shell.domain_w} the condition on $\mathcal{T}_\pm^D f_\pm$ is understood in $H^{-\frac12}(\Sigma;\C^2)$. By means of an integration by parts it can be seen (see
\eg~\cite{ourmieres2019dirac}) that
$\ccD_{\eta,\tau,\lambda,\omega}$ is the operator representing the formal differential expression \eqref{eq:Dirac.shells}.
    
Since most of this article focuses on the case 
$\omega=0$, due to the results presented in \Cref{sec:gauge.w}, for the sake of brevity we also set
$\ccD_{\eta,\tau,\lambda}:=\ccD_{\eta,\tau,\lambda,0}$, \ie
\begin{equation}\label{eq:defn.D.shell.domain}
  \begin{split}
    &\dom(\ccD_{\eta,\tau,\lambda}) := 
    \big\{  f = f_+ \oplus f_- \in H(\sigma,\Omega_+) \oplus
    H(\sigma,\Omega_-)\mid \\
    &i (\sigma\cdot \mathbf{n})(\mathcal{T}_-^D f_- -\mathcal{T}_+^D  f_+) 
     =
    \tfrac12 (\eta \bI_2 + \tau \sigma_3 + \lambda (\sigma\cdot
    \mathbf{t}))
    (\mathcal{T}_-^D f_- + \mathcal{T}_+^D
    f_+)
    \big\},
\end{split}
\end{equation}
and
\begin{equation}\label{eq:defn.D.shell}
  \ccD_{\eta,\tau,\lambda} f
  := D_0 f_+
  \, \oplus \,
  D_0 f_-,
  \quad
  \text{for all }f \in \dom(\ccD_{\eta,\tau,\lambda}).
\end{equation}

Finally, we denote
\begin{equation}\label{eq:defn.d}
  d:= \eta^2 - \tau^2 - \lambda^2\in C^\infty(\Sigma;\R).
\end{equation}

\subsection{Reduction to $\omega=0$}
In \Cref{sec:gauge.w} we will prove the following result.

\begin{theorem}\label{thm:w.gauged}
Given $\omega\in\R$ and $\eta,\,\tau,\,\lambda \in C^\infty(\Sigma;\R)$ such that 
$d:= \eta^2 - \tau^2 - \lambda^2$ is a constant function on $\Sigma$, let 
$X$ be a solution to 
\begin{equation}\label{eq:solX}
dX^2-4+(4+\omega^2-d)X=0,
\end{equation}
and 
\begin{equation}\label{eq:def.z}
z:=\frac{dX^2+4}{X(4+d-\omega^2+4\omega i)}.
\end{equation}
Then, $X\in\R\setminus\{0\}$,  $z\in\C$ satisfies $|z|=1$, and $\ccD_{\eta,\tau,\lambda,\omega}=U_z\ccD_{X\eta,X\tau,X\lambda,0}U_{\overline{z}}$, where 
$$ U_z \varphi:=\chi_{\Omega_+}\varphi+z\chi_{\Omega_-}\varphi\quad(\text{for all }\varphi\in L^2(\R^2;\C^2))$$
is unitary in $L^2(\R^2;\C^2)$. 
\end{theorem}
  
Roughly speaking, \Cref{thm:w.gauged} implies that a spectral study for
$\ccD_{\eta,\tau,\lambda} =\ccD_{\eta,\tau,\lambda,0}$  suffices to treat the general case $\ccD_{\eta,\tau,\lambda,\omega}$, hence the formal term 
$\omega (\sigma\cdot\mathbf{n})\delta_\Sigma$ in the $\delta$--shell
interaction is indeed superfluous. In a classical (absolutely
continuous) framework one would say that this term can be gauged
away: this is reminiscent of a similar effect for 
 magnetic potentials in the Coulomb gauge, 
see \cite[Remark 1.5]{cassano2018self},  \cite{cassano2019boundary,boussaid2011virial,fanelli2009magnetic,fanelli2009non}.
In \Cref{sec:gauge.w} we show that the unitary transform $U_z$ can always be taken different from the identity except for the case $(d,\omega)=(-4,0)$, which corresponds to confining $\delta$--shell interactions, see \Cref{sec:confinement}. In particular, $\ccD_{\eta,\tau,\lambda,\omega}$ never yields confinement if $(d,\omega)\neq(-4,0)$.

At the end of \Cref{sec:gauge.w} we find some
isospectral transformations as a byproduct of our result, and we
describe the charge conjugation properties of the operator $\ccD_{\eta,\tau,\lambda}$.

\subsection{The non-critical case}
We say that we are in the \emph{non-critical} case when
\eqref{eq:non-critical} holds true, \ie
everywhere on $\Sigma$
\begin{equation}\label{eq:non-critical.condition}
  \mathfrak{C}(\eta,\tau,\lambda)(x) := \left( \frac{d}{4}-1 \right)^2 -\lambda^2
= \left( \frac{\eta^2 - \tau^2 - \lambda^2}{4}-1 \right)^2 -\lambda^2
  \neq 0. 
\end{equation}
In the following theorem we gather the properties of
$\ccD_{\eta,\tau,\lambda}$ in the case when
\eqref{eq:non-critical.condition} holds true.
We point out that the non-magnetic case ($\lambda = 0$) has been
already treated in \cite[Theorem 1.1]{behrndt2019two}
for constant $\eta,\tau\in \R$.
\begin{theorem}\label{thm:D.shell.introduction}
  Let $\eta,\tau,\lambda \in C^\infty(\Sigma;\R)$ and let $\ccD_{\eta,\tau,\lambda}$ be defined as in
  \eqref{eq:defn.D.shell.domain} and \eqref{eq:defn.D.shell}. Moreover, let either $d(x)\neq 0$ for all $x\in \Sigma$,
  or let $\eta,\tau,\lambda $ be constant and such that $d = 0$.
  If \eqref{eq:non-critical.condition} holds true
  then $\ccD_{\eta,\tau,\lambda}$ is self-adjoint in $L^2(\R^2;\C^2)$
  with $\dom \ccD_{\eta,\tau,\lambda} \subset H^1(\R^2 \setminus
  \Sigma;\C^2)$. The essential spectrum of  $\ccD_{\eta,\tau,\lambda}$ is 
  \begin{equation*}
    \sigma_{ess}(\ccD_{\eta,\tau,\lambda}) =
    (-\infty, -\abs{m}] \cup [\abs{m}, +\infty)
  \end{equation*}
  and its discrete spectrum is finite.
\end{theorem}
The proof of \Cref{thm:D.shell.introduction} mimics the strategy of
the proof of \cite[Theorem 1.1]{behrndt2019two}, taking into account
the necessary modifications to treat the additional interaction
$\lambda(\sigma\cdot \mathbf{t})\delta_\Sigma$. 
It is provided in \Cref{sec:core}, where we also show a Krein-type resolvent formula, an abstract
version of the Birman-Schwinger principle, and
obtain the spectral properties of $\ccD_{\eta,\tau,\lambda}$.

\subsection{Confining $\delta$--shell interactions}
\label{ssec:confining}
If  $d = -4$ everywhere on $\Sigma$, the phenomenon of \emph{confinement} arises.
Physically, this means that a particle initially located in
$\Omega_\pm$ can not escape this region during the quantum evolution associated with
$\ccD_{\eta,\tau,\lambda}$. We describe the corresponding Hamiltonian in the following theorem, which will be proved in Section~\ref{sec:confinement}.

\begin{theorem}\label{thm:confinement.introduction}  
  Let $\eta,\tau,\lambda \in C^\infty(\Sigma;\R)$ and let $\ccD_{\eta,\tau,\lambda}$ be defined as in
  \eqref{eq:defn.D.shell.domain}, \eqref{eq:defn.D.shell}.
  If $d = -4$ everywhere on $\Sigma$ then $\ccD_{\eta,\tau,\lambda}$ decouples in the direct sum
  \begin{equation*}
    \ccD_{\eta,\tau,\lambda} = \ccD_{\eta,\tau,\lambda}^+ \oplus \ccD_{\eta,\tau,\lambda}^-,
  \end{equation*}
  with
  \begin{equation}\label{eq:defn.D.shell.domain.confinement}
    \begin{split}
      &\dom \ccD_{\eta,\tau,\lambda}^\pm := 
      \big\{ f_\pm \in H(\sigma,\Omega_\pm) \mid
       \left[\pm i (\sigma \cdot \mathbf{n}) +
        \tfrac12 (\eta \bI_2 + \tau \sigma_3 + \lambda \, (\sigma \cdot \mathbf{t})) \right] 
      \mathcal{T}_\pm^D f_{\pm} = 0
      \big\}, \\
      & \ccD_{\eta,\tau,\lambda}^\pm f_\pm := D_0 f_\pm,
      \quad \text{ for all }f_\pm \in \dom \ccD_{\eta,\tau,\lambda}^\pm.
    \end{split}
  \end{equation}
\end{theorem}

Let us look closer at $\ccD_{\eta,\tau,\lambda}^+$, the Dirac operator on
$\Omega_+$.
If $\eta = 0$, $d=-4$ implies that
there exists $\theta \in C^\infty(\Sigma;\R)$ such that $\sin\theta =
-\tfrac{\lambda}{2}$ and $\cos\theta = \tfrac{\tau}{2}$.
Since $  i  (\sigma \cdot \mathbf{n})\sigma_3 =  \sigma \cdot
\mathbf{t}$ and $(\sigma\cdot\mathbf{n})^2 = \bI_2$ (see \eqref{eq:Pauli.squares}, \eqref{eq:sigma.t}, and \eqref{eq:sigma_n.t} below),
we may rewrite the condition 
for $f_+$ in \eqref{eq:defn.D.shell.domain.confinement}  as
\begin{equation}\label{eq:quantum.dot.BC}
  \big[\bI_2 - \cos\theta (\sigma\cdot \mathbf{t}) - \sin\theta\sigma_3 \big]
  \mathcal{T}_+^D f_+ = 0.
\end{equation}
These are the \emph{quantum
  dot} boundary conditions (see
\cite{benguria2017self,benguria2017spectral, pizzichillo2019self}
and references therein).
In particular, we have the \emph{infinite mass} boundary conditions
when $(\eta,\tau,\lambda) = (0,\pm 2, 0)$ and the \emph{zig-zag}
boundary conditions when $(\eta,\tau,\lambda) = (0,0,\pm 2)$. 
By means of confinement with the electric and Lorentz scalar
$\delta$--shell interactions only, it is possible to realise the quantum dot
boundary conditions \eqref{eq:quantum.dot.BC} for 
$\theta\in [0,2\pi) \setminus\{\tfrac{\pi}{2}, \tfrac{3\pi}{2}\}$,
that is all the possible ones except
the zig-zag boundary condition,  see \cite[Remark
4.2]{behrndt2019two}.
Considering also the magnetic $\delta$--shell interaction then allows to
 describe every Dirac operator on a domain with
quantum dot boundary conditions as a Dirac operator with a
$\delta$--shell interaction.

For any choice of parameters $\eta,\tau,\lambda$
such that, everywhere on $\Sigma$,  $d=-4$  and $(\eta,\tau,\lambda)\neq (0,0,\pm2)$ the operator
$\ccD_{\eta,\tau,\lambda}$ is already described in \Cref{thm:D.shell.introduction}.
The cases $(\eta,\tau,\lambda) = (0,0,\pm2)$ are \emph{critical}, because we have
$\mathfrak{C}(\eta,\tau,\lambda)=0$. They are discussed in the following subsection.

\subsection{The critical case}
In \cite[Theorem 1.2]{behrndt2019two}, the critical cases when  $\mathfrak{C}(\eta,\tau,\lambda)= \lambda =
0$ and $\eta,\,\tau\in\R$ are described, namely, the self-adjointness is proved and the spectral properties are analysed. We complement this result by analysing
the case $\mathfrak{C}(\eta,\tau,\lambda) = \eta = \tau = 0$, \ie we consider the
purely magnetic critical interactions $\pm 2 (\sigma \cdot \mathbf{t})
\delta_\Sigma$.
In this case we prove the self-adjointness and give a detailed
description of the spectrum: we take the advantage of the phenomenon of confinement and the decomposition in
\Cref{thm:confinement.introduction} and we adapt the analysis of the massless Dirac operator on a domain with the zig-zag boundary
conditions given in \cite{schmidt1995remark}
to the case with a mass.

 \begin{theorem}\label{thm:critical}
	The operator $\ccD_{0,0,\lambda}$ with $\lambda \in \{-2,2\}$ is self-adjoint.
	The restriction $\ccD_{0,0,\lambda}\upharpoonright H^1(\mathbb{R}^2\setminus\Sigma;\mathbb{C}^2)$ is essentially self-adjoint. It holds that $\dom \ccD_{0,0,\lambda} \not\subset H^s(\mathbb{R}^2\setminus\Sigma;\mathbb{C}^2)$ for any $s >0$. Finally, the spectrum of $\ccD_{0,0,\lambda}$ is characterised as follows:
	\begin{enumerate}[label=$({\roman*})$]
        \item \label{rm_(iii)} $\sigma(\ccD_{0,0,\lambda}) = (-\infty,-|m|]\cup [|m|,\infty)$;
        \item \label{rm_(i)} 
        $\pm m$ are eigenvalues of infinite multiplicity;
        \item \label{rm_(ii)}
          there is a sequence of (embedded) eigenvalues $\{\pm\sqrt{m^2+\lambda_k}\}_{k\ge 1}$,
          where $\lambda_k$ are the eigenvalues of the Dirichlet Laplacian on $\Omega_+$
          enumerated in non-decreasing order and counted with multiplicities.
        \end{enumerate}	
\end{theorem}	
The proof of \Cref{thm:critical} is provided in \Cref{sec:critical}.
Note that the presence of embedded eigenvalues was already observed in the
non-critical confining case in the three dimensional setting in \cite[Th.~3.7]{amv2}, \cite[Prop.~3.3]{behrndt2020limiting}.
 
\subsection{Approximation of $\delta$-shell interactions by regular
  potentials}\label{ssec:approximation}

Approximations of $\delta$-shell interactions by more realistic regular potentials provide a justification for the idealized models.
In the present paper, we find approximations by regular potentials in the strong
resolvent sense for the Dirac operator with $\delta$--shell potentials
in the non-critical and non-confining case, that is for $\ccD_{\eta,\tau,\lambda}$ when $\mathfrak{C}(\eta,\tau,\lambda) \neq 0$ and
$d \neq -4$ everywhere on $\Sigma$. 
To this purpose, for $\eta,\tau,\lambda \in C^\infty(\Sigma;\R)$ we construct regular symmetric potentials $  \bV_{\eta,\tau,\lambda;\epsilon} \in
L^\infty(\R^2;\C^{2\times 2})$ supported on an $\epsilon$--neighborhood
of $\Sigma$ and such that
\begin{equation*}
  \bV_{\eta,\tau,\lambda;\epsilon} \xrightarrow[\epsilon \to 0]{}
  (\eta\bI_2 + \tau \sigma_3 + \lambda(\sigma\cdot \mathbf{t}))
  \delta_\Sigma
  \quad \text{ in the sense of distributions }
\end{equation*}
and we investigate the strong resolvent limit of $\ccD_0 +
\bV_{\eta,\tau,\lambda;\epsilon}$ as $\epsilon\to 0$.
It turns out that $\ccD_0 +
\bV_{\eta,\tau,\lambda;\epsilon} \xrightarrow[\epsilon \to 0]{}
\ccD_{\hat{\eta,}\hat{\tau},\hat{\lambda}}$  for appropriate
$\hat{\eta,}\hat{\tau},\hat{\lambda} \in C^\infty(\Sigma;\R)$ that are in general different
from the starting $\eta,\tau,\lambda$.

In the three-dimensional setting \cite{klein}, the proof of the strong resolvent convergence is an adaptation to the relativistic scenario of the approach used in \cite{approximation} for the case of Schr\"odinger operators with $\delta$-shell interactions. In \cite{approximation}, the co-dimension of the shell is strictly smaller than the order of the differential operator (the Laplacian). As a consequence, the singularities of the kernels of the boundary integral operators used in \cite{approximation} are weak enough to be controlled uniformly along the approximation procedure. This, in particular, leads to the convergence in the norm resolvent sense in the case of the Schr\"odinger operator. However, in the case of the Dirac operator, the co-dimension of the shell is exactly the same as the order of the differential operator. This has an important effect on the nature of the corresponding boundary integral operators, which now are singular integral operators instead of compact. Due to this new obstruction with respect to the Schr\"odinger case, the approach used in \cite{approximation} was adapted in \cite{klein} to the Dirac case to show the convergence in the strong resolvent sense assuming uniform smallness of the approximating potentials. 
Nevertheless, the question of strong resolvent convergence can also be addressed by more direct methods, which do not require any smallness assumption on the approximating potentials, such as by proving the convergence in the strong graph limit sense and then applying Theorem VIII.26 of \cite{reedsimon1}, which says that, in the self-adjoint setting, the strong graph convergence and the strong resolvent convergence are equivalent.

Originally, this approach was used in the one-dimensional setting \cite{hughes97,hughes99}. In this way, one can find approximating potentials for any type of $\delta$-potential. The norm resolvent convergence of approximations is harder to tackle. However, since in the one-dimensional case one can perform very explicit calculations with the resolvents  of  the approximations and the resolvents of their limit operators, it can be proved as well \cite{sebaklein,tusek19}. In the present work, we  will modify the ideas of \cite{hughes97} to get the sequence of approximating potentials for  general linear combinations of $\delta$-shell interactions, not only for the purely electrostatic or purely Lorentz scalar $\delta$-shell interactions as in \cite{klein}. It will converge in the strong resolvent sense, without any smallness  assumption on the approximating potentials.
We expect that a similar approach can be applied in the three-dimensional case as well.

\subsubsection*{Definition of $\bV_{\eta,\tau,\lambda;\epsilon}$}
In order to describe the approximating potentials $\bV_{\eta,\tau,\lambda;\epsilon}$ explicitly, we will introduce an additional notation,
referring to \Cref{sec:tubular.neighborhoods} for details.
For $\beta >0$, $\Sigma_\beta
:= \{x \in \R^2 \mid \mathrm{dist}(x,\Sigma) < \beta\}$ is
the tubular neighborhood of $\Sigma$ of width $\beta$.
If $\beta >0$ is sufficiently small,
$\Sigma_{\beta}$ is parametrized as
\begin{equation*}
  \Sigma_\beta =
  \{x_\Sigma + p \mathbf{n}(x_\Sigma) \mid x_\Sigma \in \Sigma, \, p
  \in (-\beta,\beta) \}.
\end{equation*}
Furthermore, let
\begin{equation*}
  h \in L^{\infty}(\R;\R),
  \quad \text{ with }\supp \, h \subset (-1,1), 
  \quad \int_{-1}^{1} h(t)\, dt = 1.
\end{equation*}
The function $h$ will determine the transverse profile of
$\bV_{\eta,\tau,\lambda;\epsilon}$.
For $0< \epsilon<\beta$, let
\begin{equation*}
  h_\epsilon(p) := \frac{1}{\epsilon}h\left(\frac{p}{\epsilon}\right),
  \quad \text{ for all } p \in \R.
\end{equation*}
We have $\supp \, h_\epsilon \subset (-\epsilon,\epsilon)$ 
and $\lim_{\epsilon\to 0}h_\epsilon=\delta_0$
in the
sense of the distributions $\mathcal{D}'(\R)$, where $\delta_0$ is the Dirac $\delta$-function supported at the origin, 
For $\eta,\tau,\lambda \in C^\infty(\Sigma;\R)$, let
\begin{equation}\label{eq:defb.B}
  B_{\eta,\tau,\lambda} \in C^{\infty}(\Sigma;\C^{2\times 2}),
  \quad B_{\eta,\tau,\lambda}(x_\Sigma) := (\eta \bI_2 + \tau \sigma_3 + \lambda (\sigma\cdot\mathbf{t}))(x_\Sigma).
\end{equation}
The function $B_{\eta,\tau,\lambda}$ will encode the matrix structure of
the approximating potentials; for all $x_\Sigma\in\Sigma$, the matrix
$B_{\eta,\tau,\lambda}(x_\Sigma)$ is symmetric.
Finally, for  any $\epsilon\in(0,\beta)$, we define the symmetric approximating
potentials $\bV_{\eta,\tau,\lambda;\epsilon}\in L^{\infty}(\R^2;\C^{2\times 2})$ as follows:
\begin{equation}\label{eq:defn.bV}
    \bV_{\eta,\tau,\lambda;\epsilon}(x) :=
    \begin{cases}
      B_{\eta,\tau,\lambda}(x_\Sigma) h_\epsilon(p),
      \quad &\text{ if }x = x_\Sigma + p \mathbf{n}(x_\Sigma) \in
      \Sigma_{\beta},\\
      0, \quad &\text{ if }x\in\R^2\setminus \Sigma_{\beta}.
    \end{cases}
\end{equation}
It is easy to see that
$\lim_{\epsilon\to 0}\V_{\eta,\tau,\lambda;\epsilon}= B_{\eta,\tau,\lambda}
\delta_{\Sigma}$ in
$\mathcal{D}'(\R^2;\C^{2\times 2})$.

For $0<\epsilon<\beta$, we define the family of Dirac operators $\{\ccE_{\eta,\tau,\lambda;\epsilon}\}_{\epsilon}$ 
as follows:
\begin{equation}\label{eq:defn.Eepsilon}
  \begin{split}
  & \dom \ccE_{\eta,\tau,\lambda;\epsilon} := \dom \ccD_0=H^1(\R^2;\C^2),
  \\ & \ccE_{\eta,\tau,\lambda;\epsilon} \psi := \ccD_0 \psi + \bV_{\eta,\tau,\lambda;\epsilon}\psi \quad
  \text{ for all }\psi \in \dom \ccE_{\eta,\tau,\lambda;\epsilon}.
\end{split}
\end{equation}
Since $\bV_{\eta,\tau,\lambda;\epsilon}$ are bounded and symmetric, the operators $\ccE_{\eta,\tau,\lambda;\epsilon}$ are self-adjoint, by the
Kato-Rellich theorem. 

We can now state the main result of this subsection.
\begin{theorem}\label{thm:approximation}
	Let either $\eta,\tau,\lambda \in C^\infty(\Sigma;\R)$ be such that
	$d(x)\neq k^2\pi^2$, for all $k \in \N_0$ and
	for all $x\in\Sigma$, 
	or $\eta,\tau,\lambda \in \R$ be such that $d = (2k_0)^2 \pi^2$
        for some $k_0\in \N_0$. 
	Let $\hat{\eta,}\hat{\tau},\hat{\lambda}\in C^{\infty}(\Sigma;\R)$ be
        defined as follows:
	\begin{itemize}
		\item if $d>0$ 
		then
		\begin{equation}
		\label{eq:statement.going.d>0}
		(\hat{\eta},\hat{\tau},\hat{\lambda})
		= \dfrac{\tan(\sqrt{d}/2)}{\sqrt{d}/2} (\eta,\tau,\lambda);
		\end{equation}
		\item if $d = 0$ then
		\begin{equation} \label{eq:statement.going.d=0}
		(\hat{\eta},\hat{\tau},\hat{\lambda}) = (\eta,\tau,\lambda);
		\end{equation}
		\item if $d < 0$ then
		\begin{equation}
		\label{eq:statement.going.d<0}
		(\hat{\eta},\hat{\tau},\hat{\lambda})
		= \dfrac{\tanh(\sqrt{-d}/2)}{\sqrt{-d}/2} (\eta,\tau,\lambda).
		\end{equation}
	\end{itemize}
		Let $\ccE_{\eta,\tau,\lambda;\epsilon}$ be defined as in
		\eqref{eq:defn.Eepsilon} and
		$\ccD_{\hat{\eta},\hat{\tau},\hat{\lambda}}$ be defined as in
		\eqref{eq:defn.D.shell.domain}, \eqref{eq:defn.D.shell}.
		If
		$\mathfrak{C}(\hat{\eta},\hat{\tau},\hat{\lambda})(x)
		\neq 0$
		for all $x\in \Sigma$,	
		then
		\begin{equation*}
		\ccE_{\eta,\tau,\lambda;\epsilon} \xrightarrow[\epsilon \to 0]{}
	\ccD_{\hat{\eta,}\hat{\tau},\hat{\lambda}}
	\quad \text{ in the strong resolvent sense.}
	\end{equation*}
\end{theorem}
The proof of \Cref{thm:approximation} is in \Cref{sec:approximation}.
\begin{remark}
  In the case that $d=0$ the phenomenon of {renormalization of
    the coupling constants} does not occur. This was already observed
  in the one-dimensional setting in \cite{tusek19}.
\end{remark}
\begin{remark}
  Thanks to \Cref{thm:approximation}, for all $\eta
  \in \R\setminus\{(2k+1)\pi, (\frac{1}{2}+k)\pi \mid k \in \Z\} ,\tau \in \R$ we have
$\ccE_{\eta,0,0;\epsilon} \to \ccD_{\hat{\eta},0,0}$ and
$\ccE_{0,\tau,0;\epsilon} \to \ccD_{0, \hat{\tau},0}$ as $\epsilon \to
0$, with $\hat{\eta} = 2\tan(\eta/2)$ and $\hat{\tau} =2\tanh(\tau/2)$.
Exactly the same renormalization of the coupling constant appeared in the one-dimensional \cite{sebaklein} and  three-dimensional \cite{klein} cases. The renormalization of the coupling constant for a general one-dimensional relativistic point interaction was investigated in \cite{hughes99} and later in \cite{tusek19}, where exactly the same formulae for renormalization as in  \Cref{thm:approximation} were discovered.

\end{remark}

\begin{remark}\label{rem:unitary_eq}
  From \eqref{eq:statement.going.d>0}--\eqref{eq:statement.going.d<0},
  \begin{equation*}
    \hat{d}:=\hat{\eta}^2 - \hat{\tau}^2 - \hat{\lambda}^2
    =
    \begin{cases}
      4\tan^2(\frac{\sqrt{d}}{2}) \quad &\text{ if }d>0,\\
      0 \quad &\text{ if }d=0,\\
      -4 \tanh^2(\frac{\sqrt{-d}}{2}) \quad &\text{ if }d<0.
    \end{cases}
  \end{equation*}
  Since in all the cases $\hat{d} > -4$,  \Cref{thm:approximation} does not provide strong
  convergence to a Dirac operator with a $\delta$--shell causing
  confinement (see \Cref{thm:confinement.introduction}).
  However, we recover the case $\hat{d}=-4$ in the limit $d \to
  -\infty$.  This suggests that it should be possible to get the
  confining cases by means of an approximation procedure in which we
  choose the coefficients $\eta= \eta_\epsilon$, $\tau =
  \tau_\epsilon$ and $\lambda = \lambda_\epsilon$
  dependent on the parameter $\epsilon$ so that the associated parameter $\hat{d} = \hat{d}_\epsilon$
  satisfies $\hat{d}_\epsilon > -4$ and $\hat{d}_\epsilon \rightarrow
  -4$ uniformly
   in the limit $\epsilon\rightarrow 0$. 
  Finally, the fact that $\hat{d}>-4$ is not a
  limitation, since if $\hat{\eta},\hat{\tau},\hat{\lambda} \in C^\infty(\Sigma;\R)$ are such that $\hat{d}<-4$,
  $\ccD_{\hat{\eta},\hat{\tau},\hat{\lambda}}$ is unitary equivalent
  to $\ccD_{\tilde{\eta},\tilde{\tau},\tilde{\lambda}}$, for
  $\tilde{\eta},\tilde{\tau},\tilde{\lambda}\in C^\infty(\Sigma;\R)$ such that
  $\tilde{d} := \tilde{\eta}^2 - \tilde{\tau}^2 - \tilde{\lambda}^2 >
  -4$, see \Cref{sec:spectral.relations}. 
\end{remark}

As a consequence of \Cref{thm:approximation}, we get immediately the
second result of this section.
\begin{corollary}\label{thm:approximation.back}
     Let $\hat{\eta},\hat{\tau},\hat{\lambda} \in C^\infty(\Sigma;\R)$ be such that, everywhere on $\Sigma$, 
      $\hat{d} := \hat{\eta}^2 - \hat{\tau}^2 -\hat{\lambda}^2 > -4$
      and $\mathfrak{C}(\hat{\eta},\hat{\tau},\hat{\lambda}) \neq 0$.
    Let $\eta,\tau,\lambda \in C^\infty(\Sigma;\R)$ be defined as follows:
    \begin{itemize}
    \item if $\hat{d} > 0$, then
      \begin{equation*}
        (\eta,\tau,\lambda) = \frac{\arctan \sqrt{\hat{d}}/2 +
          k\pi}{\sqrt{\hat{d}}/2}
        (\hat{\eta},\hat{\tau},\hat{\lambda}),
        \quad \text{ for } k \in \Z;
      \end{equation*}
    \item if $\hat{d} = 0$ and
      $\hat{\eta},\hat{\tau},\hat{\lambda}$ are constant, then
      \begin{equation*}
        (\eta,\tau,\lambda) = (\hat{\eta},\hat{\tau},\hat{\lambda}).
      \end{equation*}
    	In particular, if $\hat{\eta} = \hat{\tau} = \hat{\lambda} = 0$, then
      $(\eta,\tau,\lambda) = (\hat{\eta},\hat{\tau},\hat{\lambda})$ 
      or $\eta,\tau,\lambda\in \R$  are such that $\eta^2 - \tau^2 - \lambda^2 = (2k \pi)^2$,
      for $k \in \N$; 
    \item if $-4 < \hat{d} < 0$, then
            \begin{equation*}
        (\eta,\tau,\lambda) = \frac{\arctanh \sqrt{-\hat{d}}/2}{\sqrt{-\hat{d}}/2}
        (\hat{\eta},\hat{\tau},\hat{\lambda}).
      \end{equation*}
    \end{itemize}
 Then
  \begin{equation*}
   \ccE_{\eta,\tau,\lambda;\epsilon} \xrightarrow[\epsilon \to 0]{}
   \ccD_{\hat{\eta,}\hat{\tau},\hat{\lambda}}
   \quad \text{ in the strong resolvent sense.}
 \end{equation*}
\end{corollary}
The proof of \Cref{thm:approximation.back} is
also given in \Cref{sec:approximation}.
\begin{remark}
  In the case that $\hat{d} \geq 0$ the correspondence
  $(\hat{\eta},\hat{\tau},\hat{\lambda})\mapsto(\eta,\tau,\lambda)$
  is not one-to-one. One can choose the coupling constants in the approximating potentials arbitrarily large and still ends up with the same limit operator. From the physical perspective, we suppose that this surprising behaviour is possible due to the \emph{Klein effect} (also called the Klein paradox). Usually, the Klein effect is related to the scattering on the electrostatic barrier when, speaking vaguely, the transmission coefficient does not depend on the height of the barrier monotonously, see \cite{DoCa_99} for an overview.  Clearly,
  this effect occurs for the pure electrostatic interaction, for which
  $\hat d>0$.
  On the other hand, one can push $\hat d$ below zero, and thus eliminate the Klein effect, by switching  sufficiently strong Lorentz scalar/magnetic fields on.
\end{remark}

We conclude the presentation of our results underlining that, in the case that $\hat\eta=\hat\tau=0$ and
$\hat\lambda\in\R\setminus\{\pm 2\}$,  it is possible to give a
simple direct proof of \Cref{thm:approximation.back}, constructing an alternative sequence of
approximations without making use of ``parallel coordinates'', see
\Cref{sec:alternative_approximation}. 

\section{Preliminaries}
\label{sec:preliminaries}

In order to prove our results we need to introduce a number of
mathematical objects and related results.
First, we discuss in Subsections~\ref{sec:tangent.normal} and~\ref{sec:tubular.neighborhoods} planar curves and their tubular neighborhoods. 
Then in Subsection~\ref{sec:Pauli} we provide some identities related to Pauli matrices.  Further, we give in Subsection~\ref{sec:pseudo} basic ideas on the Sobolev spaces and pseudo-differential operators on $\Sigma$.
Then we recall in Subsection~\ref{sec:trace} the concept of the trace operator.
After that we outline in Subsection~\ref{sec:triples} the approach of boundary triples to the extensions theory of symmetric operators. Finally, we recall some properties of the free Dirac operator in Subsection~\ref{sec:free.dirac} and define several associated auxiliary integral operators on $\Sigma$ in Subsection~\ref{sec:auxiliary}.
In this preliminary section we partially follow the presentation in \cite{behrndt2019two}, that gives the
theoretical background and the technical instruments for our analysis.
We refer to it and to the references therein for the proofs of the results in
this section and for further details.

\subsection{Tangent, normal and curvature of $\Sigma$}
\label{sec:tangent.normal}
We gather here some elementary facts on curves, in order to fix the
notations. Details can be  found, \eg in \cite{abate2012curves}.

We recall that $\Omega \subset \R^2$ is a bounded open simply connected set with $C^{\infty}$
 boundary $\Sigma := \partial\Omega$
.
Set $\ell := |\Sigma|$ and let $\gamma : \R \big/ \ell \Z \to \Sigma
\subset \R^2$
be a smooth arc-length parametrization of $\Sigma$ with positive
orientation.
Let
\begin{align}
  \label{eq:defn.tgamma}
  & \mathbf{t}_\gamma : \R/\ell\Z \to \R^2,
  &
  &\mathbf{t}_\gamma(s) 
    = \dot \gamma(s), \\
  \label{eq:defn.ngamma}
  & \mathbf{n}_\gamma : \R/\ell\Z \to \R^2,
  & 
  &\mathbf{n}_\gamma(s) 
    =   (\dot \gamma_2(s), -\dot \gamma_1(s)),
\end{align}
where the dot stands for the derivative with respect to the arc-length
$s$. Clearly, $\{\mathbf{n}_\gamma(s),
\mathbf{t}_\gamma(s)\}$ is a positively oriented basis of $\R^2$ for
any $s\in\R/\ell\Z$. Moreover, by the Frenet-Serret formulas, there exists function $\kappa_\gamma$, called the signed curvature, such that
\begin{equation} \label{eq:FS_formulas}
  \dot{\mathbf{t}}_\gamma=-\kappa_\gamma \mathbf{n}_\gamma,
  \quad \dot{\mathbf{n}}_\gamma=\kappa_\gamma \mathbf{t}_\gamma.
\end{equation}
Therefore, we have
\begin{equation}  \label{eq:defn.kappagamma}
|\kappa_\gamma(s)| = \norm{\dot{\mathbf{t}}_\gamma(s)}
    = \norm{\ddot \gamma(s)}.
\end{equation}
We set $\mathbf{t} := \mathbf{t}_\gamma \circ \gamma^{-1}, 
  \mathbf{n} := \mathbf{n}_\gamma \circ \gamma^{-1}: \Sigma \to \R^2$, and
  $\kappa := \kappa_\gamma\circ\gamma^{-1}:\Sigma \to \R$.
The functions $\mathbf{t}, \mathbf{n}, \kappa$ are independent of the
particular choice of the positively oriented arc-length parametrization $\gamma$: $\mathbf{n}$ is the unit normal vector field along $\Sigma$
which points outwards of $\Omega_+$, $\mathbf{t}$ is the unit tangent
vector field along $\Sigma$, counter-clockwise oriented,
 and $\kappa$ is the signed curvature of $\Sigma$.
We remark that we choose the definition of the curvature $\kappa$ so
that it is non-negative for convex domains $\Omega$.

\subsection{Tubular neighborhoods of $\Sigma$}
\label{sec:tubular.neighborhoods}
Below we recall some elementary properties of tubular neighborhoods of planar curves. Details can be found in \cite[Chapter
	1]{exner2015quantum},\cite{lee2006riemannian,approximation}, see also
	\cite[Sections 1.6 and 2.2]{abate2012curves}.
For $\beta >0$,
\begin{equation*}
\Sigma_\beta
:= \{x \in \R^2 \mid \mathrm{dist}(x,\Sigma) < \beta \}
\end{equation*}
is
the tubular neighborhood of $\Sigma$ of width $\beta$.
Let us introduce the following mapping
\begin{equation*}
  \mathscr{L}_\gamma:\,\R/\ell \Z \times (-\beta,\beta) \to \R^2,
  \quad
  \mathscr{L}_\gamma(s,p) =
  \gamma(s)+ p \mathbf{n}_\gamma(s).
\end{equation*}
The following theorem shows that,  for all $\beta$ small enough, the map $\mathscr{L}_\gamma$
is a smooth parametrization of $\Sigma_\beta$.
\begin{theorem}[{\cite[Theorem 2.2.5]{abate2012curves}}]
  \label{thm:existence.tubular.neighborhood}
  There exists 
  $\beta_0\in(0,(\max |\kappa_\gamma|)^{-1})$ such that, for all          $\beta\in(0,\beta_0)$, 
  $\mathscr{L}_\gamma
  $ is a bijection of $\R/\ell\Z \times
  (-\beta,\beta)$ onto $\Sigma_\beta$.
\end{theorem}
In the following we will always assume that $0<\beta <\beta_0$, for
$\beta_0>0$ given by \Cref{thm:existence.tubular.neighborhood}.
Thanks to the second formula in \eqref{eq:FS_formulas}, we get 
\begin{equation}\label{eq:derivatives.scrL}
  \nabla\mathscr{L}_\gamma(s,p)
  =
  \begin{pmatrix}
    \partial_s \mathscr{L}_\gamma(s,p)
    &
    \partial_p \mathscr{L}_\gamma(s,p)
  \end{pmatrix}
  =
  \begin{pmatrix}
    (1+ p \kappa_\gamma(s)) \mathbf{t}_\gamma(s)
    &
    \mathbf{n}_\gamma(s)
  \end{pmatrix},
\end{equation}
where $\partial_s \mathscr{L}_\gamma,\,\partial_p \mathscr{L}_\Gamma,\, \mathbf{t}_\gamma$, and $\mathbf{n}_\gamma$ should be understood as column vectors.
Thanks to \eqref{eq:defn.ngamma}, we obtain
\begin{equation}\label{eq:det.hatL}
  \det (\nabla \mathscr{L}_\gamma) (s,p)  
  =
  \det
  \begin{pmatrix}
    (1+ p \kappa_\gamma(s)) t_1(\gamma(s)) & +t_2(\gamma(s)) \\
    (1+ p \kappa_\gamma(s)) t_2(\gamma(s)) & -t_1(\gamma(s))
  \end{pmatrix}
  =
  -(1+ p \kappa_\gamma(s)).
\end{equation}

We remark that $\det (\nabla\mathscr{L}_\gamma) (s,p) < 0$ for all
$(s,p) \in \R/\ell\Z \times (-\beta,\beta)$, since
$\abs{p\kappa_\gamma(s)} < \beta \beta_0^{-1}< 1$.
Thus, we have
\begin{equation}\label{eq:change.of.variables}
  \int_{\Sigma_\beta} f(x) \, dx  =
  \int_{0}^{\ell} \int_{-\beta}^{\beta}
  f(\gamma(s) + p\mathbf{n}_\gamma(s))
  (1 + p\kappa_\gamma(s)) \,dp ds,
  \quad \text{for all }f\in L^1(\Sigma_\beta).
\end{equation}
Next, we define
\begin{align} 
  & \mathscr{P}_\gamma := (\mathscr{L}_{\gamma}^{-1})_1 :
  \Sigma_\beta \to \R/\ell \Z, 
  &\mathscr{P}_\gamma(\gamma(s) + p \mathbf{n}_\gamma(s)) = s, \label{eq:Pgamma}
  \\
  & \mathscr{P}_\perp :=(\mathscr{L}_{\gamma}^{-1})_2 :
  \Sigma_\beta \to (-\beta,\beta), 
  & \mathscr{P}_\perp(\gamma(s) + p \mathbf{n}_\gamma(s)) = p,\label{eq:Pperp}
\end{align}
Thanks to \eqref{eq:derivatives.scrL}, \eqref{eq:det.hatL}, and the inverse function theorem, we have that for all
$x = \gamma(s) + p \mathbf{n}_\gamma(s) \in \Sigma_\beta$ 
\begin{equation} \label{eq:L_grad}
    \nabla \mathscr{P}_\gamma (x) =
    \frac{1}{1+ p \kappa_\gamma(s)}\mathbf{t}_\gamma(s),
    \quad
    \nabla \mathscr{P}_\perp(x)
    = \mathbf{n}_\gamma(s).
\end{equation}

Finally, it is also convenient to define
\begin{equation*}
    \mathscr{L} : \Sigma \times (-\beta,\beta) \to \R^2,
    \quad \mathscr{L} (x_\Sigma, p) =
    \mathscr{L}_\gamma(\gamma^{-1}(x_\Sigma),p)
    =
    x_\Sigma +  p \mathbf{n}(x_\Sigma)
    \in \Sigma_\beta,
\end{equation*}
which is 
bijection of $\Sigma\times(-\beta,\beta)$ onto $\Sigma_\beta$, by
\Cref{thm:existence.tubular.neighborhood}; and
\begin{equation}\label{eq:PSigma}
      \mathscr{P}_\Sigma := \mathscr{L}_{1}^{-1} :
  \Sigma_\beta \to \Sigma, 
  \quad \mathscr{P}_\Sigma(x_\Sigma + p \mathbf{n}(x_\Sigma)) = x_\Sigma.
\end{equation}

\subsection{Pauli matrices}\label{sec:Pauli}
By an explicit calculation, one can verify that
\begin{align}
  \label{eq:Pauli.squares}
  \sigma_3^2 = (\sigma\cdot \mathbf{n})^2 = (\sigma\cdot \mathbf{t})^2 &=\bI_2,
  \\
  \label{eq:sigma.t}
  i  (\sigma \cdot \mathbf{n})\sigma_3
  &=
  \,
  \sigma \cdot \mathbf{t}, 
  \\
  \label{eq:sigma_n.t}
  (\sigma\cdot \mathbf{n})(\sigma\cdot \mathbf{t})&=i \sigma_3,
\end{align}
where in \eqref{eq:sigma.t} and \eqref{eq:sigma_n.t} we have used the fact that $(t_1,t_2) =
(-n_2,n_1)$.

\subsection{Sobolev spaces and pseudo-differential calculus on
  $\Sigma$}\label{sec:pseudo}
We denote {by} $\bT$ the torus $\bT := \mathbb{R}/\mathbb{Z}$;
the space of the periodic smooth
functions on the torus $\bT$
and the space of periodic distributions on the torus $\bT$ will be denoted by
$\mathcal{D}(\bT) = C^\infty(\bT)$  and $\mathcal{D}(\bT)'$, respectively.
For  $f \in \mathcal{D}(\bT)'$, we define its Fourier coefficients
using the duality pairing $\langle \cdot, \cdot
\rangle_{\mathcal{D}(\bT)',\mathcal{D}(\bT)}$ as follows:
\[
	\widehat{f}(n) := \langle f,
        e_{-n}\rangle_{\mathcal{D}(\bT)',\mathcal{D}(\bT)},\quad
        e_n : t\in \bT \mapsto e^{\rmi 2\pi n t}.
\]
For $s \in \R$, the Sobolev space of order $s$ on $\bT$ is defined as
\[
	H^s(\bT) := \Big\{f \in \cD(\bT)' \, \Big| \sum_{n =-\infty}^{+\infty}(1+|n|)^{2s}|\widehat{f}(n)|^2 < +\infty\Big\}.
      \]
 A linear operator $H$ on $C^\infty(\bT)$ is a periodic pseudo-differential operator on $\bT$ if there exists $h : \bT \times \Z \to \C$ such that:
\begin{enumerate}[label=$({\roman*})$]
	\item for all $n\in \Z$, $h(\cdot,n) \in C^\infty(\bT)$,
	\item $H$ acts as $Hf = \sum_{n\in\Z} h(\cdot,n) \widehat{f}(n) e_n$,
	\item there exists $\alpha \in \R$ such that for all $p,q\in\N_0$ there exists $c_{p,q}>0$ such that there holds
	\begin{equation} \label{eq:pseudo_cond}
		\Big|\Big(\frac{d^p}{dt^p} (\omega^q h)\Big)(t,n)\Big| \leq c_{p,q}(1+|n|)^{\alpha-q},
	\end{equation}
where the operator $\omega$ is defined for all $(t,n)\in\bT\times\Z$ by $(\omega h)(t,n) := h(t,n+1) - h(t,n)$.
\end{enumerate}
The number $\alpha$ is called the order of the pseudo-differential operator $H$. The set of all pseudo-differential operators of order $\alpha$ on $\bT$ is denoted $\Psi^\alpha$,  and we define
\[
	\Psi^{-\infty} := \bigcap_{\alpha\in\R} \Psi^{\alpha}.
\]

Recall that $\ell = |\Sigma|$ and that $\gamma : \R \big/ \ell \Z \to \Sigma$
is a smooth arc-length parametrization of $\Sigma$. We define the map $U^* : \cD(\bT) \to \cD(\Sigma)$ as
\[
	(U^*g)(x) := \ell^{-1}g(\ell^{-1}\gamma^{-1}(x)),\ x\in\Sigma,
\]
where we have set $\cD(\Sigma) := C^\infty(\Sigma)$, and the map $U : \cD(\Sigma)' \to \cD(\bT)'$ as
\begin{equation}\label{eqn:defU}
	\langle U f, g\rangle_{\cD(\bT)',\cD(\bT)} := \langle  f, U^*g\rangle_{\cD(\Sigma)',\cD(\Sigma)}.
\end{equation}
The Sobolev space of order $s\in\R$ on $\Sigma$ is defined as
\[
  H^s(\Sigma) := \{ f \in \cD(\Sigma)' \mid Uf \in H^s(\bT)\}.
\]
For all $s \geq 0$, $H^{-s}(\Sigma) = (H^s(\Sigma))'$ and 
 the duality pairing $\langle \phi, \psi
\rangle_{H^{-s},H^s}$ is defined for all $\phi \in
H^{-s}(\Sigma), \psi \in H^{s}(\Sigma)$.

A linear operator $H$ on $C^\infty(\Sigma)$ is a periodic pseudo-differential operator on $\Sigma$ of order $\alpha\in \R$ if the operator $H_0 := U H U^{-1} \in \Psi^\alpha$. The set of pseudo-differential operators on $\Sigma$ of order $\alpha$ is denoted $\Psi^\alpha_{\Sigma}$ and we set
\[
	\Psi_{\Sigma}^{-\infty} := \bigcap_{\alpha\in\R}\Psi_{\Sigma}^\alpha.
\]

In the next proposition we  gather some useful properties of the
pseudo-differential operators on $\Sigma$ (for proofs see \cite[Sections
5.8 \& 5.9]{saranen2013periodic}). 
\begin{proposition}\label{prop:pseudoimpo}
  Let $\alpha,\beta\in \R$ and $A\in \Psi_{\Sigma}^\alpha, B \in \Psi_{\Sigma}^\beta$.
  \begin{enumerate}[label=$({\roman*})$]
  \item\label{item:extension} For all $s\in\R$, $A$ extends uniquely to a bounded linear operator, denoted by the same letter, from $H^s(\Sigma)$ to $H^{s-\alpha}(\Sigma)$.
  \item\label{itm:22} \label{itm:pseudoimp1}There holds
    \[
      A + B \in \Psi_{\Sigma}^{\max(\alpha,\beta)},\quad AB \in \Psi_{\Sigma}^{\alpha+\beta},\quad [A,B] \in \Psi_{\Sigma}^{\alpha+\beta -1}. 
    \]
  \end{enumerate}
\end{proposition}

\subsubsection{The operator $\Lambda^{\alpha}$}
We describe an example of pseudo-differential operator that
is useful for our purposes.
For $\alpha \in \R$, consider the operator $\Lambda^\alpha$ on $C^\infty(\Sigma)$
\begin{equation}\label{eq:defn.Lambda}
  \Lambda^{\alpha}:=   U^{-1}
  L^{\alpha} U,
  \quad \text{ with }
  L^\alpha u(x) = \sum_{n\in \Z} (1 + \abs{n})^{\frac{\alpha}{2}}
  \widehat{u}(n) e_n(t),
  \quad u \in \cD(\bT).
\end{equation}
Thanks to \eqref{eq:pseudo_cond} and \eqref{eq:defn.Lambda}, one can show that $\Lambda^{\alpha} \in \Psi_\Sigma^{\frac{\alpha}{2}}$. 
Due to \Cref{prop:pseudoimpo} \ref{item:extension}, $\Lambda^\alpha$
extends uniquely to a bounded linear operator from $H^{s}(\Sigma)$ to
$H^{s-\frac{\alpha}{2}}(\Sigma)$, for any $s\in \R$; and  such extension is in fact an isomorphism, by the
definition of $H^s(\Sigma)$.
Of course, $\Lambda^\alpha$ can also be seen as an unbounded operator
on $H^s(\Sigma)$, for all $s\in \R$.

In particular, the operator $\Lambda:=\Lambda^1$ is used repeatedly in
the paper, and its action on vector valued functions is understood component-wise.
It is useful to remark that for all $\phi, \psi \in L^2(\Sigma)$ we
have $\Lambda \phi \in H^{-\frac12}(\Sigma), \Lambda^{-1} \psi \in H^{\frac12}(\Sigma),
$ and 
\begin{equation}\label{eq:Lambda.duality} 
  \langle \Lambda \phi, \Lambda^{-1} \psi \rangle_{H^{-1/2},H^{1/2}} = \langle
  \phi, \psi \rangle_{L^2}.
\end{equation}
\subsection{Trace operators}\label{sec:trace}
For any open set ${U} \subset \R^2$, recall the definition of
$H(\sigma,{U})$ in \eqref{eq:defn.HsigmaOmega}.
Thanks to \cite[Lemma 2.2]{benguria2017self}, $H(\sigma,{U}) \subset L^2({U};\C^2)\cap H^1_{loc}(U)$ and it is a
Hilbert space, endowed with the norm
\begin{equation*}
\norm{f}_{H(\sigma,{U})}^2 =
\norm{f}_{L^2({U};\C^2)}^2 +
\norm{-i\sigma\cdot \nabla f}_{L^2({U};\C^2)}^2.
\end{equation*}
We recall that $\Omega \subset \R^2$ is a bounded open simply connected set with $C^{\infty}$
boundary $\Sigma := \partial\Omega$, and 
$\R^2 = \Omega_+ \cup \Sigma \cup \Omega_-$, where
$\Omega_+ := \Omega$ and $\Omega_-:=\R^2 \setminus \overline{\Omega_+}$.
Thanks to \cite[Lemma 2.3, Lemma
2.4]{benguria2017self} (see also \cite[Lemma 15, Lemma
18]{antunes2020variational}
), the Dirichlet trace operators
\begin{equation*}
\cT_{\pm,0}^D : H^1(\Omega_\pm;\C^2) \to  H^{\frac12}(\Sigma;\C^2)
\end{equation*}
extend uniquely to the bounded linear operators
\begin{equation*}
\cT_\pm^D : H(\sigma,\Omega_\pm) \to H^{-\frac12}(\Sigma;\C^2).
\end{equation*}
and the following holds:
\begin{proposition}\label{prop:trace.H12}
For $f \in
H(\sigma,\Omega_\pm)$,
$\cT_\pm^D f \in H^{\frac12}(\Sigma;\C^2)$ if and only if $f
\in H^1(\Omega_\pm;\C^2)$. 
\end{proposition}

\subsection{Theory of the boundary triples}
\label{sec:triples}
In this section we review the theory of the boundary triples,
referring to \cite{behrndt2019two,bruning2008spectra,derkach1991generalized, posilicano2007self} and to the
monographs \cite{behrndtboundary,schmudgen2012unbounded} for details.

We start with the definition of a {boundary triple} for a symmetric operator.
\begin{definition}\label{defn:boundary.triple}
  Let $A$ be a closed densely defined symmetric operator in a Hilbert
  space $\mathcal{H}$.
  Moreover, let $\mathcal{G}$ be another Hilbert space and 
  $\Gamma_0,\Gamma_1: \dom A^* \to \mathcal{G}$ be linear maps.
  The triple $\{\mathcal{G},\Gamma_0,\Gamma_1\}$ is a
  \emph{boundary triple} for $A^*$ if and only if
  \begin{enumerate}[label=$({\roman*})$]
  \item\label{item:boundary.triple.1} for all $f, g \in \dom A^*$ there holds
    \begin{equation*}
      \langle A^* f, g \rangle_\mathcal{H}
      - \langle f, A^* g \rangle_\mathcal{H}
      = \langle \Gamma_1 f, \Gamma_0 g\rangle_\mathcal{G}
      - \langle\Gamma_0 f, \Gamma_1 g \rangle_\mathcal{G};
    \end{equation*}
  \item\label{item:boundary.triple.2} the map $ f \in \dom A^* \mapsto (\Gamma_0 f, \Gamma_1 f )\in \mathcal{G}\times\mathcal{G}$ is surjective. 
  \end{enumerate}
\end{definition}
Let $\{\mathcal{G},\Gamma_0,\Gamma_1\}$ be a boundary triple for the adjoint $A^*$ of a
densely defined closed symmetric operator $A$ on a Hilbert space
$\mathcal{H}$. Then $B:= A^*\upharpoonright\ker\Gamma_0$ is
self-adjoint, and for any $z \in \rho(B)$, one has the direct sum
decomposition
\begin{equation*}
  \dom A^* = \dom B \mathop{\dot{+}} \ker (A^* -z) =
  \ker \Gamma_0 \mathop{\dot{+}} \ker (A^* -z).
\end{equation*}
In particular,  $\Gamma_0\upharpoonright \ker(A^* -z)$
is bijective. We define the $\gamma$--field $G_z$ and the Weyl function
$M_z$ associated to the triple $\{\mathcal{G},\Gamma_0,\Gamma_1\}$:
\begin{align}
  \label{eq:abstract.gamma.function}
  &G_z : z \in \rho (B) \mapsto \big(\Gamma_0 \upharpoonright \ker
                    (A^*-z)\big)^{-1}:\mathcal{G} \to \mathcal{H},\\
  \label{eq:abstract.weyl.function}
  &M_z : z \in \rho (B) \mapsto 
                     \Gamma_1 \,G_z:\mathcal{G}\to \mathcal{G}.
\end{align}
For $z\in\rho (B)$ the operators $G_z$ and $M_z$ are bounded,
and $z\mapsto G_z$ and $z\mapsto M_z$ are holomorphic on $\rho (B)$.
{Furthermore}, the adjoints of $G_z$ and $M_z$ are given by
\begin{equation*} 
  G_z^* = \Gamma_1 (B - \overline z)^{-1} \quad \text{and} \quad M_z^* = M_{\bar{z}}.
\end{equation*}

  For $A$ a closed densely defined symmetric operator in a Hilbert
  space $\mathcal{H}$, the knowledge of a boundary triple for the operator $A^*$ allows to move
the study of its self-adjoint restrictions and their spectral properties to the
(sometimes) easier setting of the Hilbert space $\mathcal{G}$. This is
shown in the next proposition, for which we need to introduce some
notation. 
Let $\mathcal{G}_{\Pi}$ be a closed subspace of $\mathcal{G}$, viewed
as a Hilbert space when endowed with the induced inner product.
Denote the projection and the canonical embedding as
\begin{equation*}
  \Pi : \mathcal{G} \to \mathcal{G}_{\Pi} \quad \text{and} \quad
  \Pi^* :  \mathcal{G}_{\Pi} \to \mathcal{G},
\end{equation*}
respectively.
Let $\Theta$ be a linear operator in  $\mathcal{G}_{\Pi}$.
We define the operator $B_{\Pi,\Theta} := A^* \upharpoonright \dom
B_{\Pi,\Theta}$, where
\begin{equation}\label{eq:domain.project.abstract}
  \dom B_{\Pi,\Theta} 
  := \{f \in \dom A^* \mid
  (\bI - \Pi^* \Pi)\Gamma_0 f = 0, \,
  \Pi \Gamma_0 f \in \dom \Theta, \,
  \Pi \Gamma_1 f = \Theta \Pi \Gamma_0 f
\}.
\end{equation}
\begin{theorem}[{\cite[Theorem 2.12]{behrndt2019two}}]
  \label{thm:boundary.triple.abstract}
	The operator $B_{\Pi, \Theta}$ is (essentially) self-adjoint in $\mathcal{H}$
  if and only if $\Theta$ is  (essentially) self-adjoint in $\mathcal{G}$. 
  Furthermore, if $\Theta$ is self-adjoint and $z \in \rho (B)$, then
  the following assertions hold: 
  \begin{enumerate}[label=$({\roman*})$]
    \item\label{item:abstract.spectrum} $z\in \sigma (B_{\Pi,\Theta})$ if and only if $0 \in \sigma( \Theta-\Pi M_z\Pi^*)$;
    \item\label{item:abstract.point.spectrum} $z\in \sigma_\textup{p} (B_{\Pi,\Theta})$ if and only if $0 \in \sigma_\textup{p}( \Theta-\Pi M_z\Pi^*)$, and in that case
		 $\ker (B_{\Pi,\Theta}-z) = G_z \Pi^* \ker  ( \Theta-\Pi M_z\Pi^*)$;
    \item\label{item:abstract.Krein} for all $z \in \rho (B_{\Pi, \Theta}) \cap \rho (B)$ one has
    \begin{equation*} 
        (B_{\Pi,\Theta}-z)^{-1}=(B-z)^{-1}+G_z\Pi^*( \Theta-\Pi M_z\Pi^*)^{-1} \Pi G_{\Bar{z}}^*.
    \end{equation*}
  \end{enumerate}
\end{theorem}

\subsection{The free Dirac operator}\label{sec:free.dirac}
Recall that the free Dirac operator $\ccD_0$  is defined as follows: 
\begin{equation*}
      \ccD_0 f := D_0 f, \quad      \dom \ccD_0 := H^1(\R^2;\C^2).
 \end{equation*}
For any $z \in \rho(\ccD_0) =\C\setminus((-\infty,-|m|]\cup[|m|,+\infty))$ 
we have
\begin{equation*}
  (\ccD_0 - z)^{-1} f(x) =
  \int_{\R^2} \phi_z(x-y) f(y)\,dy,
  \quad f \in L^2(\R^2;\C^2),
\end{equation*}
where the Green function $\phi_z$ is given for $x \neq 0$ by 
\begin{equation} \label{def_Green_function}
  \phi_z(x)
  := \frac1{2\pi}K_0\big(\sqrt{m^2 -z^2}|x|\big)\big(m\sigma_3 + z\bI_2\big)
  +
  i\Big(\sigma\cdot \frac{x}{|x|}\Big)\frac{\sqrt{m^2 - z^2}}{2\pi} K_1\big(\sqrt{m^2 - z^2}|x|\big),
\end{equation}
the functions $K_j$ are the modified Bessel functions of the second kind of
order $j$ and we are taking the principal
square root function, \ie  for $z \in \C\setminus (-\infty,0]$, $\Re \sqrt{z} > 0$.

We denote $S$ the
restriction of $\ccD_0$ to the functions vanishing at $\Sigma$, \ie
\begin{equation}\label{eq:defn.S}
  S f = (-i\sigma\cdot \nabla + m\sigma_3) f,
  \quad \dom S = H_0^1 (\R^2 \setminus \Sigma;\C^2).
\end{equation}
It is easy to see that $S^*$
is the maximal realization of $D_0$ in
$\R^2 \setminus \Sigma$, i.e., for $f = f_+ \oplus f_- \in
L^2(\R^2;\C^2) \equiv L^2(\Omega_+;\C^2) \oplus L^2(\Omega_-;\C^2)$,
\begin{equation}\label{eq:defn.S*}
  \begin{split}
  &\dom S^* = \{f = f_+ \oplus f_- \in L^2(\Omega_+;\C^2) \oplus
  L^2(\Omega_-;\C^2) \mid f_\pm \in H(\sigma,\Omega_\pm)\},\\
  &S^* f = (-i\sigma\cdot \nabla + m\sigma_3) f_+ \oplus
  (-i\sigma\cdot \nabla + m\sigma_3) f_-.
\end{split}
\end{equation}
 We finally recall some properties of the essential spectrum of any
 self-adjoint extension of $S$.
 \begin{proposition}[{\cite[Propositions~3.8~and~3.9]{behrndt2019two}}]
   \label{prop:extensions.spectral.properties}
  Let $A$ be a self-adjoint extension of $S$. Then the following hold:
  \begin{enumerate}[label=$({\roman*})$]
  \item\label{item:extensions.spectral.properties.1} $(-\infty,-\abs{m}] \cup [\abs{m},+\infty)
    \subset \sigma_{ess}(A)$;
  \item\label{item:extensions.spectral.properties.2} if $\dom A \subset H^s(\R^2\setminus\Sigma;\C^2)$ for some
    $s>0$, then the spectrum of $A$ in $(-\abs{m},\abs{m})$ is purely
    discrete and finite.
  \end{enumerate}
\end{proposition}

\subsection{Auxiliary integral operators}
\label{sec:auxiliary}
We introduce now several integral operators related to the Green
function $\phi_z$.

Let us denote the Dirichlet trace operator in $H^1(\R^2; \C^2)$  on $\Sigma$
by $\mathcal{T}^D: H^1(\R^2; \C^2) \to H^{\frac12}(\Sigma; \C^2)$. 
  It is well known that $\mathcal{T}^D$ is bounded, surjective, and
  $\ker \mathcal{T}^D = H^1_0(\R^2 \setminus \Sigma; \C^2)$, see
   \cite[Theorems~3.37~and~3.40]{mclean2000strongly}.
For $z \in \rho(\ccD_0)$ we define
\begin{equation} \label{def_Phi_z_prime}
  \Phi_z' := \mathcal{T}^D (\ccD_0 - \Bar{z})^{-1}: L^2(\mathbb{R}^2; \mathbb{C}^2) \to H^{\frac{1}{2}}(\Sigma; \mathbb{C}^2)
\end{equation}
and its anti-dual
\begin{equation} \label{def_Phi_z}
  \Phi_z := \big( \mathcal{T}^D (\ccD_0 - \overline{z})^{-1} \big)': H^{- \frac{1}{2}}(\Sigma; \mathbb{C}^2) \rightarrow L^2(\mathbb{R}^2; \mathbb{C}^2).
\end{equation}
The potential operator $\Phi_z$ is a bounded bijective operator from
  $H^{-\frac{1}{2}}(\Sigma; \mathbb{C}^2)$ onto  $\ker (S^* - z)$. 
  Moreover, for $\varphi \in L^2(\Sigma; \mathbb{C}^2)$ one has the
  integral representation
  \begin{equation*}
    \Phi_z \varphi (x) = \int_\Sigma \phi_z(x-y) \varphi(y) \, ds(y) \quad \text{for a.e. }x \in \mathbb{R}^2\setminus\Sigma.
  \end{equation*}
  
  We denote  $C_\Sigma$ the
  \emph{Cauchy transform} on $\Sigma$. To define it, we identify $\mathbb{R}^2 \sim \C$: writing
  $\mathbb{R}^2 \ni x=(x_1, x_2) \sim x_1 + \rmi x_2 =: \xi \in
  \mathbb{C}$ and $\mathbb{R}^2 \ni y=(y_1, y_2) \sim y_1 + \rmi y_2
  =: \zeta \in \mathbb{C}$. Then we set
  \begin{equation} \label{def_Cauchy_transform}
    C_\Sigma u(\xi) := \frac{\rmi}{\pi}\, \text{p.v. } \int_\Sigma \frac{u(\zeta)}{\xi - \zeta} \text{d} \zeta, 
    \quad \text{ for all }u \in C^\infty(\Sigma),\, \xi \in \Sigma,
  \end{equation}
  where the complex line integral is understood as its principal
  value.
Furthermore, let $C_\Sigma'$ be the operator which satisfies $
(C_\Sigma u, v)_{L^2(\Sigma)} = (u, C_\Sigma'  v)_{L^2(\Sigma)}$ for
all $u, v \in C^\infty(\Sigma)$.
The periodic pseudodifferential operators $C_{\Sigma}, C_{\Sigma}'$ belong to $\Psi_\Sigma^0$ and give rise to bounded operators in $H^s(\Sigma)$, for all $s\in
\R$. Moreover, 
\begin{equation}\label{eq:psidiff_inf}
C_{\Sigma}' C_{\Sigma} - \bI,\,  C_{\Sigma} C_{\Sigma}' -
\bI,\, C_\Sigma-C_\Sigma' \in \Psi_{\Sigma}^{-\infty}
\end{equation}
(see \cite[Proposition 2.9]{behrndt2019two})
and $C_\Sigma^2 =C_\Sigma'^2 = \bI$ (see \cite[Lemma 4.2.3]{saranen2013periodic}).
  
For $z \in \rho(\ccD_0)$ we define the boundary integral operator
\begin{equation}\label{eq:defn.Cz}
  \mathcal{C}_{z} \, \varphi (x)
  :=
  \text{p.v.} \int_{\Sigma}
  \phi_z (x-y) \varphi(y)\, ds(y),
  \quad \text{ for a.e. }x \in \Sigma \text{ and for all }\varphi \in C^{\infty}(\Sigma;\C^2).
\end{equation}
The pseudodifferential operator $\mathcal{C}_z$ belongs to $\Psi_{\Sigma}^0$,
and, in particular, it gives rise to a bounded operator
in $H^s(\Sigma;\C^2)$, for any $s \in \R$. Its realization in
$L^2(\Sigma;\C^2)$ it satisfies $(\mathcal{C}_z)^* = \mathcal{C}_{\bar
  z}$.
Furthermore, for the tangent vector field $\mathbf{t} = ( {t}_1, {t}_2)$ along $\Sigma$ we denote
  \begin{equation}\label{eq:defn.T}
    T = {t}_1 + i {t}_2.
  \end{equation}
 Then one has
  \begin{equation} \label{Lambda_C_z_Lambda}
     \mathcal{C}_z =
    \dfrac{1}{2}\,
    \begin{pmatrix}
     0 &  C_\Sigma \overline{T} \\
      T C_\Sigma'  & 0
    \end{pmatrix}
    + \dfrac{\ell}{4\pi} 
    \begin{pmatrix} (z+m)\bI  & 0
      \\ 0 &  (z-m)\bI \end{pmatrix}
    \Lambda^{-2}
    +\widehat\Psi,
  \end{equation}
where $\ell$ is the length of $\Sigma$ and
$\widehat\Psi\in\Psi^{-2}_\Sigma$, see \cite[Proposition 3.4]{behrndt2019two}.

The operators $\Phi_z$ and $\mathcal{C}_z$ are related to each other
by the following relation, analogous in this context to
the Plemelj-Sokhotskii formula (see \cite[Proposition 3.5]{behrndt2019two}):
  \begin{equation*}
    \cT_\pm^D \Phi_z \varphi =
    \mp \frac{i}{2} (\sigma\cdot \mathbf{n}) \varphi
    +
    \mathcal{C}_z \varphi,
    \quad
    \text{ for all }
    \varphi \in H^{-\frac12}(\Sigma;\C^2).
  \end{equation*}

  \section{Unitary equivalences}
  \label{sec:gauge.w}
\subsection{Reduction to $\omega=0$}
Recall that the operator $\ccD_{\eta,\tau,\lambda,\omega}$
is defined as in~\eqref{eq:defn.D.shell.domain_w}, ~\eqref{eq:defn.D.shell_w}.
The purpose of this section is to show that, in many cases, 
$\ccD_{\eta,\tau,\lambda,\omega}$ is unitarily equivalent to 
$\ccD_{\tilde\eta,\tilde\tau,\tilde\lambda,0}$ for certain 
$\tilde\eta,\, \tilde\tau,\, \tilde\lambda: \Sigma \to\R$ defined in terms of
$\eta,\,\tau,\,\lambda$, and $\omega$. 
This unitary equivalence is based on the following simple transformation. Given $z\in\C$ such that $|z|=1$, let 
\begin{equation*}
U_z:L^2(\R^2;\C^2)\to L^2(\R^2;\C^2),\qquad U_z \varphi=\chi_{\Omega_+}\varphi+z\chi_{\Omega_-}\varphi,\quad\text{for all }\varphi\in L^2(\R^2;\C^2).
\end{equation*}
It is clear that $(U_z)^*=U_{\overline z}$ and, since $z\overline z=|z|^2=1$, that 
$(U_z)^*U_z=U_z(U_z)^*=\bI_2$. Hence $U_z$ is a unitary operator in $L^2(\R^2;\C^2)$. With this at hand, we can introduce the operator
\begin{equation*}
  \begin{split}
   &\dom(\ccD_{\eta,\tau,\lambda,\omega}^z) :=
   U_{\overline z} (\dom(\ccD_{\eta,\tau,\lambda,\omega})),\\
    &\ccD_{\eta,\tau,\lambda,\omega}^z f
    := U_{\overline z}\ccD_{\eta,\tau,\lambda,\omega} U_zf,
    \quad
    \text{for all }f \in \dom(\ccD_{\eta,\tau,\lambda,\omega}^z),
\end{split}
\end{equation*}
which is unitarily equivalent to $\ccD_{\eta,\tau,\lambda,\omega}$ by construction. 

Before addressing the proof of Theorem~\ref{thm:w.gauged},
let us make some observations on the values of $X$ and $z$,
which were introduced in \eqref{eq:solX} and \eqref{eq:def.z}, respectively, depending on $d$ and $\omega$.
\begin{itemize}
\item {\em Case $d=0$:} In this situation, \eqref{eq:solX} and \eqref{eq:def.z} rewrite as 
$$X=\frac{4}{4+\omega^2},\qquad z=\frac{4+\omega^2}{4-\omega^2+4\omega i}.$$
Therefore, we clearly have $X\in\R\setminus\{0\}$ and $z\in\C$ are constant. Also, to check that $|z|=1$ is straightforward.
\item {\em Case $d\neq0$ and $\omega=0$:} In this situation, \eqref{eq:solX} and \eqref{eq:def.z} rewrite as 
$$dX^2-4+(4-d)X=0,\qquad z=\frac{dX^2+4}{X(4+d)}.$$
Since the number of solutions of the first equation strongly depends on the values of $d$, we must distinguish two cases. On one hand, if $d\neq-4$ then we get 
$$X=\frac{1}{2d}\Big(d-4\pm\sqrt{(d-4)^2+16d}\Big)
=\frac{1}{2d}(d-4\pm|d+4|)$$
and, thus, the solutions to \eqref{eq:solX} are $X=1$ and $X=-4/d$. For $X=1$ we get $z=1$, and for $X=-4/d$ we get $z=-1$.
On the other hand, if $d=-4$ then \eqref{eq:solX} rewrites as $(X-1)^2=0$, hence the unique solution is $X=1$. In this case,  $z$ formally corresponds to 
$$z=\frac{dX^2+4}{X(4+d)}=\frac{d+4}{4+d}=1.$$
\item {\em Case $d\neq0$ and $\omega\neq0$:} From \eqref{eq:solX} we get
$$X=\frac{1}{2d}\Big(d-\omega^2-4\pm\sqrt{(d-\omega^2-4)^2+16d}\Big).$$
A simple computation shows that
$$(d-\omega^2-4)^2+16d
=(d-\omega^2+4)^2+16\omega^2\geq 16\omega^2>0
$$
and, therefore, $X$ can chosen as either one of the two real, nonzero,
and different solutions to \eqref{eq:solX}.
In this setting, it is clear that $z\in\C$ is constant because $X,\,\omega\neq0$.
Let us now show that $|z|=1$. We have
\begin{equation*}
\begin{split}
|z|^2&=\frac{(dX^2+4)^2}{X^2\big((4+d-\omega^2)^2+16\omega^2 \big)}
=\frac{(dX^2+4)^2}{X^2\big((-4+d-\omega^2)^2+16d \big)}\\
&=\frac{(dX^2+4)^2}{\big((d-\omega^2-4)X\big)^2+16dX^2}
=\frac{(dX^2+4)^2}{(dX^2-4)^2+16dX^2},
\end{split}
\end{equation*}
where we used \eqref{eq:solX} in the last equality above. Therefore,
\begin{equation*}
\begin{split}
|z|^2&=\frac{(dX^2+4)^2}{(dX^2-4)^2+16dX^2}
=\frac{(dX^2+4)^2}{(dX^2+4)^2}=1,
\end{split}
\end{equation*}
as desired.
\end{itemize}

We have checked that we always have $X\in\R\setminus\{0\}$, and $z\in\C$ always satisfies $|z|=1$. This shows the first statement of Theorem~\ref{thm:w.gauged}. The rest of the proof of Theorem~\ref{thm:w.gauged} strongly relies on the following result, which requires some notation. We set
\begin{equation}\label{def:M+-}
M^{\pm}_{\eta,\tau,\lambda,\omega}:= \pm i  (\sigma \cdot \mathbf{n}) + \frac{1}{2}
        \big( \eta \bI_2 + \tau \sigma_3 + \lambda (\sigma \cdot \mathbf{t})+ \omega(\sigma \cdot \mathbf{n})\big).
\end{equation}
Note that the boundary condition in \eqref{eq:defn.D.shell.domain_w} can be expressed as
$$M^{+}_{\eta,\tau,\lambda,\omega}\mathcal{T}_+^Df_+
=-M^{-}_{\eta,\tau,\lambda,\omega}\mathcal{T}_-^Df_-.$$

\begin{lemma}\label{lemma:key.gauge}
Given $\eta,\,\tau,\,\lambda,\,\omega\in\R$, let $X$ and $z$ be as in Theorem~\ref{thm:w.gauged}. If $(d,\omega)\neq(-4,0)$ then, for every $x\in\Sigma$,
$M^{\pm}_{\eta,\tau,\lambda,\omega}$ and 
$M^{\pm}_{X\eta,X\tau,X\lambda,0}$ are invertible matrices. Moreover,
\begin{equation}\label{eq:rel.initial.final.M}
(M^{+}_{\eta,\tau,\lambda,\omega})^{-1}
M^{-}_{\eta,\tau,\lambda,\omega}
(M^{-}_{X\eta,X\tau,X\lambda,0})^{-1}
M^{+}_{X\eta,X\tau,X\lambda,0}=\overline z\,\bI_2.
\end{equation}
\end{lemma}

\begin{proof}

We introduce the auxiliary matrix
\begin{equation*}
\widetilde M^{\pm}_{\eta,\tau,\lambda,\omega}:= \mp i  (\sigma \cdot \mathbf{n}) + \frac{1}{2}
        ( \eta \bI_2 - \tau \sigma_3 - \lambda (\sigma \cdot \mathbf{t})- \omega(\sigma \cdot \mathbf{n})).
\end{equation*}
Thanks to \eqref{eq:Pauli.squares}, \eqref{eq:sigma.t},
\begin{equation*}
\begin{split}
M^{\pm}_{\eta,\tau,\lambda,\omega}
\widetilde M^{\pm}_{\eta,\tau,\lambda,\omega}
&= \bigg(\!\pm i  (\sigma \cdot \mathbf{n}) + \frac{1}{2}
        ( \eta \bI_2 + \tau \sigma_3 + \lambda (\sigma \cdot \mathbf{t})+ \omega(\sigma \cdot \mathbf{n}))\bigg)\\
        &\hskip50pt
        \bigg(\!\mp i  (\sigma \cdot \mathbf{n}) + \frac{1}{2}
        ( \eta \bI_2 - \tau \sigma_3 - \lambda (\sigma \cdot \mathbf{t})- \omega(\sigma \cdot \mathbf{n}))\bigg)\\
        &=\bigg(1+\frac{1}{4}(\eta^2-\tau^2-\lambda^2-\omega^2)
        \mp\omega i\bigg)\bI_2
        =\frac{1}{4}(4+d-\omega^2\mp4\omega i)\bI_2.
\end{split}
\end{equation*}
Since $(d,\omega)\neq(-4,0)$, we see that 
$4+d-\omega^2\mp4\omega i\neq0$. Therefore, 
$M^{\pm}_{\eta,\tau,\lambda,\omega}$ is invertible whenever $(d,\omega)\neq(-4,0)$, and 
its inverse is given by
\begin{equation*}
\begin{split}
 (M^{\pm}_{\eta,\tau,\lambda,\omega})^{-1}
 &=\frac{4}{4+d-\omega^2\mp4\omega i}
 \,\widetilde M^{\pm}_{\eta,\tau,\lambda,\omega}\\
 &= \frac{4}{4+d-\omega^2\mp4\omega i}
 \bigg(\!\mp i  (\sigma \cdot \mathbf{n}) + \frac{1}{2}
        ( \eta \bI_2 - \tau \sigma_3 - \lambda (\sigma \cdot \mathbf{t})- \omega(\sigma \cdot \mathbf{n}))\bigg).
        \end{split}
\end{equation*}
A similar computation can be carried out to find
$(M^{\pm}_{X\eta,X\tau,X\lambda,0})^{-1}$. In this case, one gets
\begin{equation}\label{eq:def.invM.X}
\begin{split}
 (M^{\pm}_{X\eta,X\tau,X\lambda,0})^{-1}
 &= \frac{4}{4+dX^2}
 \bigg(\!\mp i  (\sigma \cdot \mathbf{n}) + \frac{X}{2}
        ( \eta \bI_2 - \tau \sigma_3 - \lambda (\sigma \cdot \mathbf{t}))\bigg).
        \end{split}
\end{equation}
Observe that, using \eqref{eq:solX}, that $(d,\omega)\neq(-4,0)$, and that $X\neq0$, we have
\begin{equation*}
\begin{split}
(dX^2+4)^2&=(dX^2-4)^2+16dX^2
=\big((d-\omega^2-4)X\big)^2+16dX^2\\
&=X^2\big((d-\omega^2-4)^2+16d \big)
=X^2\big((d-\omega^2+4)^2+16\omega^2 \big)>0.
\end{split}
\end{equation*}
Hence the right hand side of \eqref{eq:def.invM.X} is well defined. This shows that
$M^{\pm}_{X\eta,X\tau,X\lambda,0}$ is invertible whenever $(d,\omega)\neq(-4,0)$.

Let us address the proof of \eqref{eq:rel.initial.final.M}. The first step is to compute $(M_{\eta,\tau,\lambda,\omega}^{\pm})^{-1}M_{\eta,\tau,\lambda,\omega}^\mp$. We have
\begin{equation*}
\begin{split}
  (M^{\pm}_{\eta,\tau,\lambda,\omega}&)^{-1} 
 M_{\eta,\tau,\lambda,\omega}^\mp
 \\
 &= \frac{4}{4+d-\omega^2\mp4\omega i}
 \bigg(\!\mp i  (\sigma \cdot \mathbf{n}) + \frac{1}{2}
        ( \eta \bI_2 - \tau \sigma_3 - \lambda (\sigma \cdot \mathbf{t})- \omega(\sigma \cdot \mathbf{n}))\bigg)\\
        &\hskip105pt
        \bigg(\mp i  (\sigma \cdot \mathbf{n}) + \frac{1}{2}
        ( \eta \bI_2 + \tau \sigma_3 + \lambda (\sigma \cdot \mathbf{t})+ \omega(\sigma \cdot \mathbf{n}))\bigg)\\
        &=\frac{4}{4+d-\omega^2\mp4\omega i}
        \bigg(\frac{-4+d-\omega^2}{4}\bI_2\pm
        ( \lambda \sigma_3 - \eta i(\sigma \cdot \mathbf{n}) - \tau (\sigma \cdot \mathbf{t}))\bigg),
        \end{split}
\end{equation*}
and, similarly,
\begin{equation*}
\begin{split}
 (M^{\pm}_{X\eta,X\tau,X\lambda,0}&)^{-1}
 M_{X\eta,X\tau,X\lambda,0}^\mp
 \\
 & =\frac{4}{4+dX^2}
        \bigg(\frac{-4+dX^2}{4}\bI_2\pm
        X( \lambda \sigma_3 - \eta i(\sigma \cdot \mathbf{n}) - \tau (\sigma \cdot \mathbf{t}))\bigg).
        \end{split}
\end{equation*}
From these calculations, we obtain
\begin{equation}\label{eq:m+m-.1}
\begin{split}
(M^{+}_{\eta,\tau,\lambda,\omega}& )^{-1}
  M_{\eta,\tau,\lambda,\omega}^-
 (M^{-}_{X\eta,X\tau,X\lambda,0})^{-1}
 M_{X\eta,X\tau,X\lambda,0}^+\\
 &=\frac{4}{4+d-\omega^2-4\omega i}
        \bigg(\frac{-4+d-\omega^2}{4}\bI_2+
         \lambda \sigma_3 - \eta i(\sigma \cdot \mathbf{n}) - \tau (\sigma \cdot \mathbf{t})\bigg)\\
        &\hskip40pt \frac{4}{4+dX^2}
        \bigg(\frac{-4+dX^2}{4}\bI_2-
        X \lambda \sigma_3 + X \eta i(\sigma \cdot \mathbf{n}) + X\tau (\sigma \cdot \mathbf{t})\bigg).
\end{split}
\end{equation}
Given $a,\,\tilde a\in\R$, a computation shows that 
\begin{equation}\label{eq:m+m-.2}
\begin{split}
        \big(a\bI_2+
         \lambda \sigma_3 -  \eta i(\sigma \cdot \mathbf{n}) - \tau (\sigma \cdot \mathbf{t})\big)
        \big(\tilde a\bI_2-
        X \lambda \sigma_3 + X \eta i(\sigma \cdot \mathbf{n}) + X\tau (\sigma \cdot \mathbf{t})\big)\\
        =(a\tilde a +dX)\bI_2+(\tilde a-aX)\big(\lambda\sigma_3
        -\eta i(\sigma \cdot \mathbf{n})
        -\tau(\sigma \cdot \mathbf{t})\big).
\end{split}
\end{equation}
By taking $a=\frac{-4+d-\omega^2}{4}$ and $\tilde a=\frac{-4+dX^2}{4}$, and using \eqref{eq:solX}, we see that
\begin{equation}\label{eq:m+m-.2a}
\begin{split}
a\tilde a+dX&
=\frac{(-4+d-\omega^2)X(-4+dX^2)}{16X}+dX
=\frac{(dX^2-4)^2}{16X}+dX
=\frac{(dX^2+4)^2}{16X}
\end{split}
\end{equation}
and
\begin{equation}\label{eq:m+m-.2b}
\begin{split}
\tilde a-aX=\frac{1}{4}\big(dX^2-4+(4-d+\omega^2)X\big)=0.
\end{split}
\end{equation}
Plugging \eqref{eq:m+m-.2a} and \eqref{eq:m+m-.2b} in \eqref{eq:m+m-.2}, and combining then with 
\eqref{eq:m+m-.1}, we conclude that
\begin{equation*}
\begin{split}
(M^{+}_{\eta,\tau,\lambda,\omega})^{-1}
 M_{\eta,\tau,\lambda,\omega}^-
 &(M^{-}_{X\eta,X\tau,X\lambda,0})^{-1}
 M_{X\eta,X\tau,X\lambda,0}^+\\
 &=\frac{4}{4+d-\omega^2-4\omega i}\,\frac{4}{4+dX^2}\,
 \frac{(dX^2+4)^2}{16X}\,\bI_2\\
 &=\frac{dX^2+4}{X(4+d-\omega^2-4\omega i)}\,\bI_2=\overline z\bI_2,
\end{split}
\end{equation*}
where we used \eqref{eq:def.z} in the last equality above.
Therefore, \eqref{eq:rel.initial.final.M} holds and the lemma follows.
\end{proof}

\begin{proof}[Proof of Theorem~\ref{thm:w.gauged}]
We have already shown that $X\in\R\setminus\{0\}$ and $z\in\C$ satisfies $|z|=1$, see the comments above. It only remains to prove that 
\begin{equation}\label{eq:unit.equiv.X}
\ccD_{\eta,\tau,\lambda,\omega}^z
=\ccD_{X\eta,X\tau,X\lambda,0}.
\end{equation}
Once this is shown we would get that
$\ccD_{\eta,\tau,\lambda,\omega}$ and 
$\ccD_{X\eta,X\tau,X\lambda,0}$ are unitarily equivalent  through the unitary operator $U_z$, because 
$\ccD_{\eta,\tau,\lambda,\omega}^z 
 = U_{\overline z}\ccD_{\eta,\tau,\lambda,\omega} U_z$ by definition.

Note that \eqref{eq:unit.equiv.X} is obvious if $(d,\omega)=(-4,0)$
 because then $X=1$, $z=1$, and  $U_z=\bI_2$. From now on we assume that $(d,\omega)\neq(-4,0)$. Then, Lemma \ref{lemma:key.gauge} yields
\begin{equation}\label{eq:rel.initial.final.M.gauge}
(M^{-}_{\eta,\tau,\lambda,\omega})^{-1}
M^{+}_{\eta,\tau,\lambda,\omega}
=z(M^{-}_{X\eta,X\tau,X\lambda,0})^{-1}
M^{+}_{X\eta,X\tau,X\lambda,0}.
\end{equation}
Recalling now \eqref{eq:defn.D.shell.domain_w} and \eqref{def:M+-}, we have
\begin{equation*}
  \begin{split}
    \dom(\ccD_{\eta,\tau,\lambda,\omega}) &= 
    \big\{  f = f_+ \oplus f_- \in H(\sigma,\Omega_+) \oplus
    H(\sigma,\Omega_-)\mid \\
    &\hskip50ptM^{+}_{\eta,\tau,\lambda,\omega}\mathcal{T}_+^Df_+
 =-M^{-}_{\eta,\tau,\lambda,\omega}\mathcal{T}_-^Df_-
    \big\}.
\end{split}
\end{equation*}
Therefore, using \eqref{eq:rel.initial.final.M.gauge} and that 
$z\overline z=|z|^2=1$,
we deduce that
\begin{equation*}
  \begin{split}
   \dom(\ccD_{\eta,\tau,\lambda,\omega}^z) &=
   U_{\overline z} (\dom(\ccD_{\eta,\tau,\lambda,\omega}))\\
   &=\big\{U_{\overline z}f = f_+ \oplus \overline z f_- \in H(\sigma,\Omega_+) \oplus
    H(\sigma,\Omega_-)\mid \\
    &\hskip50pt M^{+}_{\eta,\tau,\lambda,\omega}
    \mathcal{T}_+^Df_+
    =-M^{-}_{\eta,\tau,\lambda,\omega}\mathcal{T}_-^Df_-
    \big\}\\
    &=\big\{U_{\overline z}f = f_+ \oplus \overline z f_- \in H(\sigma,\Omega_+) \oplus
    H(\sigma,\Omega_-)\mid \\
    &\hskip50pt 
    (M^{-}_{\eta,\tau,\lambda,\omega})^{-1}M^{+}_{\eta,\tau,\lambda,\omega}
    \mathcal{T}_+^Df_+
    =-\mathcal{T}_-^Df_-
    \big\}\\
    &=\big\{U_{\overline z}f = f_+ \oplus \overline z f_- \in H(\sigma,\Omega_+) \oplus
    H(\sigma,\Omega_-)\mid \\
    &\hskip50pt 
    z(M^{-}_{X\eta,X\tau,X\lambda,0})^{-1}
M^{+}_{X\eta,X\tau,X\lambda,0}
    \mathcal{T}_+^Df_+
    =-\mathcal{T}_-^Df_-
    \big\}\\
    &=\big\{U_{\overline z}f = f_+ \oplus \overline z f_- \in H(\sigma,\Omega_+) \oplus
    H(\sigma,\Omega_-)\mid \\
    &\hskip50pt 
M^{+}_{X\eta,X\tau,X\lambda,0}
    \mathcal{T}_+^Df_+
    =- M^{-}_{X\eta,X\tau,X\lambda,0}\mathcal{T}_-^D(\overline zf_-)\big\}\\
    &=\big\{g = g_+ \oplus g_- \in H(\sigma,\Omega_+) \oplus
    H(\sigma,\Omega_-)\mid \\
    &\hskip50pt 
M^{+}_{X\eta,X\tau,X\lambda,0}
    \mathcal{T}_+^Dg_+
    =- M^{-}_{X\eta,X\tau,X\lambda,0}\mathcal{T}_-^Dg_-\big\}\\
    &=\dom(\ccD_{X\eta,X\tau,X\lambda,0}).
\end{split}
\end{equation*}

Now, let $f\in \dom(\ccD_{\eta,\tau,\lambda,\omega}^z)
=\dom(\ccD_{X\eta,X\tau,X\lambda,0})$. Then
$f=U_{\overline z}\varphi$ for some $\varphi\in \dom(\ccD_{\eta,\tau,\lambda,\omega})$. Recall that, although $\eta,\,\tau$, and $\lambda$ may be non-constant, we assume that $d$ and $\omega$ are constant along 
$\Sigma$, which implies that $z$ is also constant in $\overline{\Omega_-}$. This yields $\overline zD_0\varphi_-=D_0(\overline z\varphi_-)$ in $\Omega_-$. Hence,
\begin{equation}\label{eq:Dz.DX}
\begin{split}
\ccD_{\eta,\tau,\lambda,\omega}^z f
&=U_{\overline z}\ccD_{\eta,\tau,\lambda,\omega} U_zf
=U_{\overline z}\ccD_{\eta,\tau,\lambda,\omega} U_zU_{\overline z}\varphi
=U_{\overline z}\ccD_{\eta,\tau,\lambda,\omega} \varphi\\
&=U_{\overline z}(D_0 \varphi_+ \, \oplus \,D_0 \varphi_-)
=D_0 \varphi_+ \, \oplus \,\overline zD_0 \varphi_-
= D_0 \varphi_+\, \oplus \,D_0 (\overline z\varphi_-)\\
&=D_0 (U_{\overline z}\varphi)_+ \, \oplus \,D_0 (U_{\overline z}\varphi)_-
=D_0 f_+ \, \oplus \,D_0 f_-
=\ccD_{X\eta,X\tau,X\lambda,0}f.
\end{split}
\end{equation}
That is, $\ccD_{\eta,\tau,\lambda,\omega}^z f=\ccD_{X\eta,X\tau,X\lambda,0}f$ for all $f\in \dom(\ccD_{\eta,\tau,\lambda,\omega}^z)=\dom(\ccD_{X\eta,X\tau,X\lambda,0})$. Therefore, \eqref{eq:unit.equiv.X} holds and Theorem \ref{thm:w.gauged} follows.
\end{proof}

\subsection{Spectral relations}\label{sec:spectral.relations}
From the proof of Theorem~\ref{thm:w.gauged} we realize that, if
$\omega=0$, we can also allow $d$ to be variable in $\Sigma$ in the
conclusion of Theorem~\ref{thm:w.gauged}, as long as $d(x)\notin\{0,-4\}$
for all $x\in \Sigma$. This is because \eqref{eq:Dz.DX} holds whenever $z$ is constant in 
$\Omega_-$, and for $\omega=0$ and $d\neq0,-4$ we can always take $z=-1$, as we explained below the statement of Theorem~\ref{thm:w.gauged}. 
Thus, we can take $X=-4/d$, which yields the isospectral transformation of parameters
\begin{equation*}
(\eta,\tau,\lambda)\mapsto  
(X\eta,X\tau,X\lambda)=-\frac{4}{d}(\eta,\tau,\lambda)
\end{equation*}
for all $\eta,\,\tau,\,\lambda\in C^\infty(\Sigma;\R)$ such that
$\eta^2(x)-\tau^2(x)-\lambda^2(x)=d(x)\notin\{0,-4\}$ for all $x\in\Sigma$, with
no more restrictions on $d$ in $\Sigma$.
We underline that this correspondence maps the set $\{(\eta,\tau,\lambda) \in \R^3 \mid
\eta^2 - \tau^2 -\lambda^2 < -4\}$ onto
$\{(\eta,\tau,\lambda) \in \R^3 \mid
-4 < \eta^2 - \tau^2 -\lambda^2 < 0\}$ and $\{(\eta,\tau,\lambda) \in \R^3 \mid
\eta^2 - \tau^2 -\lambda^2 > 0\}$ onto itself.

Next, we apply the observation above on $\ccD_{\eta,\tau,\lambda,0}= \ccD_{\eta,\tau,\lambda}$ with constant parameters. Moreover, we
investigate the spectral relation between $\ccD_{\eta,\tau,\lambda}$ and its charge conjugation.
\begin{proposition}
  Let $\eta,\,\tau,\,\lambda \in \R$, and $\ccD_{\eta,\tau,\lambda}$ be
  defined as in
  \eqref{eq:defn.D.shell.domain}, \eqref{eq:defn.D.shell}. The following hold:
    \begin{enumerate}[label=$({\roman*})$]
  \item\label{item:isospectral.spectrum} if $d\neq 0$, then $z \in \sigma_p(\ccD_{\eta,\tau,\lambda})$ if
    and  only if $z \in \sigma_p(\ccD_{-4\eta/d,-4\tau/d,-4\lambda/d})$.
  \item\label{item:charge.conjugation.spectrum} $z \in
    \sigma_p(\ccD_{\eta,\tau,\lambda})$ if and only if $-z \in
    \sigma_p(\ccD_{-\eta,\tau,-\lambda})$.
  \end{enumerate}
\end{proposition}
\begin{proof}
  \ref{item:isospectral.spectrum}
  The case $d=-4$ is obvious. The case $d\neq-4$ follows from Theorem~\ref{thm:w.gauged} with $\omega=0$ simply noting that one can take $X=-4/d$ and $z=-1$.

  \ref{item:charge.conjugation.spectrum}.
  Let $C$ be the antilinear charge conjugation operator
  \begin{equation*}
    C: L^2(\R^2;\C^2) \to L^2(\R^2;\C^2), \quad C f = \sigma_1 \overline{f}, \quad f \in L^2(\R^2;\C^2).
  \end{equation*}
The operator $C$ is an involution, \ie $C^2 f = f$ for all $f \in
L^2(\R^2;\C^2)$. The result follows if we show that 
\begin{equation}\label{eq:charge.conjugation.1}
  C D_{\eta,\tau,\lambda} = - D_{-\eta,\tau,-\lambda} C.
\end{equation}
  Taking the complex conjugate of the condition in the definition of $\dom\ccD_{\eta,\tau,\lambda}$, we see that $f \in
  \dom D_{\eta,\tau,\lambda}$ if and only if
  \begin{equation*}
    -i (\overline{\sigma}\cdot \mathbf{n})(\mathcal{T}_-^D \overline{f_-} -\mathcal{T}_+^D  \overline{f_+}) 
     =
    \tfrac12 (\eta \bI_2 + \tau \sigma_3 + \lambda (\overline{\sigma}\cdot
    \mathbf{t}))
    (\mathcal{T}_-^D \overline{f_-} + \mathcal{T}_+^D
    \overline{f_+}),
  \end{equation*}
  where we denoted
  $\overline\sigma:=(\overline{\sigma_1},\overline{\sigma_2})$ and
  $\overline{\sigma_{j}}$ is the matrix that has the
  conjugate entries of the matrix $\sigma_j$, $j=1,2$. Since
  $\overline{\sigma} = (\sigma_1,-\sigma_2)$, multiplying the last
  equation by $\sigma_1$ we get
  \begin{equation*}
    i ({\sigma}\cdot \mathbf{n})(\mathcal{T}_-^D
    (\sigma_1 \overline{f_-})
    -\mathcal{T}_+^D (\sigma_1 \overline{f_+})) 
     =
    \tfrac12 (-\eta \bI_2 + \tau \sigma_3 - \lambda ({\sigma}\cdot
    \mathbf{t}))
    (\mathcal{T}_-^D (\sigma_1 \overline{f_-}) + \mathcal{T}_+^D
    (\sigma_1 \overline{f_+})),
  \end{equation*}
  \ie $Cf \in \dom D_{-\eta,\tau,-\lambda}$. We have  showed
  that $\dom (C D_{\eta,\tau,\lambda}) = \dom (D_{-\eta,\tau,-\lambda} C)$.
  With an explicit computation one sees that $ (-i \sigma\cdot \nabla
  +m \sigma_3 ) C f = - C (-i\sigma\cdot\nabla +m\sigma_3) f$. Thus, we get
  \eqref{eq:charge.conjugation.1}.
\end{proof}

We finally mention that the three-dimensional analogue of 
$\ccD_{\eta,0,0,\omega}$ was investigated in \cite{mas2017dirac}, where the same transformation of the coefficients $(\eta,\omega)\to(\tilde \eta,0)$ by means of $X$ and $z$ was discovered. Since here we also admit $\tau,\,\lambda\neq0$, and, with a restriction, also non-constant coefficients,
Theorem~\ref{thm:w.gauged} can be understood as a generalization of \cite{mas2017dirac} to the two-dimensional scenario for more general $\delta$-shell interactions.

\section{Confinement}
\label{sec:confinement}
In the following lemma we describe the properties of confinement and
 transmission induced by the boundary condition in \eqref{eq:defn.D.shell.domain}.
\begin{lemma}\label{lem:confinement}
  Let $\eta,\tau,\lambda \in C^\infty(\Sigma;\R)$ and let $\ccD_{\eta,\tau,\lambda}$ be
  defined as in \eqref{eq:defn.D.shell.domain}, \eqref{eq:defn.D.shell}.
  Then the following hold:
  \begin{enumerate}[label=$(\roman*)$]
      \item\label{item:transmission} if ${d} \neq -4$ everywhere on $\Sigma$, there exists an
    invertible matrix function $R_{\eta,\tau,\lambda}$ (explicitly defined in \eqref{def_R_eta} below)
    such that every
    $f = f_+ \oplus f_- \in H(\sigma,\Omega_+) \oplus H(\sigma,\Omega_-)$
    belongs to $\dom \ccD_{\eta,\tau,\lambda}$
    if and only if
    \begin{equation} \label{eq:BC}
      \mathcal{T}_+^D f_+ = R_{\eta,\tau,\lambda} \mathcal{T}_-^D f_-;
    \end{equation}
  \item\label{item:confinement} if $d = -4$ everywhere on $\Sigma$, every $f = f_+ \oplus f_- \in
    H(\sigma,\Omega_+) \oplus H(\sigma,\Omega_-)$ belongs to $\dom \ccD_{\eta,\tau,\lambda}$
    if and only if
    \begin{equation}\label{eq:confinement.BC}
      \left[i (\sigma \cdot \mathbf{n}) \pm
        \tfrac12 (\eta \bI_2 + \tau \sigma_3 + \lambda \, (\sigma \cdot \mathbf{t})) \right] 
      \mathcal{T}_\pm^D f_\pm = 0.
    \end{equation}
  \end{enumerate}
\end{lemma}
\begin{proof}
Let  $f = f_+ \oplus f_- \in H(\sigma,\Omega_+) \oplus H(\sigma,\Omega_-)$.
From \eqref{eq:defn.D.shell.domain}, $f \in \dom\ccD_{\eta,\tau,\lambda}$ if and only if 
\begin{equation}\label{eq:transmission1}
  \begin{split}
      &\left( i  (\sigma \cdot \mathbf{n}) + \frac{1}{2}
        ( \eta \bI_2 + \tau \sigma_3 +  \lambda \, (\sigma \cdot \mathbf{t})) \right) \mathcal{T}_+^D f_+
     \\ &  = \left( i  (\sigma \cdot \mathbf{n}) - \frac{1}{2}
        ( \eta \bI_2 + \tau \sigma_3 + \lambda \, (\sigma \cdot \mathbf{t})) \right)
      \mathcal{T}_-^D f_-.
    \end{split}
  \end{equation}
Thanks to \eqref{eq:Pauli.squares}, this is equivalent to
  \begin{equation} \label{equation_transmission_condition1}
      \left( \bI_2 +  M \right) \mathcal{T}_+^D f_+
= \left( \bI_2 - M \right)
      \mathcal{T}_-^D f_-,
    \end{equation}
    with
    \begin{equation*}
      M := -\frac{i}{2}  (\sigma \cdot \mathbf{n})
        ( \eta \bI_2 + \tau \sigma_3 + \lambda \, (\sigma \cdot \mathbf{t})).
    \end{equation*}
    Due to \eqref{eq:Pauli.squares}, \eqref{eq:sigma.t}, and \eqref{eq:sigma_n.t}, we have
    \begin{gather}
      \notag            M = 
        -\frac{\eta}{2}i(\sigma\cdot\mathbf{n}) - \frac{\tau}{2}
        \sigma\cdot\mathbf{t} + \frac{\lambda}{2}\sigma_3,
        \\  \label{eq:observation}
         M^{2} = - \frac{d}{4} \bI_2,
        \quad
          (\bI_2 + M)(\bI_2 - M) = \frac{4 + d}{4}\bI_2.
    \end{gather}
  When ${d} \neq -4$
  the matrix $\bI_2 + M$ 
 is invertible, by \eqref{eq:observation}, and 
 we get \ref{item:transmission} by setting
\begin{equation} \label{def_R_eta}
  \begin{split}
    R_{\eta,\tau,\lambda} := & (\bI_2 + M)^{-1} (\bI_2 - M)
    = \frac{4}{4 + d}(\bI_2 - M)^2
    = \frac{4}{4 + d}\left(\frac{4-d}{4} \bI_2  - 2M\right)
    \\
     = & \frac{4}{4+d}\left(\frac{4-d}{4}\,\bI_2+i\eta(\sigma\cdot\mathbf{n})+\tau
       (\sigma\cdot\mathbf{t}) -\lambda\sigma_3 \right).
   \end{split}
 \end{equation}
 
 If $d = -4$ then
 $M^2 = \bI_2$. Multiplying
 \eqref{equation_transmission_condition1}
 by $\bI_2 \pm M$ we get
 \begin{equation*}
   0 = (\bI_2 \pm M)^2  \mathcal{T}_\pm^D f_\pm
   = 2 (\bI_2 \pm M) \mathcal{T}_\pm^D f_\pm.
 \end{equation*}
 Multiplying the previous equation by $\tfrac{i}{2}(\sigma\cdot\mathbf{n})$ and using \eqref{eq:Pauli.squares} we arrive at
\eqref{eq:confinement.BC}. 
Viceversa, \eqref{eq:confinement.BC} implies
 \eqref{eq:transmission1}  trivially.
\end{proof}

If ${d} = \eta^2 -\tau^2 - \lambda^2
\neq -4$ everywhere on $\Sigma$, \Cref{lem:confinement} \ref{item:transmission}
states that the values of $f_+$ and $f_-$ along $\Sigma$ are related via the matrix
$R_{\eta,\tau,\lambda}$: the presence of the $\delta$-shell implies a transmission
condition for the functions in the domain of $\ccD_{\eta,\tau,\lambda}$
across the surface $\Sigma$.

If $d = \eta^2 -\tau^2 -\lambda^2 = -4$ everywhere on $\Sigma$, \Cref{lem:confinement} \ref{item:confinement} implies \Cref{thm:confinement.introduction}.

\section{The non-critical case}
\label{sec:core}

In this section we prove Theorem~\ref{thm:D.shell.introduction} stated in \Cref{sec:Main.results}. The operator
$\ccD_{\eta,\tau,\lambda}$ is defined again as in \eqref{eq:defn.D.shell.domain},
\eqref{eq:defn.D.shell}, where in general the condition in
\eqref{eq:defn.D.shell.domain} is understood in the sense of $H^{-\frac12}(\Sigma)$. We show the self-adjointness and some further
properties of the operator
$\ccD_{\eta,\tau,\lambda}$: the strategy of the
proofs mainly follows \cite{behrndt2019two}, but we need to modify the boundary
triple that we use in order to include the magnetic interaction.

Recall that we assume that 
$\eta,\tau,\lambda \in C^\infty(\Sigma;\R)$ and $\Sigma=\partial\Omega
\in C^\infty$ in order to exploit
the theory of pseudo-differential operators,
but we expect that these
assumptions can be greatly weakened. However, in the case that
$\Sigma$ has only Lipschitz regularity  and in the case that
$\eta,\tau,\lambda$ are not regular
different phenomena are expected, see
\cite{le2018self,pizzichillo2019self,cassano2019self},
see also \cite{behrndt2019self, rabinovich2020boundary} where
coefficients with lower
regularity are considered   in the three
dimensional case.

It will be convenient to introduce some extra notation. Recall the
definition $T :=  \mathbf{t}_1 + i \mathbf{t}_2 \in \C$ from
\eqref{eq:defn.T}, where $\mathbf{t} = ({t}_1, {t}_2)$ is  the
tangent vector to $\Sigma$ at the point $x \in \Sigma$.  We define the following matrix-valued functions on $\Sigma$:
\begin{equation}\label{eq:defn.V.and.B}
  V:=
  \begin{pmatrix}
    1 & 0 \\
    0 & \overline{T}
  \end{pmatrix},
  \quad
     B:=
   \begin{pmatrix}
     \eta + \tau & \lambda \\
     \lambda  & \eta - \tau
   \end{pmatrix}.
\end{equation}
For all $x\in \Sigma$, the matrix $V(x)$ is  unitary and
\begin{equation}\label{eq:propr.V}
\eta\bI_2 + \tau  \sigma_3 + \lambda (\sigma\cdot \mathbf{t}) =
   V^* B V \quad \text{ in }\Sigma.
 \end{equation}
Finally, we have  $\det B = \eta^2 -\tau^2 -\lambda^2 = d \in C^{\infty}(\Sigma;\R)$,  cf.~\eqref{eq:defn.d}.

In the following proposition we adapt to our setting the boundary triple constructed in 
\cite[Proposition 3.6]{behrndt2019two} for the operator
$S^*$, defined in \eqref{eq:defn.S*}.
In the formulation of the below proposition we extend the operators
$\Lambda$ and $\mathcal{C}_\zeta$ defined in~\eqref{eq:defn.Lambda}
and \eqref{eq:defn.Cz} respectively onto two-component functions applying the respective mappings component-wise.

\begin{proposition}\label{prop:boundary.triples.for.Dirac}
  Let $\zeta\in\rho (\ccD_0)$ and $\Gamma_0, \Gamma_1: \dom S^* \to L^2(\Sigma;
\C^2)$ be defined by
\begin{equation} \label{eq:defn.Gamma}
  \begin{split}
    \Gamma_0 f&= i \Lambda^{-1} V (\sigma\cdot \mathbf{n})\big(\cT_+^D f_+ -\cT_-^D f_-),\\
    \Gamma_1 f&=\dfrac{1}{2} \Lambda V \Big((\cT_+^D f_+ + \cT_-^D
    f_-) -(\mathcal{C}_\zeta+\mathcal{C}_{\Bar\zeta}) V^* \Lambda  \Gamma_0 f \Big), 
  \end{split}
\end{equation}
where  $f = f_+ \oplus f_- \in \dom S^*$.
Then $\{ L^2(\Sigma; \C^2), \Gamma_0, \Gamma_1 \}$ is a boundary triple for $S^*$ such that $\ccD_0=S^*\upharpoonright\ker\Gamma_0$. Moreover, the 
corresponding $\gamma$--field is 
  \begin{equation}\label{eq:defn.Gz}
    G_z : z \in \rho (\ccD_0) \mapsto  \Phi_z V^* \Lambda
  \end{equation}
  and the Weyl function is
  \begin{equation}\label{eq:defn.Mz}
      M_z : z \in \rho(\ccD_0) \mapsto \Lambda V\Big(\mathcal{C}_z -\dfrac{1}{2} \big(\mathcal{C}_\zeta+\mathcal{C}_{\Bar\zeta} \big) \Big)V^*\Lambda.
  \end{equation}
\end{proposition}
\begin{proof}
  In \cite[Proposition~3.6]{behrndt2019two} it is proved that $\{
  L^2(\Sigma; \C^2), \widetilde \Gamma_0, \widetilde \Gamma_1 \}$ is a
  boundary triple for $S^*$, with
\begin{equation*} 
  \begin{split}
    \widetilde \Gamma_0 f&= i \Lambda^{-1}  (\sigma\cdot \mathbf{n})\big(\cT_+^D f_+ -\cT_-^D f_-),\\
    \widetilde \Gamma_1 f&=\dfrac{1}{2} \Lambda  \Big((\cT_+^D f_+ +
    \cT_-^D f_-) -(\mathcal{C}_\zeta+\mathcal{C}_{\Bar\zeta}) \Lambda
    \widetilde \Gamma_0 f \Big), 
    \quad\text{where } f = f_+ \oplus f_- \in \dom S^*.
  \end{split}
\end{equation*}
Moreover, $\ccD_0=S^*\upharpoonright\ker\widetilde\Gamma_0$ and the $\gamma$--field $\widetilde G_z$ and the Weyl function
$\widetilde M_z$ associated to the boundary triple 
$\{ L^2(\Sigma; \C^2), \widetilde \Gamma_0, \widetilde \Gamma_1 \}$
are defined by
\begin{equation}
  \begin{split}
    &\widetilde{G}_z : z \in \rho (D_0) \mapsto \big(\widetilde
    \Gamma_0 \upharpoonright \ker(A^*-z)\big)^{-1} = \Phi_z \Lambda,\\
    &\widetilde{M}_z :    z \in \rho(D_0) \mapsto \widetilde \Gamma_1
  \widetilde G_z = \Lambda  \Big(\mathcal{C}_z -\dfrac{1}{2} \big(\mathcal{C}_\zeta+\mathcal{C}_{\Bar\zeta} \big) \Big)\Lambda.
\end{split}
\end{equation}
We define $\Gamma_0,\Gamma_1$ as in \eqref{eq:defn.Gamma}.
The key observation of our proof is that
\begin{equation}\label{eq:gamma.gammatilde}
 \Gamma_0 = \Lambda^{-1}V\Lambda \, \widetilde \Gamma_0, 
\quad  \Gamma_1 = \Lambda V \Lambda^{-1} \, \widetilde
\Gamma_1,
\end{equation}
and $\Lambda^{-1}V\Lambda, \Lambda V \Lambda^{-1}$ are
bounded and boundedly invertible on $L^2(\Sigma;\C^2)$, since 
 $\Lambda : H^{s}(\Sigma) \to H^{s - \frac12}(\Sigma)$ is an
isomorphism for all $s\in \R$ and $V\in C^{\infty}(\Sigma;\C^2)$ is pointwise unitary.
Consequently, the map $f \in \dom S^* \mapsto (\Gamma_0 f, \Gamma_1
f )\in L^2(\Sigma;\C^2)\times L^2(\Sigma;\C^2)$ is surjective, \ie~the condition
\ref{item:boundary.triple.2} in  Definition
\ref{defn:boundary.triple} of boundary triple is fulfilled. In order to verify the condition \ref{item:boundary.triple.1} of the definition, we observe
that, for all $\phi, \psi \in \dom S^*$,
\begin{equation}\label{eq:gamma.gammatilde.1}
  \langle \Gamma_0 \phi , \Gamma_1 \psi  \rangle_{L^2} 
  =
  \langle \Lambda \Gamma_0 \phi , \Lambda^{-1} \Gamma_1 \psi\rangle_{H^{-1/2},H^{1/2}} 
  =
  \langle V \Lambda \, \widetilde \Gamma_0 \phi,
  V \Lambda^{-1} \, \widetilde \Gamma_1 \psi \rangle_{H^{-1/2},H^{1/2}},
\end{equation}
due to \eqref{eq:Lambda.duality} and \eqref{eq:gamma.gammatilde}.
The last term in the previous equation equals
\begin{equation}\label{eq:gamma.gammatilde.2}
  \langle  \Lambda \, \widetilde \Gamma_0 \phi,
  V^* V \Lambda^{-1} \,  \widetilde \Gamma_1 \psi \rangle_{H^{-1/2},
    H^{1/2}}
  =
  \langle  \Lambda \, \widetilde \Gamma_0 \phi,
   \Lambda^{-1} \,  \widetilde \Gamma_1 \psi \rangle_{H^{-1/2},
    H^{1/2}},
\end{equation}
because $V^* V = \bI_2$.
Combining \eqref{eq:gamma.gammatilde.1} and \eqref{eq:gamma.gammatilde.2}, and using \eqref{eq:Lambda.duality} again, we get
\begin{equation}\label{eq:last.boundary}
  \langle \Gamma_0 \phi , \Gamma_1 \psi  \rangle_{L^2}
  =
  \langle  \Lambda \, \widetilde \Gamma_0 \phi,
   \Lambda^{-1} \,  \widetilde \Gamma_1 \psi \rangle_{H^{-1/2},
     H^{1/2}}
   =
   \langle  \widetilde \Gamma_0 \phi,
   \widetilde \Gamma_1 \psi \rangle_{L^2}.
 \end{equation}
This yields \ref{item:boundary.triple.1}
 in Definition \ref{defn:boundary.triple}. Therefore,
we conclude that $\{ L^2(\Sigma; \C^2), \Gamma_0, \Gamma_1 \}$
is a boundary triple for $S^*$. Since $\ker \Gamma_0 = \ker \widetilde
\Gamma_0$, it is true that $\ccD_0=S^*\upharpoonright\ker\Gamma_0$. From
\eqref{eq:abstract.gamma.function}, we get that, for all $z \in \rho(\ccD_0)$,
\begin{equation*}
  G_z :=  \big(\Gamma_0 \upharpoonright \ker(A^*-z)\big)^{-1}
  = \big(\widetilde \Gamma_0 \upharpoonright \ker(A^*-z)\big)^{-1}
  \Lambda^{-1} V^* \Lambda
  = \widetilde G_z   \Lambda^{-1} V^* \Lambda = \Phi_z V^* \Lambda,
\end{equation*}
\ie \eqref{eq:defn.Gz}.
Plugging this result together with \eqref{eq:gamma.gammatilde} into \eqref{eq:abstract.weyl.function}, we obtain
\begin{equation*}
  \begin{split}
  M_z := & \, \Gamma_1 \,G_z = \Lambda V \Lambda^{-1} \widetilde \Gamma_1
  \widetilde G_z   \Lambda^{-1} V^* \Lambda
  =  \Lambda V \Lambda^{-1} 
  \widetilde{M}_z   \Lambda^{-1} V^* \Lambda\\
  = & \, \Lambda V \Big(\mathcal{C}_z -\dfrac{1}{2}
  \big(\mathcal{C}_\zeta+\mathcal{C}_{\Bar\zeta} \big) \Big)V^*
  \Lambda,
\end{split}
\end{equation*}
for all $z \in \rho(\ccD_0)$. This is just \eqref{eq:defn.Mz}. 
\end{proof}

The following lemma is a regularity result concerning the boundary
triple defined in \Cref{prop:boundary.triples.for.Dirac}.
\begin{lemma}\label{lem:regularity}
  Let $f \in \dom S^*$. Then $f \in H^1(\R^2\setminus\Sigma;\C^2)$ if
  and only if $\Gamma_0 f \in H^1(\Sigma;\C^2)$.  
\end{lemma}
\begin{proof}
  The proof is analogous to \cite[Lemma~3.7]{behrndt2019two}, reasoning as in the proof of \Cref{prop:boundary.triples.for.Dirac}.
\end{proof}

The following proposition is a modification of \cite[Proposition
4.3]{behrndt2019two} taking into account the magnetic $\delta$--shell
interaction. 

\begin{proposition}\label{prop:boundary.to.dirac}
  Let $\eta,\tau,\lambda \in C^{\infty}(\Sigma ; \R)$,
  and let $B$ be defined as in \eqref{eq:defn.V.and.B}. Let the operator $\ccD_{\eta,\tau,\lambda}$ be defined as in~\eqref{eq:defn.D.shell.domain},~\eqref{eq:defn.D.shell}. Then the following hold:
  \begin{enumerate}[label=$({\roman*})$]
  \item\label{itm:different.from.zero}
    Assume $d(x)\neq 0$ for all $x\in \Sigma$. Let $\theta \in \Psi_{\Sigma}^1$ be given by
    \begin{equation}\label{eq:defn.theta}
      \theta:=
      -\Lambda\bigg[B^{-1}
      + \frac12\,V(\mathcal{C}_\zeta + \mathcal{C}_{\bar\zeta})V^* \bigg]\Lambda,
    \end{equation}
    and let $\Theta$ be its maximal realization, \ie
    \begin{equation*}
      \Theta \varphi := \theta \varphi, \quad
      \dom \Theta := \{ \varphi \in L^2(\Sigma;\C^2) \mid \theta\varphi
      \in L^2(\Sigma;\C^2)\}.
    \end{equation*}
    Then
    \begin{equation}\label{eq:domain.D.invertible.bB}
      \dom \ccD_{\eta,\tau,\lambda} =
      \{f \in \dom S^* \mid \Gamma_0 f \in \dom \Theta, \, \Gamma_1 f = \Theta
      \Gamma_0 f\}.
    \end{equation}
  \item\label{item:non.invertible.bB.+} 
    Assume $\eta,\tau,\lambda \in \R$ and $\eta = \pm \sqrt{\tau^2 + \lambda^2} \neq 0$.
    Then there exist 
    \begin{equation*}
      \Pi_\pm: L^2(\Sigma; \mathbb{C}^2) \to L^2(\Sigma),
      \quad
      \Pi_\pm^* : L^2(\Sigma) \to L^2(\Sigma; \mathbb{C}^2),
    \end{equation*}
    such that $\Pi_\pm^* \Pi_\pm$ are orthogonal projectors and,
  defining $\theta_{\pm} \in \Psi_{\Sigma}^1$ by
  \begin{equation*}
    \theta_{\pm} :=
    - \Lambda \bigg[ \frac{1}{2\eta} \bI
    + \Pi_{\pm} \,\frac12\,
    V(\mathcal{C}_\zeta + \mathcal{C}_{\bar\zeta})V^*
    \Pi_{\pm}^*\bigg]\Lambda,
  \end{equation*}
  and letting $\Theta_{\pm}$ be its maximal realization, \ie
  \begin{equation*}
    \Theta_{\pm} \varphi := \theta_{\pm} \varphi, \quad
    \dom \Theta_{\pm} := \{ \varphi \in L^2(\Sigma) \mid
    \theta_{\pm}\varphi \in L^2(\Sigma)\},
  \end{equation*}
  we have
  \begin{equation}\label{eq:domain.D.singular.bB}
    \begin{split}
    \dom \ccD_{\eta,\tau,\lambda} =
    \{ f \in \dom S^* \mid     
    \Pi_{\pm} \Gamma_0 f
    \in\dom\Theta_{\pm},
    \, \Pi_{\pm} \Gamma_1 f = \Theta_{\pm}
    \Pi_{\pm} \Gamma_0 f, \, \, \, \, 
    \\
    (\bI  - \Pi_{\pm}^* \Pi_{\pm}) \Gamma_0 f=0 \}. 
  \end{split}
\end{equation}
\end{enumerate}
\end{proposition}
\begin{remark}
 In the case that $\eta = \pm \sqrt{\tau^2 + \lambda^2} = 0$, 
 $\ccD_{\eta,\tau,\lambda}$ is in fact the free Dirac operator $\ccD_0$, defined in \Cref{sec:free.dirac}.
\end{remark}

\begin{proof}
  From \eqref{eq:defn.Gamma}, we see that
  \begin{align*}
i (\sigma\cdot \mathbf{n})\,\big( \cT_+^D f_+ - \cT_-^D f_-\big)&=V^* \Lambda \Gamma_0 f,\\
\dfrac{1}{2}\, \big( \cT_+^D f_+ + \cT_-^D f_-\big)&=V^* \Lambda^{-1}\Gamma_1 f
+\dfrac{1}{2}\,(\mathcal{C}_\zeta+\mathcal{C}_{\Bar\zeta})V^* \Lambda \Gamma_0 f,
\end{align*}
so the transmission condition in \eqref{eq:defn.D.shell.domain}
rewrites as follows:
\begin{equation}
-V^* \Lambda\Gamma_0 f = (\eta\bI_2 + \tau\sigma_3 +\lambda(\sigma\cdot\mathbf{t})) \Big(V^*\Lambda^{-1}\Gamma_1 f
+\dfrac{1}{2}\,(\mathcal{C}_\zeta+\mathcal{C}_{\Bar\zeta}) V^* \Lambda \Gamma_0 f\Big).
\end{equation}
Multiplying the last equation  by $V$ and then using \eqref{eq:propr.V} together with the identity $V^* V = V V^* = \bI_2$, we get
\begin{equation}\label{eq:transmission.rewritten}
B \Lambda^{-1}\Gamma_1 f = - \Big( \bI_2
+\dfrac{1}{2}B\,V(\mathcal{C}_\zeta+\mathcal{C}_{\Bar\zeta}) V^* \Big) \Lambda \Gamma_0 f.
\end{equation}
We prove now \ref{itm:different.from.zero}.
In the case that $d(x) \neq 0$ 
 the matrix $\bB = \bB(x)$ is invertible for all $x\in\Sigma$. 
Thanks to \eqref{eq:transmission.rewritten},  
we obtain the representation in \eqref{eq:domain.D.invertible.bB}.

Next, we pass to the proof of  \ref{item:non.invertible.bB.+}.
  Let
  \begin{equation*}
    \Xi_+, \Xi_-: L^2(\Sigma; \mathbb{C}^2) \to L^2(\Sigma),
    \quad \Xi_+\begin{pmatrix} \varphi_1\\ \varphi_2\end{pmatrix}
    := \varphi_1,
    \quad
    \Xi_-\begin{pmatrix} \varphi_1\\ \varphi_2\end{pmatrix}
    := \varphi_2,
  \end{equation*}
  and $\bU$ be the unitary matrix such that
  \begin{equation}\label{eq:UBU}
    \bB = \bU^*
    \begin{pmatrix}
      \eta + \sqrt{\tau^2 + \lambda^2 } & 0 \\
      0 &       \eta - \sqrt{\tau^2 + \lambda^2 }
    \end{pmatrix}
    \bU.
  \end{equation}
  Finally, let $\Pi_{\pm}:= \Xi_\pm \bU$.
  One sees immediately that $\Pi_{\pm}^* \Pi_{\pm}$ are orthogonal
  projectors in $L^2(\Sigma;\C^2)$, and that $\Pi_{\pm}^*
  \Pi_{\pm}(L^2(\Sigma;\C^2))\subset L^2(\Sigma;\C^2)$ is
  isometrically isomorphic
  to $\Pi_\pm(L^2(\Sigma;\C^2)) \subset L^2(\Sigma)$.

  We give  only a proof in the case that $\eta = \sqrt{\tau^2 +
  \lambda^2}$, the other case being analogous.
  From \eqref{eq:transmission.rewritten},  we infer that
\begin{equation}\label{eq:boundary.1}
  \begin{split}
  -\Lambda \Gamma_0 f = &
  \bU^*
  \begin{pmatrix}
    2\eta & 0 \\
    0 & 0
  \end{pmatrix}
  \bU
  \Big(\Lambda^{-1}\Gamma_1 f
  +\frac{1}{2}\,V(\mathcal{C}_\zeta+\mathcal{C}_{\Bar\zeta})V^*
  \Lambda \Gamma_0 f\Big)
  \\
  = &
  2 \eta \, \Pi_+^* \Pi_+ 
  \Big(\Lambda^{-1}\Gamma_1 f
  +\frac{1}{2}\,V(\mathcal{C}_\zeta+\mathcal{C}_{\Bar\zeta})V^*
  \Lambda \Gamma_0 f\Big).
\end{split}
\end{equation}
We will show that this equation is equivalent to
\eqref{eq:domain.D.singular.bB}.
Multiplying \eqref{eq:boundary.1} by $\bI-\Pi_+^* \Pi_+$ we get
\begin{equation}\label{eq:project.1}
  (\bI-\Pi_+^* \Pi_+) \Lambda \Gamma_0 f = 0.
\end{equation}
Since $\Pi_+ \Pi_+^* = \bI$, multiplying \eqref{eq:boundary.1} by $\Pi_+$
we get
\begin{equation*}
  \Pi_+ \left[ \frac{1}{2\eta}\bI 
    +\frac{1}{2}\,V(\mathcal{C}_\zeta+\mathcal{C}_{\Bar\zeta})V^*
  \right] \Lambda \Gamma_0 f
  =
  - \Pi_+ \Lambda^{-1} \Gamma_1 f.
\end{equation*}
Thanks to \eqref{eq:project.1}, we have
\begin{equation}\label{eq:project.2}
   \left[ \frac{1}{2\eta} \bI
     +\Pi_+ \frac{1}{2}\,V(\mathcal{C}_\zeta+\mathcal{C}_{\Bar\zeta})V^*
     \Pi_+^*
  \right] \Pi_+ \Lambda \Gamma_0 f
  =
  - \Pi_+ \Lambda^{-1} \Gamma_1 f.
\end{equation}
Since $\eta,\tau,\lambda \in \R$, $\Lambda$ commutes with $\Pi_+$ and $\Pi_+^*$. Therefore, taking bijectivity of $\Lambda$ into account, \eqref{eq:project.1} and \eqref{eq:project.2} yield
\begin{equation*}
  \begin{split}
     & (\bI-\Pi_+^* \Pi_+)  \Gamma_0 f = 0, \\
     & \Pi_+  \Gamma_1 f = -\Lambda \left[ \frac{1}{2\eta} \bI
       +\Pi_+ \frac{1}{2}\,V(\mathcal{C}_\zeta+\mathcal{C}_{\Bar\zeta})V^*
       \Pi_+^*
     \right]  \Lambda  \, \Pi_+\Gamma_0 f,
  \end{split}
\end{equation*}
that is, we get the conditions in \eqref{eq:domain.D.singular.bB}.
\end{proof}

In the \emph{non-critical} case, $\mathfrak{C}(\eta,\tau,\lambda)\neq 0$ everywhere on $\Sigma$, with $\mathfrak{C}(\eta,\tau,\lambda)$ defined as in~\eqref{eq:non-critical.condition},
we can show the self-adjointness of
$\ccD_{\eta,\tau,\lambda}$ using  \Cref{thm:boundary.triple.abstract} and
\Cref{prop:boundary.to.dirac},  together with the self-adjointness of
the operator $\Theta$.   The proof is an adaptation of the proof of \cite[Lemma
  4.5]{behrndt2019two} {that takes the new interaction into account}.
\begin{lemma} \label{lem:dom_Theta}
  Let $\eta,\tau,\lambda \in C^{\infty}(\Sigma;\R)$ {be such that}
  $\mathfrak{C}(\eta,\tau,\lambda)\neq 0$ for all $x \in \Sigma$.
  Then the following hold:
  \begin{enumerate}[label=$({\roman*})$]
  \item\label{itm:non-critical.invertible} if $d(x)
    \neq 0$ for all $x \in \Sigma$,
    then $\dom \Theta =
      H^1(\Sigma;\C^2)$ and $\Theta$ is self-adjoint in
      $L^2(\Sigma;\C^2)$,
    \item\label{itm:other.case}
      if $\eta,\tau,\lambda \in \R$ and $\eta = \pm \sqrt{\tau^2 + \lambda^2}\neq 0$, then $\dom \Theta_\pm
      = H^1 (\Sigma)$ and $\Theta_\pm$ is self-adjoint in $L^2 (\Sigma)$.
\end{enumerate}
\end{lemma}
\begin{proof}
  Since the multiplication by $V$ is bounded in $L^2(\Sigma;\C^2)$, from \eqref{Lambda_C_z_Lambda} we get
\begin{equation}\label{eq:self-adjointness.1}
  V(\mathcal{C}_\zeta+\mathcal{C}_{\Bar\zeta})V^* =
  V
  \begin{pmatrix}
    0 & C_\Sigma \overline{T} \\
    T C_\Sigma' & 0
  \end{pmatrix}
  V^* + \widehat\Psi_1 =
  \begin{pmatrix}
    0 & C_\Sigma \\
    C_\Sigma' & 0
  \end{pmatrix}
  + \widehat\Psi_1,
\end{equation}
with $\widehat \Psi_1 \in \Psi_{\Sigma}^{-1}$.

We start the proof of \ref{itm:non-critical.invertible} by putting $\Theta_1:=\Theta\upharpoonright
H^1(\Sigma;\C^2)$.
Since $\theta \in \Psi_{\Sigma}^1$, the operator $\Theta_1$ is
well defined as an operator in $L^2(\Sigma;\C^2)$. To show
\ref{itm:non-critical.invertible} we {prove} that $\Theta_1 = \Theta$ and
$\Theta_1$ is self-adjoint in $L^2(\Sigma;\C^2)$.
Since $(V(\mathcal{C}_\zeta+\mathcal{C}_{\Bar\zeta})V^*)^* =
V(\mathcal{C}_{\Bar\zeta}+ \mathcal{C}_\zeta)V^*$ and $\Lambda$ is
self-adjoint as an operator in $L^2(\Sigma)$, $\Theta_1$ is symmetric.
Moreover, since $(\Theta\upharpoonright C^{\infty}) \subset {\Theta_1}$,
we have $\Theta_1 \subset \Theta_1^* \subset (\Theta\upharpoonright
C^{\infty})^*$. We conclude that $\Theta_1 \subset \Theta$ since
$\Theta$ is the maximal realization of 
$\theta$, and so $(\Theta\upharpoonright
C^{\infty})^* = \Theta$.

Hence, to show that $\Theta_1 = \Theta$, it is now sufficient to prove
that $\Theta \subset \Theta_1$, that
is to say $\dom \Theta \subset \dom \Theta_1 = H^1(\Sigma;\C^2)$.
We fix $\varphi \in \dom \Theta$.
Thanks to 
\Cref{prop:boundary.to.dirac} \ref{itm:different.from.zero}  and
\eqref{eq:self-adjointness.1} we have
\begin{equation*}
  \theta \varphi =
  - \Lambda P \Lambda \varphi + \widehat \Psi_2 \varphi,
\end{equation*}
with $\widehat \Psi_2 \in \Psi_{\Sigma}^0$ and  
\begin{equation}\label{eq:P}
  P =  \frac{1}{d}
  \begin{pmatrix}
    \eta - \tau & -\lambda  \\
    -\lambda & \eta +\tau
  \end{pmatrix} +
  \frac12
  \begin{pmatrix}
    0 & C_{\Sigma} \\
     C_{\Sigma}' & 0
  \end{pmatrix}
  =
  \frac12 \begin{pmatrix}
    \frac{2(\eta - \tau)}{d} & C_{\Sigma} - \frac{2\lambda}{d} \\
    C_{\Sigma}' - \frac{2\lambda}{d} & \frac{2(\eta + \tau)}{d}
    \end{pmatrix}.
\end{equation}
We have $\Lambda P \Lambda \varphi \in L^2(\Sigma;\C^2)$ and $P
\Lambda \varphi \in H^{\frac12}(\Sigma;\C^2)$ because $\Lambda :
H^{\frac12}(\Sigma;\C^2) \to L^{2}(\Sigma;\C^2)$ is an isomorphism.
Since $C_{\Sigma}, C_{\Sigma}'\in \Psi_{\Sigma}^0$, these
pseudodifferential operators give rise to bounded operators in
$H^{\frac12}(\Sigma;\C^2)$, and this implies that
\begin{equation*}
  \begin{split}
    & \frac12
    \begin{pmatrix}
      \frac{2(\eta + \tau)}{d} & - C_{\Sigma} - \frac{2\lambda}{d} \\
     - C_{\Sigma}' - \frac{2\lambda}{d} & \frac{2(\eta - \tau)}{d}
    \end{pmatrix}
    P
    \Lambda \varphi
    \\
    &
    = \frac{1}{d^2}
    \begin{pmatrix}
      d + 2\lambda^2 - \frac{d^2}{4} C_\Sigma C_{\Sigma}'+\frac{\lambda d}{2}(C_\Sigma-C_\Sigma') &
      -2(\eta + \tau)\lambda \\
      -2(\eta - \tau)\lambda &       d + 2\lambda^2 - \frac{d^2}{4}  C_\Sigma' C_\Sigma +\frac{\lambda d}{2}(C_\Sigma'-C_\Sigma)
    \end{pmatrix}
    \Lambda \varphi
  \end{split}
\end{equation*}
belongs to $H^{\frac12}(\Sigma;\C^2)$.
Using \eqref{eq:psidiff_inf}, we conclude that
\begin{equation*}
  M \Lambda \varphi
  :=
  \frac{1}{d^2}
  \begin{pmatrix}
      \eta^2 - \tau^2 + \lambda^2 - \frac{d^2}{4} &
      {-}2(\eta + \tau)\lambda \\
      {-}2(\eta - \tau)\lambda &       \eta^2 - \tau^2 + \lambda^2 - \frac{d^2}{4}
    \end{pmatrix}
    \Lambda \varphi \in H^{\frac12}(\Sigma;\C^2).
  \end{equation*}
  Note that
  \begin{equation}\label{eq:condition.sa}
    \det M =\frac{1}{d^4}\left[\left(
        \eta^2-\tau^2+\lambda^2-\frac{d^2}{4}
      \right)^2
      -4(\eta^2-\tau^2)\lambda^2\right]
    =
    \frac{1}{d^2}
    \mathfrak{C}(\eta,\tau,\lambda) 
  \end{equation}
 and, by our hypothesis, $M$ is invertible.   Therefore, we get $\Lambda \varphi \in
H^{\frac12}(\Sigma;\C^2)$ and $\varphi  \in H^1(\Sigma;\C^2)$, because $\Lambda :
H^{1}(\Sigma;\C^2) \to H^{\frac12}(\Sigma;\C^2)$ is an isomorphism.
This completes the proof of the case \ref{itm:non-critical.invertible}. 

Now we pass to the proof of  \ref{itm:other.case}.
Arguing as in the proof of \ref{itm:non-critical.invertible}, let
$\Theta_{\pm,1}:=\Theta_\pm\upharpoonright H^1(\Sigma)$. It is true that
$\Theta_{\pm,1} \subset \Theta_{\pm,1}^* \subset \Theta_\pm$, and  so we conclude
the proof if we show that $\dom \Theta_\pm \subset \dom \Theta_{\pm,1} =
H^1(\Sigma)$. By \Cref{prop:boundary.to.dirac}
\ref{item:non.invertible.bB.+} and \eqref{eq:self-adjointness.1}, we have
\begin{equation*}
  \Theta_\pm \varphi = -\Lambda \bigg[ \frac{1}{2 \eta} \bI + \dfrac{1}{2} \Pi_\pm \begin{pmatrix}
     0 & C_\Sigma \\
       C_\Sigma' & 0
    \end{pmatrix}  \Pi_\pm^* \bigg]\Lambda \varphi  + \widehat{\Psi}
    \varphi\quad
    \text{ for all }\varphi \in \dom \Theta_\pm,
\end{equation*}
with some symmetric operator $\widehat{\Psi} \in \Psi_\Sigma^0$. Since
$\Lambda : H^{\frac12}(\Sigma)\to L^2(\Sigma)$ is an isomorphism, the
last equation implies that
\begin{equation}\label{eq:d.zero.1}
  \bigg[ \frac{1}{2 \eta}\bI + \dfrac{1}{2} \Pi_\pm \begin{pmatrix}
     0 & C_\Sigma \\
       C_\Sigma' & 0
    \end{pmatrix}  \Pi_\pm^* \bigg]\Lambda \varphi \in H^{\frac12}(\Sigma).
\end{equation}
For the unitary matrix $\bU$ in \eqref{eq:UBU}, we may choose
\begin{equation*}
  \bU =
  \begin{cases}
    \bI_2
    \quad &\text{ if } \lambda = 0,
    \\
    ((\tau+\sqrt{\tau^2 +\lambda^2})^2+\lambda^2)^{-\frac12}
    \begin{pmatrix}
      \tau + \sqrt{\tau^2 +\lambda^2} &  \lambda \\
      -\lambda  &  \tau + \sqrt{\tau^2 +\lambda^2}
    \end{pmatrix}
    \quad &\text{ if } \lambda \neq 0.
  \end{cases}
\end{equation*}

When $\lambda = 0$ our choice of $\bU$ gives $\Pi_\pm =
\Xi_\pm$, and from \eqref{eq:d.zero.1} we get
\begin{equation}\label{eq:2}
   \bigg[ \frac{1}{2 \eta} \bI + \dfrac{1}{2} \Pi_\pm \begin{pmatrix}
     0 & C_\Sigma \\
       C_\Sigma' & 0
    \end{pmatrix}  \Pi_\pm^* \bigg]\Lambda \varphi
=
\frac{1}{2\eta} \Lambda \varphi \in H^{\frac12}(\Sigma).
\end{equation}
Since $\Lambda : H^{1}(\Sigma) \to
H^{\frac12}(\Sigma)$ is an isomorphism, we get $\varphi \in
H^1(\Sigma)$.
If $\lambda \neq 0$ then \eqref{eq:d.zero.1} yields
\begin{equation*}
  \bigg[ \frac{1}{2 \eta} \bI + \dfrac{1}{2} \Pi_\pm \begin{pmatrix}
    0 & C_\Sigma \\
    C_\Sigma' & 0
  \end{pmatrix}  \Pi_\pm^* \bigg]\Lambda \varphi
  =
  \left[
  \frac{1}{2\eta}\bI \pm \frac{\lambda}{4\eta} (C_\Sigma + C_\Sigma')\right] \Lambda \varphi
  \in H^{\frac12}(\Sigma).
\end{equation*}
Since $C_\Sigma,C_\Sigma \in \Psi_\Sigma^0$, we get
\begin{equation*}
  \Big(\bI \mp \frac{\lambda}{2} (C_\Sigma + C_\Sigma')\Big) 
  \Big(\bI \pm \frac{\lambda}{2} (C_\Sigma + C_\Sigma')\Big) \Lambda \varphi
  =
  \left[\bI -\frac{\lambda^2}{4}(C_\Sigma + C_\Sigma')^2\right] \Lambda \varphi
  \in H^{\frac12}(\Sigma).
\end{equation*}
 
Taking \eqref{eq:psidiff_inf} into account, we finally obtain
\begin{equation*}
    (\bI -{\lambda^2}) \Lambda \varphi \in H^{\frac12}(\Sigma).
\end{equation*}
We conclude that $\Lambda \varphi \in H^{\frac12}(\Sigma)$ since for $d=0$
the condition $\mathfrak{C}(\eta,\tau,\lambda)\neq 0$
forces $\lambda^2 \neq 1$. Since $\Lambda : H^{1}(\Sigma) \to
H^{\frac12}(\Sigma)$ is an isomorphism, we get $\varphi \in
H^1(\Sigma)$ also in this case.
\end{proof}

We are now ready to show the
self-adjointness of $\ccD_{\eta,\tau,\lambda}$ in the non-critical case.
\begin{theorem}\label{thm:self-adjointess.noncritical}
  Let $\eta,\tau,\lambda \in C^{\infty}(\Sigma;\R)$ be such that
  $\mathfrak{C}(\eta,\tau,\lambda)(x)\neq 0$. Moreover, let either $d(x)\neq 0$ for all $x\in \Sigma$,
  or let $\eta,\tau,\lambda$ be constant and such that $d = 0$.
  Then
  $\ccD_{\eta,\tau,\lambda}$, defined by \eqref{eq:defn.D.shell.domain} and \eqref{eq:defn.D.shell}, is self-adjoint in $L^2(\R^2;\C^2)$ with
  domain $\dom \ccD_{\eta,\tau,\lambda} \subset H^1(\R^2 \setminus \Sigma;\C^2)$.
  Moreover, for all $z \in \rho(\ccD_{\eta,\tau,\lambda}) \cap
  \rho(\ccD_0)$ the operator $\bI_2 + (\eta\bI_2 + \tau\sigma_3
  +\lambda(\sigma\cdot\mathbf{t}))\mathcal{C}_z$ is bounded and
  boundedly invertible in $H^{\frac12}(\Sigma;\C^2)$ and
  \begin{equation} \label{krein_noncritical}
    \begin{split}
      &(\ccD_{\eta,\tau,\lambda} - z)^{-1} \\
      &= (\ccD_0 - z)^{-1} - \Phi_z \big(
      \bI_2 + (\eta \bI_2 + \tau \sigma_3 + \lambda(\sigma\cdot\mathbf{t})
      ) \mathcal{C}_z \big)^{-1} (\eta \bI_2 + \tau \sigma_3 +
      \lambda(\sigma\cdot\mathbf{t})) \Phi_{\Bar{z}}'.
    \end{split}
  \end{equation}
\end{theorem}
\begin{proof}
 The proof is analogous to the proof of \cite[Theorem
  4.6]{behrndt2019two}, now when we have set the right
  framework. 
  
  The self-adjointness of $\ccD_{\eta,\tau,\lambda}$  follows from the
  self-adjointness of $\Theta$ and $\Theta_\pm$ in $L^2(\Sigma;\C^2)$
  and $L^2(\Sigma)$ respectively, thanks to
  \Cref{thm:boundary.triple.abstract}.
  Moreover, \Cref{lem:regularity} implies that $\dom
  \ccD_{\eta,\tau,\lambda} \subset H^1(\R^2 \setminus \Sigma;\C^2)$
  since, by Lemma~\ref{lem:dom_Theta}, $\dom \Theta = H^1(\Sigma;\C^2)$ and
  $\dom \Theta_\pm = H^1(\Sigma)$.
  
  We show \eqref{krein_noncritical} in the case that $d(x)\neq 0$ for
  all $x \in \Sigma$.
  By \Cref{thm:boundary.triple.abstract} \ref{item:abstract.Krein},
  $\Theta - M_z$ is boundedly invertible in $L^2(\Sigma;\C^2)$, for $z \in \rho(\ccD_{\eta,\tau,\lambda})\cap \rho(\ccD_0)$, and
  \begin{equation*}
      (\ccD_{\eta,\tau,\lambda} - z)^{-1} = (\ccD_0 - z)^{-1} + G_z \big( \Theta - M_z \big)^{-1} G_{\Bar{z}}^*.
  \end{equation*}
  From the definition of $M_z$ \eqref{eq:defn.Mz}, we get
  \begin{equation}\label{eq:to.be.used}
    \Theta - M_z = -\Lambda V(\eta \bI_2 + \tau \sigma_3 +
    \lambda(\sigma\cdot\mathbf{t}))^{-1}
    (\bI_2 + (\eta \bI_2 + \tau \sigma_3 +
    \lambda(\sigma\cdot\mathbf{t}))\mathcal{C}_z)
    V^* \Lambda.
  \end{equation}
  The operator $\Theta - M_z$ is bijective in $L^2(\Sigma;\C^2)$ when defined on $\dom
  \Theta = H^1(\Sigma;\C^2)$, and $(\bI_2 + (\eta \bI_2 + \tau \sigma_3 +
    \lambda(\sigma\cdot\mathbf{t}))\mathcal{C}_z)$ is well defined and
    bounded in $H^{\frac12}(\Sigma;\C^2)$. Recalling the definition of
    $G_z$ from \eqref{eq:defn.Gz} we get \eqref{krein_noncritical}.
  The case $d=0$ is analogous and will be omitted.
\end{proof}

In the next proposition we gather some basic results on the spectrum
of $\ccD_{\eta,\tau,\lambda}$.
\begin{proposition} \label{prop:spectral.properties.noncritical}
  Let $\eta,\tau,\lambda \in C^{\infty}(\Sigma;\R)$ be such that
  $\mathfrak{C}(\eta,\tau,\lambda)(x)\neq 0$ everywhere on $\Sigma$. Moreover, let either $d(x)\neq 0$ for all $x\in \Sigma$,
  or let $\eta,\tau,\lambda $ be constant and such that $d = 0$.
  Let $\ccD_{\eta,\tau,\lambda}$ be defined as in
  \eqref{eq:defn.D.shell.domain},\eqref{eq:defn.D.shell}.
  Then the following hold:
  \begin{enumerate}[label=$({\roman*})$]
  \item\label{item:spectrum.noncritical.1} We have $\sigma_{ess} (\ccD_{\eta,\tau,\lambda})=(-\infty,
    -\abs{m}] \cup [\abs{m},+\infty\big)$;
    if in particular $m=0$, then $\sigma (\ccD_{\eta,\tau,\lambda}) = \sigma_{ess} (\ccD_{\eta,\tau,\lambda})= \R$.
  \item\label{item:spectrum.noncritical.2} If $m\neq 0$, then $\ccD_{\eta,\tau,\lambda}$ has at most
    finitely many eigenvalues in $\big(-|m|,|m|\big)$.
  \item\label{item:spectrum.noncritical.3} Assume $m\neq 0$. Then $z \in (-|m|, |m|)$ is a discrete
    eigenvalue of  $\ccD_{\eta,\tau,\lambda}$ if and only if
    there exists $\varphi \in H^{\frac{1}{2}}(\Sigma; \mathbb{C}^2)$
    such that $\big(\bI + (\eta \bI + \tau \sigma_3 +
    \lambda(\sigma\cdot\mathbf{t})) \mathcal{C}_z \big) \varphi = 0$.
  \end{enumerate}
\end{proposition}
\begin{proof}
  The proof is analogous to the proof of
  \cite[Theorem~4.7]{behrndt2019two}.
  Firstly, let us show
  \ref{item:spectrum.noncritical.1} and
  \ref{item:spectrum.noncritical.2}.
  Thanks to \Cref{prop:extensions.spectral.properties}
  \ref{item:extensions.spectral.properties.1},
  $(-\infty,-\abs{m}] \cup [\abs{m},+\infty\big) \subset
  \sigma_{ess}(\ccD_{\eta,\tau,\lambda})$.
  Moreover, since $\dom \ccD_{\eta,\tau,\lambda} \subset H^1(\R^2\setminus \Sigma;\C^2)$, due to \Cref{thm:self-adjointess.noncritical},
  \Cref{prop:extensions.spectral.properties}
  \ref{item:extensions.spectral.properties.2} implies that the
  spectrum of $\ccD_{\eta,\tau,\lambda}$ in $(-\abs{m}, \abs{m})$ is
  discrete and finite. 

  We show  \ref{item:spectrum.noncritical.3} only in the case that
  $d(x)\neq 0$ for all $x \in \Sigma$, the case $d= 0$ being
  similar. 
  By \Cref{thm:boundary.triple.abstract}
  \ref{item:abstract.spectrum} combined with \eqref{eq:to.be.used}, $z \in \rho(\ccD_0)$ is an eigenvalue of
  $\ccD_{\eta,\tau,\lambda}$ if and only if there exists $\psi \in \dom
  \Theta =  H^1(\Sigma;\C^2)$ such that
  \begin{equation*}
    -\Lambda V (\eta \bI_2 + \tau \sigma_3 +
    \lambda(\sigma\cdot\mathbf{t}))^{-1}
    (\bI_2 + (\eta \bI_2 + \tau \sigma_3 +
    \lambda(\sigma\cdot\mathbf{t}))\mathcal{C}_z)
    V^* \Lambda \psi = 0,
  \end{equation*}
  \ie if and only if $\varphi = V^* \Lambda \psi \in
  H^{\frac12}(\Sigma;\C^2)$ satisfies
  \begin{equation*}
    (\bI_2 + (\eta \bI_2 + \tau \sigma_3 +
    \lambda(\sigma\cdot\mathbf{t}))\mathcal{C}_z)
    \varphi = 0.\qedhere
  \end{equation*}
\end{proof}

\section{The purely magnetic critical interaction}
\label{sec:critical}
In this section we give the proof of \Cref{thm:critical}: we consider
the case $\lambda = 2$ only, since the case $\lambda= -2$ can be treated analogously.
Recall that by Theorem~\ref{thm:confinement.introduction} the operator $\ccD_{0,0,2}$ can be decomposed into the
orthogonal sum
\begin{equation*}
	\ccD_{0,0,2} = \ccD_{0,0,2}^+ \oplus \ccD_{0,0,2}^- =: \ccD^+ \oplus \ccD^-,
\end{equation*}
with
\begin{equation*}
\begin{split}
	&\dom \ccD^\pm := 
	\big\{ f_\pm \in H(\sigma,\Omega_\pm) \mid
	\left[\pm i (\sigma \cdot \mathbf{n}) +
	 (\sigma \cdot \mathbf{t})) \right] 
	\mathcal{T}_\pm^D f_{\pm} = 0
	\big\}, \\
	& \ccD^\pm f_\pm := \ccD_0 f_\pm,
	\quad \text{ for all }f_\pm \in \dom \ccD^\pm.
\end{split}
\end{equation*}
Using that $n_1 = t_2$, $n_2 = -t_1$ we find
\begin{equation*}
	i(\sigma\cdot\mathbf{n}) + (\sigma\cdot \mathbf{t}) 
	= 
	\begin{pmatrix} 0 & 0 \\ 2T & 0 \end{pmatrix},
        \quad
        -i(\sigma\cdot\mathbf{n}) + (\sigma\cdot \mathbf{t}) 
	=  \begin{pmatrix} 0 & 2\overline{T}\\ 0 & 0 \end{pmatrix}.
\end{equation*}
Hence, $\ccD^\pm$ have the following representations
\[
\begin{aligned}
	\ccD^+ f = \ccD_0f,&\qquad \dom \ccD^+ = \{f = (f_1,f_2)^\top
        \in H(\sigma, \Omega_+)\mid
	\mathcal{T}_+^D f_{1} = 0\},\\
	\ccD^- f = \ccD_0f,&\qquad \dom \ccD^- = \{f = (f_1,f_2)^\top
        \in H(\sigma, \Omega_-)\mid
        \mathcal{T}_-^D f_{2} = 0\}.
\end{aligned}	
\]
Using the Cauchy-Riemann differential expressions
\[
	\partial_z := \frac12(\partial_1 -i \partial_2)\qquad\text{and}\qquad
	\partial_{\bar z} :=  \frac12(\partial_1 +i \partial_2),
\]
we can represent $\ccD^\pm$ as follows
\[
\begin{aligned}
	\ccD^\pm f & = 
	\begin{pmatrix} m & -2i\partial_z\\
	-2i\partial_{\bar z} & -m\end{pmatrix}
	\begin{pmatrix}f_1\\ f_2\end{pmatrix},\\
	\dom \ccD^+ &= \{f = (f_1,f_2)^\top\mid
	f_1,f_2,\partial_z f_2, \partial_{\bar z} f_1\in L^2(\Omega_+), \mathcal{T}_+^D f_{1} = 0\},\\
	%
	\dom \ccD^- &= \{f = (f_1,f_2)^\top \mid f_1,f_2,\partial_z
        f_2, \partial_{\bar z} f_1\in L^2(\Omega_-),
        \mathcal{T}_-^D f_{2} = 0\}.
\end{aligned}
\]
In view of~\cite[Lemma 18]{antunes2020variational} (see also~\cite[Lemma 3.1]{behrndt2019two}) the domains of $\ccD^\pm$ can alternatively be given by
\[
\begin{aligned}
	\dom \ccD^+ & = 
	\{f = (f_1,f_2)^\top\mid
	f_2, \partial_{z} f_2\in L^2(\Omega_+), f_1\in H^1_0(\Omega_+)\},\\
	\dom \ccD^-& = \{f = (f_1,f_2)^\top\mid f_1,\partial_{\bar z} f_1\in L^2(\Omega_-), f_2\in H^1_0(\Omega_-)\}.
\end{aligned}
\]
The operators $\ccD^\pm$ can be viewed as bounded symmetric perturbations of the respective massless Dirac operators. Since by~\cite[Proposition 1]{schmidt1995remark}
the unperturbed massless Dirac operators are self-adjoint, we conclude that $\ccD^\pm$ are self-adjoint as well.

Next, we will show the symmetry of the spectrum of $\ccD^\pm$. For $\nu\in\C$ such that $\nu^2-m^2\ne 0$ we introduce the matrix
\[
	T_\nu = \begin{pmatrix} \sqrt{\frac{m+\nu}{m-\nu}} & 0 \\
	0& \sqrt{\frac{m-\nu}{m+\nu}}\end{pmatrix}.
\]
Clearly, $T_\nu$ is invertible and $T_\nu^{-1} = T_{-\nu}$. Moreover, for any $\nu\in\mathbb{C}\setminus \{-m,m\}$ we have $T_\nu(\dom \ccD^\pm) = \dom \ccD^\pm$ and
\begin{equation}\label{eq:commutation}
	(\ccD^\pm +\nu)T_{-\nu} = T_\nu(\ccD^\pm - \nu).
\end{equation}
Let $\nu\in\sigma(\ccD^\pm)\setminus\{-m,m\}$, then there exists a sequence $(\psi_n)_n$ in $\dom \ccD^\pm$
such that
\[
	\lim_{n\rightarrow\infty} \frac{\|(\ccD^\pm -\nu)\psi_n\|}{\|\psi_n\|} =0.
\]
Let $\phi_n := T_\nu^{-1}\psi_n = T_{-\nu}\psi_n$. Then we get 
\[
\begin{aligned}
	\frac{\|(\ccD^\pm +\nu)\phi_n\|}{\|\phi_n\|} & = \frac{\|(\ccD^\pm +\nu)T_{-\nu}\psi_n\|}{\|T_{-\nu}\psi_n\|}=
	\frac{\|T_{\nu}(\ccD^\pm -\nu)\psi_n\|}{\|T_{\nu}^{-1}\psi_n\|}\\
	& \le\|T_\nu\|^2 \frac{\|(\ccD^\pm -\nu)\psi_n\|}{\|\psi_n\|}\rightarrow 0,\qquad\text{as}~ 
	n\rightarrow\infty.
\end{aligned}
\]
Hence, we conclude that $-\nu\in\sigma(\ccD^\pm)$. Moreover, if $\nu\neq\pm m$ is an eigenvalue of $\ccD^\pm$, then in view of identity~\eqref{eq:commutation} $-\nu$ is also an eigenvalue of $\ccD^\pm$.

Now we perform the spectral analysis of $\ccD^+$. 
First of all, we notice the inclusion 
\[
	\ker(\ccD^+ + m) \supset\left\{\begin{pmatrix}0 \\
            f_2\end{pmatrix} \, \middle\vert \, f_2 \in L^2(\Omega_+),
	\partial_{z} f_2 = 0\right\}. 
\]
Indeed,
\[
	\ccD^+\begin{pmatrix} 0 \\ f_2 \end{pmatrix} = \begin{pmatrix} -2i \partial_{z} f_2\\ -mf_2\end{pmatrix} = -m\begin{pmatrix} 0\\ f_2 \end{pmatrix}. 
\]
Since the space of square-integrable anti-holomorphic functions on
$\Omega_+$ is infinite-dimensional, $-m$ is an eigenvalue of infinite
multiplicity in the spectrum of $\ccD^+$. In particular, $\dom \ccD^+ \not\subset
H^s(\Omega_+;\mathbb{C}^2)$ for any $s > 0$, as otherwise the spectrum of $\ccD^+$  would be purely discrete, due to the compactness of embedding of the Sobolev spaces $H^s(\Omega_+;\C^2)$, $s > 0$, into $L^2(\Omega_+;\C^2)$. Thus, $\dom \ccD_{0,0,2} \not\subset H^s(\mathbb{R}^2\setminus\Sigma;\mathbb{C}^2)$ for any $s > 0$. 

Consider the auxiliary operators 
\begin{equation}\label{eq:Apm}
\begin{aligned}
A_+\psi & = -2i\partial_{\bar z} \psi, \qquad \dom A_+ = H^1_0(\Omega_+),\\
A_-\psi & = -2i\partial_{z} \psi, \qquad \dom A_- = H^1_0(\Omega_-).\\
\end{aligned}
\end{equation}
The adjoints of $A_\pm$ are characterised in the spirit of~\cite[Proposition 1]{schmidt1995remark} as
\[
\begin{aligned}
	A_+^*\psi & = -2i\partial_{ z}\psi,\qquad \dom A_+^* = \big\{\psi\in L^2(\Omega_+)\mid \partial_{z}\psi\in L^2(\Omega_+)\big\},\\
	A_-^*\psi & = -2i\partial_{\bar z}\psi,\qquad \dom A_-^* = \big\{\psi\in L^2(\Omega_-)\mid \partial_{\bar z}\psi\in L^2(\Omega_-)\big\}.\\
\end{aligned}
\]
The quadratic form for $(\ccD^+)^2$ is given by
\[
	\mathfrak{h}^+[f] := \|\ccD^+f\|^2_{L^2(\Omega_+;\mathbb{C}^2)},\qquad \dom\mathfrak{h}^+ := \dom \ccD^+.
\] 
Next, we compute
\[
\begin{aligned}
	\mathfrak{h}^+[f]  = & \, \|\ccD^+f\|^2_{L^2(\Omega_+;\mathbb{C}^2)} =
	\|mf_1 - 2i\partial_zf_2\|^2_{L^2(\Omega_+)}
	+
	\|-mf_2 - 2i\partial_{\bar z}f_1\|^2_{L^2(\Omega_+)}\\
	 = & \,
	4\|\partial_zf_2\|^2_{L^2(\Omega_+)}
	+
	4\|\partial_{\bar z}f_1\|^2_{L^2(\Omega_+)} \\
	&
        +
	4m\Re\left[ (f_2,i\partial_{\bar z} f_1)_{L^2(\Omega_+)} -
          (f_1,i\partial_z f_2)_{L^2(\Omega_+)}   \right]
        \\
        & + 
	m^2\|f_1\|_{L^2(\Omega_+)}^2 + m^2\|f_2\|_{L^2(\Omega_+)}^2.
\end{aligned}	
\]
Integrating by parts, we find with the aid of $\dom\mathfrak{h}^+ = \dom \ccD^+$ that
\begin{equation*}
  \begin{split}
	 \Re &\left[ (f_2,i\partial_{\bar z} f_1)_{L^2(\Omega_+)}  -
          (f_1,i\partial_z f_2)_{L^2(\Omega_+)} \right]
        \\
	 & \quad =
	\Re\left[(i \partial_z f_2, f_1)_{L^2(\Omega_+)} -
          (f_1,i\partial_z f_2)_{L^2(\Omega_+)}\right]= 0. 
      \end{split}
    \end{equation*}
Thus, the expression for $\mathfrak{h}^+$ simplifies as
\[
\begin{aligned}
	\mathfrak{h}^+[f] & = 4\|\partial_zf_2\|^2_{L^2(\Omega_+)}
	+
	4\|\partial_{\bar z}f_1\|^2_{L^2(\Omega_+)} + 
	m^2\|f_1\|_{L^2(\Omega_+)}^2 + m^2\|f_2\|_{L^2(\Omega_+)}^2\\
	& = \|A_+^*f_2\|^2_{L^2(\Omega_+)}
	+
	\|A_+f_1\|^2_{L^2(\Omega_+)} + 
	m^2\|f_1\|_{L^2(\Omega_+)}^2 + m^2\|f_2\|_{L^2(\Omega_+)}^2.
\end{aligned}	
\]
The domain of $\mathfrak{h}^+$ can be written as
\[
	\dom\mathfrak{h}^+ =\big\{ f = (f_1,f_2)^\top\mid f_1\in \dom A_+, f_2\in\dom A_+^*\big\}
	=	\dom A_+\oplus \dom A_+^*.
\]
Therefore, we end up with the orthogonal decomposition
\[
	(\ccD^+)^2 = (A_+^*A_+ +m^2)\oplus (A_+A_+^* + m^2),
\]
from which we deduce, using ~\cite[Propositions 2 and 3]{schmidt1995remark}, that 
\[
	\sigma((\ccD^+)^2)\setminus\{m^2\} =\big\{m^2+\mu\mid \mu\in\sigma(-\Delta_{\rm D}^{\Omega_+})\big\},
\]
where $-\Delta_{\rm D}^{\Omega_+}$ is the Dirichlet Laplacian on $\Omega_+$.  
Using the symmetry of the spectrum shown above we obtain that
\[
	\sigma(\ccD^+)\setminus \{-|m|,|m|\} = \big\{\pm\sqrt{m^2+\mu}\mid \mu\in\sigma(-\Delta_{\rm D}^{\Omega_+})\big\}.
\]
Hence,  \ref{rm_(ii)} of \Cref{thm:critical} follows. Moreover, we observe that $\ccD^+\upharpoonright H^1(\Omega_+;\mathbb{C}^2)$ is essentially self-adjoint provided that $H^1_0(\Omega_+)\oplus H^1(\Omega_+)$
is a core for $\ccD^+$.  The latter follows from the density
of $H^1(\Omega_+)$ in $\dom A_+^*$; cf. \cite[Lemma 14]{antunes2020variational}.
 
Now we perform the spectral analysis of $\ccD^-$. As in the analysis of $\ccD^+$  we notice the inclusion 
\[
	\ker(\ccD^- - m) \supset\left\{\begin{pmatrix}f_1 \\
	0\end{pmatrix} \, \middle\vert \, f_1 \in L^2(\Omega_-),
	\partial_{\overline{z}} f_1 = 0\right\}. 
\]
Indeed,
\[
\ccD^-\begin{pmatrix} f_1 \\ 0 \end{pmatrix} = \begin{pmatrix} mf_1\\-2i \partial_{\overline{z}} f_1 \end{pmatrix} = m\begin{pmatrix} f_1\\ 0 \end{pmatrix}. 
\]
We observe that the space of square-integrable holomorphic functions on
$\Omega_-$ is also infinite-dimensional.
Indeed, for an arbitrary $z_0 \in \Omega_+$  the family of linear independent functions
$\{(z-z_0)^{-k}\}_{k\ge 2}$ is square-integrable and holomorphic in $\Omega_-$
Hence, $m$ is an eigenvalue of infinite
multiplicity in the spectrum of $\ccD^-$ and thus combining with the fact that $-m$
is an eigenvalue of infinite multiplicity in the spectrum of $\ccD_+$ shown above the claim of (ii) of Theorem \ref{thm:critical}
follows.
Next we consider the quadratic form
\[
	\mathfrak{h}^-[f] := \|\ccD^-f\|^2_{L^2(\Omega_-;\mathbb{C}^2)}\qquad \dom\mathfrak{h}^- := \dom \ccD^-.
\] 
 for $(\ccD^-)^2$. Repeating the same type of computation as we did for $\ccD^+$ we get
\[
\begin{aligned}
	\mathfrak{h}^-[f] & := \|\ccD^-f\|^2_{L^2(\Omega_-;\mathbb{C}^2)} =
	\|mf_1 - 2i\partial_zf_2\|^2_{L^2(\Omega_-)}
	+
	\|-mf_2 - 2i\partial_{\bar z}f_1\|^2_{L^2(\Omega_-)}\\
	& =
	4\|\partial_zf_2\|^2_{L^2(\Omega_-)}
	+
	4\|\partial_{\bar z}f_1\|^2_{L^2(\Omega_-)} + 
	m^2\|f_1\|_{L^2(\Omega_-)}^2 + m^2\|f_2\|_{L^2(\Omega_-)}^2\\
	& =
	\|A_-f_2\|^2_{L^2(\Omega_-)}
	+
	\|A_-^*f_1\|^2_{L^2(\Omega_-)} + 
	m^2\|f_1\|_{L^2(\Omega_-)}^2 + m^2\|f_2\|_{L^2(\Omega_-)}^2.
\end{aligned}	
\]
The domain of $\mathfrak{h}^-$ can be written as
\[
\dom\mathfrak{h}^- =\dom A_-^*\oplus \dom A_-.
\]
Hence, we have the orthogonal decomposition
\[
	(\ccD^-)^2 = (A_-A^*_- +m^2)\oplus (A^*_-A_- + m^2),
\]
which implies in view of~\cite[Propositions 2 and 3]{schmidt1995remark} that 
\[
\sigma((\ccD^-)^2)\setminus \{m^2\} =\big\{m^2+\mu\mid \mu\in\sigma(-\Delta_{\rm D}^{\Omega_-})\setminus \{0\}\big\},
\]
where $-\Delta_{\rm D}^{\Omega_-}$ is the Dirichlet Laplacian on $\Omega_-$.  
Taking that $\sigma(-\Delta_{\rm D}^{\Omega_-}) = [0,\infty)$ into account we get
$\sigma((\ccD^-)^2) = [m^2,+\infty)$. In view of the symmetry of the spectrum of $\ccD^-$ shown above we necessarily get that $\sigma(\ccD^-) = (-\infty,-|m|]\cup [|m|,+\infty)$. Hence, \ref{rm_(iii)} of \Cref{thm:critical} is shown. Essential self-adjointness of $\ccD^-\upharpoonright H^1(\Omega_-;\mathbb{C}^2)$ follows analogously to that of $\ccD^+\upharpoonright H^1(\Omega_+;\mathbb{C}^2)$.

Essential self-adjointness of $\ccD_{0,0,2}\upharpoonright H^1(\mathbb{R}^2\setminus\Sigma;\mathbb{C}^2)$ follows from essential self-adjointness of $\ccD^\pm \upharpoonright H^1(\Omega_\pm;\mathbb{C}^2)$.
Thus the proof is concluded.

\begin{remark}
	The above spectral analysis of $\ccD^\pm$ is reminiscent of the spectral analysis of three-dimensional Dirac operators with zig-zag boundary conditions on general open sets performed in~\cite{H20}, which has appeared while the present paper was under preparation.
\end{remark}

\section{Approximation of $\delta$--shell interactions by regular
  potentials}
\label{sec:approximation}
In this section we 
prove Theorem~\ref{thm:approximation} on approximation of the Dirac operator with $\delta$-shell interaction by a sequence of Dirac operators with regular scaled potentials. 
\begin{proof}[Proof of Theorem~\ref{thm:approximation}]
  For all $0<\epsilon<\beta$ we define the self-adjoint operators
  $\ccE_{\eta,\tau,\lambda;\epsilon}$  according to
  \eqref{eq:defn.Eepsilon}
and we define the operators
$\ccD_{\hat{\eta,}\hat{\tau},\hat{\lambda}}$ according to
\eqref{eq:defn.D.shell.domain}, \eqref{eq:defn.D.shell}.
Since $\mathfrak{C}(\hat{\eta},\hat{\tau},\hat{\lambda}) \neq 0$ everywhere on $\Sigma$,
\Cref{thm:D.shell.introduction} tells us that these operators are
self-adjoint and $\dom \ccD_{\hat{\eta},\hat{\tau},\hat{\lambda}} \subset H^1(\R^2 \setminus
  \Sigma;\C^2)$.
Thanks to \cite[Theorem VIII.26]{reedsimon1}, since the limiting operators and the limit operator are self-adjoint,
 the family
  $\{\ccE_{\eta,\tau,\lambda;\epsilon}\}_{\epsilon\in(0,\beta)}$ converges in the
  strong resolvent sense to $\ccD_{\hat{\eta,}\hat{\tau},\hat{\lambda}}$ as $\epsilon\to 0$ if and only if it converges in the strong
  graph limit sense. The latter means that,   for all $\psi \in \dom
  \ccD_{\hat{\eta},\hat{\tau},\hat{\lambda}}$, there exists a family of vectors
  $\{\psi_{\epsilon}\}_{\epsilon\in(0,\beta)} \subset
   \dom \ccE_{\eta,\tau,\lambda;\epsilon} =
   H^1(\R^2;\C^2)$ such that
  \begin{equation}\label{eq:strong.convergence}
    \lim_{\epsilon\to 0}\psi_{\epsilon} = \psi
    \quad \text{ and } \quad
    \lim_{\epsilon\to 0}\ccE_{\eta,\tau,\lambda;\epsilon} \psi_{\epsilon}
    =
    \ccD_{\hat{\eta,}\hat{\tau},\hat{\lambda}} \psi
    \quad
    \text{ in }L^2(\R^2;\C^2).
  \end{equation}  
Without loss of generality we can assume $m =0$, because, by its very definition, the strong
graph convergence \eqref{eq:strong.convergence} is stable with respect to bounded symmetric perturbations.

Let $\psi\equiv\psi_+\oplus\psi_- \in \dom \ccD_{\hat{\eta,}\hat{\tau},\hat{\lambda}}$.
From \eqref{eq:statement.going.d>0} -- 
      \eqref{eq:statement.going.d<0},  we observe that $\hat{d}:=\hat{\eta}^2 - \hat{\tau}^2 -
  \hat{\lambda}^2> -4$. Therefore, by Lemma  \ref{lem:confinement}~(i),
  \begin{equation}\label{eq:equality_traces_1}
          \mathcal{T}_+^D \psi_+ =
          R_{\hat{\eta,}\hat{\tau},\hat{\lambda}} \mathcal{T}_-^D
          \psi_-,
 \end{equation}
 where 
 \begin{equation}
     R_{\hat{\eta,}\hat{\tau},\hat{\lambda}}(x_\Sigma)
      =
      \frac{4}{4+\hat{d}}\left(\frac{4-\hat{d}}{4}\,\bI_2+i\hat{\eta}(\sigma\cdot\mathbf{n})+\hat{\tau}
        (\sigma\cdot\mathbf{t}) -\hat{\lambda}\sigma_3 \right)(x_\Sigma).
 \end{equation}
Clearly, $R_{\hat{\eta,}\hat{\tau},\hat{\lambda}} \in C^{\infty}(\Sigma;\C^{2\times 2})$. Moreover, by Proposition~\ref{prop:trace.H12} combined with the fact that $\dom \ccD_{\hat{\eta,}\hat{\tau},\hat{\lambda}}\subset H^1(\R^2\setminus\Sigma;\C^2)$, see Theorem \ref{thm:self-adjointess.noncritical}, $\mathcal{T}_\pm^D \psi_\pm \in
H^{\frac12}(\Sigma;\C^2)$.

Recall that $\ccE_{\eta,\tau,\lambda;\epsilon}=\ccD_0+\bV_{\eta,\tau,\lambda;\epsilon}$, where, for all $x$ in $\Sigma_\epsilon$, \ie in the $\epsilon$--tubular neighborhood of $\Sigma$,  
\begin{equation*}
\bV_{\eta,\tau,\lambda;\epsilon}=  B_{\eta,\tau,\lambda}(x_\Sigma) h_\epsilon(p),
\end{equation*}
where $B_{\eta,\tau,\lambda}$ was introduced in  \eqref{eq:defb.B}, and $\bV_{\eta,\tau,\lambda;\epsilon}=0$ everywhere else, see \eqref{eq:defn.bV}. According~to \Cref{lem:exp.isnB}, 
  \begin{equation}\label{eq:expA=R}
    \exp[i(\sigma\cdot \mathbf{n}(x_\Sigma))
    B_{\eta,\tau,\lambda}(x_\Sigma)]
    =
    R_{\hat{\eta,}\hat{\tau},\hat{\lambda}}(x_\Sigma),
    \quad \text{ for all }x_\Sigma \in \Sigma,
  \end{equation}
Since, by definition, $\int_{-\epsilon}^{\epsilon} h_{\epsilon}(t)\,dt = 1$, we can rewrite \eqref{eq:equality_traces_1} as 
\begin{equation}\label{eq:equality.traces}
  \begin{split}
  & \exp\left[- i\left(\int_{-\epsilon}^0 h_\epsilon(t)\,dt\right)
    (\sigma\cdot \mathbf{n}) B_{\eta,\tau,\lambda} 
  \right]\mathcal{T}_+^D \psi_+
  \\ &\qquad  =
  \exp\left[i \left( \int_0^{\epsilon} h_\epsilon(t)\,dt \right) (\sigma\cdot \mathbf{n}) B_{\eta,\tau,\lambda}
  \right] \mathcal{T}_-^D \psi_-.
\end{split}
\end{equation}

Let us now construct the family $\{\psi_\epsilon\}_{\epsilon\in(0,\beta)}$.
For all $\epsilon\in(0,\beta)$, we put
\begin{equation*}
    H_{\epsilon}: \R\setminus\{0\} \to \R, 
  \quad H_{\epsilon}(p):=
  \begin{cases}
    \int_{p}^{\epsilon} h_{\epsilon}(t)\,dt \quad &0 <  p <\epsilon ,\\
    -\int_{-\epsilon}^p h_{\epsilon}(t)\,dt \quad &-\epsilon < p < 0,\\
    0 \quad &\abs{p}\geq \epsilon.
  \end{cases}
\end{equation*}
Note that $\supp\,H_{\epsilon} \subset (-\epsilon,\epsilon)$ and
$H_\epsilon \in L^\infty(\R)$. Since $\norm{H_\epsilon}_{L^{\infty}(\R)}\leq \norm{h}_{L^1(\R)}$, $\{H_\epsilon\}_{\epsilon}$ is bounded uniformly in $\epsilon$.
For all $\epsilon\in(0,\beta)$, the restrictions of $H_\epsilon$ to
$\R_\pm$  are uniformly continuous, so  finite limits at $p=0$ exist,
and differentiable a.e. with derivative being bounded, since
$h_\epsilon\in L^\infty(\R;\R)$. Furthermore, $H_\epsilon$ has a jump at the origin of
size $\int_{-\epsilon}^\epsilon h_\epsilon(t)\,dt =1$.
Next, we set
\begin{equation}\label{eq:defn.bUepsilon}
  \begin{split}
    &\bU_{\epsilon} : \R^2 \setminus \Sigma \to \C^{2\times 2},
    \\ 
    &  \bU_{\epsilon}(x):=
  \begin{cases}
    \exp[i(\sigma\cdot\mathbf{n})B_{\eta,\tau,\lambda}(\mathscr{P}_\Sigma(x))
    \, H_{\epsilon}(\mathscr{P}_\perp(x))]
    \quad & x \in \Sigma_\epsilon\setminus\Sigma,\\
    \bI_2 \quad & x \in \R^2\setminus \Sigma_\epsilon,
  \end{cases}
\end{split}
\end{equation}
where the mappings $\mathscr{P}_\Sigma$ and $\mathscr{P}_\perp$ are defined as in \eqref{eq:PSigma} and~\eqref{eq:Pperp}, respectively.
The matrix functions $\bU_{\epsilon}$ are bounded, uniformly in $\epsilon$, and  uniformly continuous in $\Omega_\pm$, with a jump discontinuity
across $\Sigma$: for all $x_\Sigma \in\Sigma$ we have
\begin{equation}\label{eq:bUepsilon+}
  \bU_\epsilon(x_\Sigma^+) :=\lim_{\substack{y \to x_\Sigma \\ y \in \Omega_+}} \bU_\epsilon(y)
  =
  \exp\left[-i\left(\int_{-\epsilon}^0 h_\epsilon(t)\,dt\right)(\sigma\cdot\mathbf{n}(x_\Sigma))B_{\eta,\tau,\lambda}(x_\Sigma)\right]
\end{equation}
and 
\begin{equation}\label{eq:bUepsilon-}
  \bU_\epsilon(x_\Sigma^-) :=\lim_{\substack{y \to x_\Sigma \\ y \in \Omega_-}} \bU_\epsilon(y)
  =
  \exp\left[ i\left(\int_0^{\epsilon} h_\epsilon(t)\,dt\right)(\sigma\cdot\mathbf{n}(x_\Sigma))B_{\eta,\tau,\lambda}(x_\Sigma)
  \right].
\end{equation}
Finally, we put 
\begin{equation}\label{eq:defn.psiepsilon}
  \psi_\epsilon = \psi_{\epsilon,+}\oplus \psi_{\epsilon,-} :=
  \bU_\epsilon \psi \in L^2(\R^2;\C^2).
\end{equation}

It is immediate to see,  by the dominated convergence theorem, that 
\begin{equation}\label{eq:conv.L2}
  \psi_\epsilon \xrightarrow[\epsilon \to 0]{} \psi \quad \text{ in }L^2(\R^2;\C^2)
\end{equation}
since $\psi_\epsilon - \psi = (\bU_\epsilon -\bI) \psi$,
$\bU_{\epsilon}\in L^{\infty}(\R^2;\C^{2\times 2})$
with a uniform bound
in $\epsilon\in(0,\beta)$, $\supp (\bU_\epsilon -
\bI) \subset \Sigma_\epsilon$ and $\abs{\Sigma_\epsilon}\to 0$ as
$\epsilon \to 0$. 

We show now that $\psi_\epsilon \in \dom
\ccE_{\eta,\tau,\lambda;\epsilon} = H^1 (\R^2 ;\C^2)$ for all
$\epsilon\in(0,\beta)$. To do so,  we verify that
$\psi_{\epsilon, \pm} \in H^1(\Omega_\pm; \C^2)$ and that
$\mathcal{T}_+^D \psi_{\epsilon,+} = \mathcal{T}_-^D \psi_{\epsilon,-}
\in H^{\frac{1}{2}}(\Sigma;\C^2)$.
Let $\gamma : \R \big/ \ell \Z \to \Sigma
\subset \R^2$ be, as always,  a smooth arc-length parametrization of $\Sigma$ with positive
orientation and 
\begin{equation*}
     A \in C^{\infty}(\R/\ell\Z ; \C^{2\times 2}), \quad
     A(s) := i(\sigma\cdot \mathbf{n}(\gamma(s)))
     B_{\eta,\tau,\lambda}(\gamma(s)).
   \end{equation*}
Thus, we may write \eqref{eq:defn.bUepsilon} as
\begin{equation}\label{eq:defn.bUepsilon.2}
      \bU_{\epsilon}(x)=
  \begin{cases}
    \exp[A(\mathscr{P}_\gamma(x))H_{\epsilon}(\mathscr{P}_\perp(x))]
    \quad & x \in \Sigma_\epsilon\setminus\Sigma,\\
    \bI_2 \quad & x \in \R^2\setminus \Sigma_\epsilon,
  \end{cases}
\end{equation}
where $\mathscr{P}_\gamma$ is defined as in~\eqref{eq:Pgamma}.
For $j=1,2$, $\supp\, \partial_j \bU_\epsilon \subset\Sigma_\epsilon$ and,
thanks to the Wilcox formula, \cf \cite[eq.~(4.1)]{wilcox1967exponential}, for $x\in
\Sigma_\epsilon\setminus\Sigma$ we have
\begin{equation*}  
  \begin{split}
  \partial_j \bU_{\epsilon}(x) =
  \int_0^1 & \Big[
    e^{z A(\mathscr{P}_\gamma(x))H_{\epsilon}(\mathscr{P}_\perp(x))}
    \partial_j
    \big[A(\mathscr{P}_\gamma(x))H_{\epsilon}(\mathscr{P}_\perp(x))
    \big] \\
    & \quad
    e^{(1-z) A(\mathscr{P}_\gamma(x))H_{\epsilon}(\mathscr{P}_\perp(x))}
    \Big]\,dz.
  \end{split}
\end{equation*}
Recall that we have set $s= \mathscr{P}_\gamma(x)$
and $p=\mathscr{P}_\perp(x)$. Using \eqref{eq:L_grad}, we obtain
\begin{equation*}  
    \partial_j
    \big[A(\mathscr{P}_\gamma(x))H_{\epsilon}(\mathscr{P}_\perp(x))
    \big] 
     = 
    \partial_s A(s)
    \frac{(\mathbf{t}_\gamma(s))_j}{1+p\kappa_\gamma(s)}
    H_\epsilon(p)
    -
    A(s) h_\epsilon(p)(\mathbf{n}_\gamma(s))_j.
\end{equation*}
Therefore, we arrive at
\begin{equation}\label{eq:derivative.Uepsilon}
  \begin{split}
    \partial_j \bU_\epsilon(x)
    = &  -
    A(s)h_\epsilon(p)
    (\mathbf{n}_\gamma(s))_j \bU_\epsilon(x)
    \\ & + H_\epsilon(p) \int_0^1
    e^{z A(s)H_{\epsilon}(p)}
    \partial_s A(s)
    \frac{(\mathbf{t}_\gamma(s))_j}{1+p\kappa_\gamma(s)}
    \cdot e^{(1-z) A(s)H_{\epsilon}(p)}\,dz
    \\ = &
    - i(\sigma\cdot\mathbf{n}(x)) \bV_{\eta,\tau,\lambda;\epsilon}(x) (\mathbf{n}_\gamma(s))_j \bU_\epsilon(x)
    + R_{j;\epsilon}(x),
  \end{split}
\end{equation}
where 
\begin{equation*}
  R_{j;\epsilon}(x):=
  H_\epsilon(p) \int_0^1
  e^{z A(s)H_{\epsilon}(p)}
  \partial_s A(s)
  \frac{(\mathbf{t}_\gamma(s))_j}{1+p\kappa_\gamma(s)}
  \cdot e^{(1-z) A(s)H_{\epsilon}(p)}\,dz.
\end{equation*}
The matrix-valued functions  $R_{j;\epsilon}$ are bounded, uniformly in $\epsilon\in(0,\beta)$, and $\supp R_{j;\epsilon} \subset \Sigma_\epsilon$.
We observe that
\begin{equation} \label{eq:U_bounds}
\bU_\epsilon, \partial_1 \bU_\epsilon, \partial_2 \bU_\epsilon \in
 L^\infty(\Omega_\pm ;\C^{2\times 2}).
\end{equation}
Since $\psi_\pm \in H^1(\Omega_\pm;\C^2)$,
we conclude that $\psi_{\epsilon,\pm}=\bU_\epsilon \psi_\pm \in H^1(\Omega_\pm;\C^2)$.

Thanks to \Cref{prop:trace.H12},
$\mathcal{T}_+^D \psi_{\epsilon,\pm}\in H^{\frac{1}{2}}(\Sigma;\C^2)$
and, thanks to \cite[Chapter 4]{evans2015measure}, for a.e.~$x_\Sigma\in \Sigma$,
\begin{equation*}
  \begin{split}
    \mathcal{T}_\pm^D \psi_{\epsilon,\pm}(x_{\Sigma})
    & =
  \lim_{r\to 0} \frac{1}{|B_r(x_{\Sigma})|}
  \int_{\Omega_\pm\cap B_r(x_{\Sigma})}  \psi_{\epsilon}(y)\,dy\\
    & = \lim_{r\to 0} \frac{1}{|B_r(x_{\Sigma})|}
    \int_{\Omega_\pm\cap B_r(x_{\Sigma})}  \bU_\epsilon(y)\psi(y)\,dy;
  \end{split}
\end{equation*}
similarly, we have
\begin{equation*}
 \bU_\epsilon(x_{\Sigma}^\pm)  \mathcal{T}_\pm^D \psi_{\pm}(x_{\Sigma}) =
  \lim_{r\to 0} \frac{1}{|B_r(x_{\Sigma})|}
  \int_{\Omega_\pm\cap B_r(x_{\Sigma})} \bU_\epsilon(x_{\Sigma}^\pm)  \psi(y)\,dy.
\end{equation*}
Since $\bU_\epsilon$ is continuous on $\overline{\Omega_+}$ and $\overline{\Omega_-}$, respectively,  we get 
\begin{equation*}
 \mathcal{T}_\pm^D \psi_{\epsilon,\pm}(x_{\Sigma})= \bU_\epsilon(x_{\Sigma}^\pm)   \mathcal{T}_\pm^D \psi_{\pm}(x_{\Sigma}).
\end{equation*}
Taking \eqref{eq:equality.traces}, \eqref{eq:bUepsilon+}, and
\eqref{eq:bUepsilon-} into account, this yields
$  \mathcal{T}_+^D  \psi_{\epsilon,+}  =   \mathcal{T}_-^D
\psi_{\epsilon,-} \in H^{\frac{1}{2}}(\Sigma;\C^2)$,  and so we conclude
that $\psi_\epsilon \in H^1(\R^2;\C^2)$ for all $\epsilon\in(0,\beta)$.

To finish the proof, it remains to show that $\lim_{\epsilon\to 0}\ccE_{\eta,\tau,\lambda;\epsilon}
\psi_{\epsilon}=\ccD_{\eta,\tau,\lambda}\psi$. We have
\begin{equation}\label{eq:approx.gathering.1}
  \begin{split}
  \ccE_{\eta,\tau,\lambda;\epsilon} \psi_{\epsilon} -
  \ccD_{\eta,\tau,\lambda}\psi 
  = &
  -i \sigma\cdot \nabla (\bU_\epsilon \psi)
  + \bV_{\eta,\tau,\lambda;\epsilon}  \psi_\epsilon
  +   i \sigma\cdot \nabla \psi.
  \\ = &
  -i \sum_{j=1}^2 \sigma_j 
  [(\partial_j \bU_\epsilon) \psi
  + (\bU_\epsilon - \bI_2)\partial_j \psi ]
  + \bV_{\eta,\tau,\lambda;\epsilon}  \psi_\epsilon.
\end{split}
\end{equation}
Applying \eqref{eq:derivative.Uepsilon} together with \eqref{eq:Pauli.squares}, we get 
\begin{equation}\label{eq:approx.gathering.2}
  \begin{split}
  -i \sum_{j=1}^2 \sigma_j 
  (\partial_j \bU_\epsilon) \psi
  = &
  -i \sum_{j=1}^2 \sigma_j 
  [ -i (\sigma\cdot\mathbf{n}) \bV_{\eta,\tau,\lambda;\epsilon}
   \,n_j\, \bU_\epsilon   \psi
   + R_{j;\epsilon}\psi ]
   \\ = &
   - (\sigma\cdot\mathbf{n})
   (\sigma\cdot\mathbf{n})
   \bV_{\eta,\tau,\lambda;\epsilon}
   \bU_\epsilon   \psi
   - i \sum_{j=1}^2 \sigma_j
    R_{j;\epsilon}\psi 
   \\ = & 
   - \bV_{\eta,\tau,\lambda;\epsilon}\psi_\epsilon
   + Q_{\epsilon}\psi,
 \end{split}
\end{equation}
 where $Q_{\epsilon} \in L^{\infty}(\R^2;\C^{2\times 2})$, with the $L^\infty$-norm uniformly bounded in $\epsilon\in(0,\beta)$, and $\supp\, Q_{\epsilon} \subset \Sigma_\epsilon$.
According to \eqref{eq:approx.gathering.1} and \eqref{eq:approx.gathering.2}, we get
\begin{equation}\label{eq:strong.convergence.2}
  \ccE_{\eta,\tau,\lambda;\epsilon} \psi_{\epsilon} -
  \ccD_{\eta,\tau,\lambda}\psi 
  =
    -i \sum_{j=1}^2 \sigma_j 
  [(\bU_\epsilon - \bI_2)\partial_j \psi ]
  + Q_{\epsilon}\psi \xrightarrow[\epsilon \to 0]{} 0
  \quad \text{ in }L^2(\R^2;\C^2)
 \end{equation}
 by the dominated convergence, since $\psi\in H^1(\R^2;\C^2)$,
 $\bU_\epsilon-\bI_2$ and $Q_\epsilon$ are uniformly bounded in $\epsilon\in(0,\beta)$ and supported on $\Sigma_\epsilon$, and $\lim_{\epsilon\to 0}\abs{\Sigma_\epsilon} =0$.
 Putting \eqref{eq:strong.convergence.2} and \eqref{eq:conv.L2} together we obtain~\eqref{eq:strong.convergence}.
\end{proof}

The problem of finding regular approximations for $\ccD_{\eta,\tau,\lambda,\omega}$ with $\eta,\tau,\lambda,\omega \in \R$ reduces to the problem of finding the approximations when $\omega=0$. Indeed, according to Theorem \ref{thm:w.gauged}, there exist $X\in\R\setminus\{0\}$ and a unitary operator $U_z$, where $z\in\C:\, |z|=1$ is a parameter that may be calculated in terms of $\eta,\tau,\lambda,\omega$, such that
$U_{\bar z}\ccD_{\eta,\tau,\lambda,\omega} U_z=\ccD_{X\eta,X\tau,X\lambda}$. We will assume that $X^2d>-4$, because if $X^2 d<-4$ then, employing Theorem \ref{thm:w.gauged} again, we can sandwich $\ccD_{X\eta,X\tau,X\lambda}$ by another unitary transform to get $\ccD_{\tilde\eta,\tilde\tau,\tilde\lambda}$ such that $\tilde d=\tilde\eta^2-\tilde\tau^2-\tilde\lambda^2>-4$. Now, using Corollary \ref{thm:approximation.back}, we find a family of approximating operators $\ccE_{\eta',\tau',\lambda';\epsilon}$ such that $\ccE_{\eta',\tau',\lambda';\epsilon}\to\ccD_{X\eta,X\tau,X\lambda}$ in the strong resolvent sense as $\epsilon\to 0$. 
If, for $a\in\R$, we define the unitary multiplication operator
  \begin{equation*}
    W_{a;\epsilon} :=
    \begin{cases}
      \bI_2 \quad &\text{ in }\Omega_+ \setminus \Sigma_\epsilon, \\
      \exp\left[ia \int_{-\epsilon}^{\mathscr{P}_\perp(\cdot)}
        h_\epsilon(t)\,dt \right]\bI_2 \quad &\text{ in }\Sigma_\epsilon, \\
      e^{ia} \bI_2 \quad & \text{ in }\Omega_- \setminus \Sigma_\epsilon.
    \end{cases}
  \end{equation*}
then, with the help of  \eqref{eq:L_grad}, we get
  \begin{equation*}
    W_{a;\epsilon}^* \,  \ccE_{\eta,\tau,\lambda;\epsilon} \,W_{a;\epsilon}
    =
    \ccE_{\eta,\tau,\lambda;\epsilon} + a(\sigma\cdot\mathbf{n}) \,  \chi_{\Sigma_\epsilon} \, h_\epsilon,
  \end{equation*}
 where $\chi_{\Sigma_\epsilon}$ is the indicator function of $\Sigma_\epsilon$.  Note that $\lim_{\epsilon\to 0} W_{-\arg z;\epsilon}=U_{\bar z}$ in the strong operator topology. Recalling that $U_z^{-1}=U_{\bar z}=U_z^*$, we conclude that
 \begin{equation*}
   \begin{split}
  \ccE_{\eta',\tau',\lambda';\epsilon} -\arg z(\sigma\cdot\mathbf{n})
  \,  \chi_{\Sigma_\epsilon} \,  h_\epsilon
  & =
  W_{-\arg z;\epsilon}^* \, \ccE_{\eta',\tau',\lambda';\epsilon}\,
  W_{-\arg z;\epsilon}
  \\
  &   \to U_z\,\ccD_{X\eta,X\tau,X\lambda}\, U_{\bar
    z}=\ccD_{\eta,\tau,\lambda,\omega},
\end{split}
\end{equation*}
in the strong resolvent sense as $\epsilon\to 0$.

\begin{proof}[{{Proof of \Cref{thm:approximation.back}}}] 
The proof is immediate from \Cref{thm:approximation} and
\Cref{lem:going.back.home}.
\end{proof}

\subsection{Alternative approximations for purely magnetic interaction}\label{sec:alternative_approximation}
If $\hat\eta=\hat\tau=0$, and $\hat\lambda\in\R\setminus\{\pm 2\}$ then 
$R_{0,0,\hat\lambda}=\diag\left(\frac{2-\hat\lambda}{2+\hat\lambda},\frac{2+\hat\lambda}{2-\hat\lambda}\right)$
 is constant along $\Sigma$. This makes it possible to construct an alternative sequence of approximations without employing ``parallel coordinates'' $(s,p)$.  The strategy will be to apply the method of \cite{hughes97,hughes99}, that works for any type of one-dimensional $\delta$-interaction. We will restrict ourselves to the case $\hat\lambda\in(-2,2)$, the remaining cases including their approximations may be recovered using unitary equivalences, \cf \Cref{rem:unitary_eq}. Let us start with introducing a bounded operator $\mathcal{W}:=\chi_{\Omega_+}\bI_2+\chi_{\Omega_-}R_{0,0,\hat\lambda}$ in $L^2(\R^2;\C^2)$.

\begin{lemma} \label{lem:D_alt}
$\mathscr{D}_{0,0,\hat\lambda}=\mathcal{W}(-i\sigma\cdot\nabla)
\mathcal{W}$, where the operator at right-hand side is defined on $\{\psi\in L^2(\R^2;\C^2)|\, \mathcal{W}\psi\in H^1(\R^2;\C^2)\}$.
\end{lemma}
\begin{proof}
Since $-i\sigma\cdot\nabla$ is self-adjoint on $H^1(\R^2;\C^2)$, $\mathcal{W}^*=\mathcal{W}$, and $\mathcal{W}$ together with $\mathcal{W}^{-1}$ are bounded, $\mathcal{W}(-i\sigma\cdot\nabla) \mathcal{W}$ is also self-adjoint. Therefore, it is sufficient to show that $\mathscr{D}_{0,0,\hat\lambda}\subset \mathcal{W}(-i\sigma\cdot\nabla) \mathcal{W}$. Take $\psi\equiv\psi_+\oplus\psi_- \in\dom(\mathscr{D}_{0,0,\hat\lambda})\subset H^1(\Omega_+;\C^2)\oplus H^1(\Omega_-;\C^2)$. Then $\mathcal{T}_{+}^D(\mathcal{W}\psi)_+=\mathcal{T}_{+}^D\psi_+$ and, by \eqref{eq:equality_traces_1}, $\mathcal{T}_{-}^D(\mathcal{W}\psi)_-=R_{0,0,\hat\lambda}\mathcal{T}_{-}^D\psi_-=\mathcal{T}_{+}^D\psi_+$. Hence, $\mathcal{W}\psi\in H^1(\R^2;\C^2)$, \ie,  $\psi\in\dom(\mathcal{W}(-i\sigma\cdot\nabla) \mathcal{W})$. Finally, using the fact that for $j=1,2,$ $R_{0,0,\hat\lambda}\sigma_j R_{0,0,\hat\lambda}=\sigma_j$, we get
\begin{equation*}
  \begin{split}
\mathcal{W}\sigma\cdot\nabla
\mathcal{W}\psi & =\chi_{\Omega_+}\sigma\cdot\nabla\psi_++\chi_{\Omega_-}R_{0,0,\hat\lambda}\sigma
R_{0,0,\hat\lambda}\cdot\nabla\psi_- \\
& =\sigma\cdot\nabla\psi_+\oplus\sigma\cdot\nabla\psi_-=i
\mathscr{D}_{0,0,\hat\lambda}\psi.
\qedhere
\end{split}
\end{equation*}
\end{proof}

Next, let $(g_\epsilon)_{\epsilon>0}$ be the standard two-dimensional mollifiers, \ie, 
\begin{equation*} 
g_\epsilon(x):=\frac{1}{\epsilon^{2}}g\left(\frac{x}{\epsilon}\right) \text{ with } g\in C^\infty(\R^2;[0,+\infty)):\,\supp(g)\subset\overline{B(0,1)}\text{ and }\int_{B(0,1)}g=1.
\end{equation*} 
Note that we may write $\mathcal{W}=\exp(-\lambda\chi_{\Omega_-}\sigma_3)$ with $\lambda:=2\arctanh\frac{\hat\lambda}{2}$, because $R_{0,0,\hat\lambda}=\exp(-\lambda\sigma_3)$. This suggests to introduce $\mathcal{W}_\epsilon:=\exp(-\lambda\chi_{\Omega_-}^\epsilon\sigma_3)$, where $\chi_{\Omega_-}^\epsilon:=g_\epsilon*\chi_{\Omega_-}.$ Then we have

\begin{proposition}
Let $\hat\lambda\in(-2,2)$ be constant and $\lambda=2\arctanh\frac{\hat\lambda}{2}$. Then
\begin{equation*} 
\mathscr{D}_0+\lambda(\sigma_2, -\sigma_1)\cdot \nabla\chi_{\Omega_-}^\epsilon\xrightarrow[\epsilon \to 0]{} \mathscr{D}_{0,0,\hat\lambda}
\end{equation*}
in the strong resolvent sense. 
\end{proposition}
\begin{proof}
First, using Lemma \ref{lem:D_alt}, one shows that the self-adjoint operator $\mathscr{D}_\lambda^\epsilon:=\mathcal{W}_\epsilon(-i\sigma\cdot\nabla)\mathcal{W}_\epsilon$  defined on $\{\psi\in L^2(\R^2;\C^2)|\, \mathcal{W}_\epsilon\psi\in H^1(\R^2;\C^2)\}$ converges to $\mathscr{D}_{0,0,\hat\lambda}$ in the strong graph limit sense as $\epsilon\to 0$; for details see the proof of \cite[Theorem 2]{hughes97}. This implies also the strong resolvent convergence. 
Since
$\dom(\mathscr{D}_\lambda^\epsilon)=\mathcal{W}_\epsilon^{-1}H^1(\R^2;\C^2)$
and both $\mathcal{W}_\epsilon$ and $\mathcal{W}_\epsilon^{-1}$,
viewed as matrix-valued functions, are smooth and bounded (including
their derivatives), $\dom \mathscr{D}_\lambda^\epsilon =H^1(\R^2;\C^2)$. Next, for any $\psi\in\dom(\mathscr{D}_\lambda^\epsilon)$, we have
\begin{multline*}
\mathscr{D}_\lambda^\epsilon\psi=\mathcal{W}_\epsilon(-i\sigma\cdot\nabla)\mathcal{W}_\epsilon\psi=-i\mathcal{W}_\epsilon\sigma \mathcal{W}_\epsilon\cdot\nabla\psi+i\lambda\mathcal{W}_\epsilon\sigma \mathcal{W}_\epsilon\cdot\nabla\chi_{\Omega_-}^\epsilon \sigma_3\psi\\
=-i\sigma\cdot\nabla\psi+i\lambda\sigma\cdot\nabla\chi_{\Omega_-}^\epsilon \sigma_3\psi,
\end{multline*}
where we used the observation that $\mathcal{W}_\epsilon\sigma \mathcal{W}_\epsilon=\sigma$ in the last equality. Therefore, $\mathscr{D}_\lambda^\epsilon=\mathscr{D}_0+\lambda(\sigma_2, -\sigma_1)\cdot \nabla\chi_{\Omega_-}^\epsilon$.
\end{proof}

\section{Final remark: higher dimensions}
\label{sec:higher.dimensions}
We conclude the paper with a discussion on a possible generalization of our results to higher dimensional cases.

  It is possible to define an analogue of the magnetic interaction in
  higher dimensions. This is not immediate, because the tangent unit
  vector is not uniquely defined. However,  since $\sigma\cdot \mathbf{t} =i(\sigma\cdot\mathbf{n})\sigma_3$, see \eqref{eq:sigma.t}, this issue may be overcome. We rewrite the formal expression for \eqref{eq:Dirac.shells} as follows
\begin{equation*}
    D_{\eta,\tau,\lambda,\omega} =
    D_0
    + (\eta \bI_2 + \tau \sigma_3
    + \lambda \, i(\sigma\cdot \mathbf{n})\sigma_3
    + \omega (\sigma\cdot \mathbf{n}))\delta_\Sigma.
 \end{equation*}

 Put $N := 2^{\lfloor \frac{n+1}{2}\rfloor}$, where
$\lfloor \cdot \rfloor$ denotes the integer part of a real number.
It is well known (see, \eg~\cite{friedrich2000dirac,jost2008riemannian})
that there exist Hermitian matrices
$\alpha_1,\dots,\alpha_n, \alpha_{n+1} \in \C^{N\times N}$
that satisfy the anticommutation relations
\begin{equation*} 
  \alpha_j \alpha_k +
  \alpha_k \alpha_j =
  2 \delta_{j,k} \mathbb{I}_N, 
  \quad 
  1\leq j,k \leq n+1,
\end{equation*}
where $\delta_{j,k}$ stands for the Kronecker delta.
The Dirac differential expression with a
$\delta$--shell interaction in $\R^n$ acts on functions $\psi : \R^n \to \C^{N}$  as follows:
\begin{equation*}
  D_{\eta,\tau,\lambda,\omega}^{[n]}  :=
  -i \alpha \cdot \nabla + (\eta \bI_N + \tau\alpha_{n+1} 
  + \lambda i (\alpha \cdot \mathbf{n})\alpha_{n+1}
  + \omega (\alpha\cdot\mathbf{n}))\delta_\Sigma,
\end{equation*}
where $\alpha \cdot \nabla:=\sum_{j=1}^n \alpha_j \partial_j$.
 In particular, we define the Dirac differential expression in $\R^3$ with a $\delta$--shell
 interaction as follows: denoting $\beta = \alpha_4$, 
 \begin{equation*}
   D^{[3]}_{\eta,\tau,\lambda,\omega} =
   -i\alpha\cdot\nabla +m\beta + (\eta \bI_4 + \tau \beta
  + \lambda i (\alpha \cdot \mathbf{n})\beta
  + \omega (\alpha\cdot\mathbf{n}))\delta_\Sigma.
\end{equation*}
Adopting the terminology used for the potentials, we will call the
interaction $\lambda i (\alpha\cdot\mathbf{n}) \beta \,\delta_\Sigma$ the
\emph{anomalous-magnetic} $\delta$--shell interaction.

We point out that when we were finishing this work we learnt that the
three dimensional case was being considered in \cite{benhellal2021}. In there, the
author introduces the $\delta$--shell interaction corresponding to the
differential expression $D_0^{[3]}+(\xi\gamma_5+i\lambda(\alpha\cdot
\mathbf{n})\beta)\delta_\Sigma$, for $\xi,\lambda \in \R$
and  $\gamma_5 := -i\alpha_1\alpha_2\alpha_3$.
Using the strategy developed in \cite{amv1}, that is based
on fundamental solutions, in \cite[Section 6]{benhellal2021} the author shows some
results which, in the case $\xi=\eta=\tau=0$, agree with our Theorems
\ref{thm:D.shell.introduction}, \ref{thm:confinement.introduction},
and (the statements about self-adjointness in) \ref{thm:critical}. It is
worth mentioning that his approach also works on surfaces $\Sigma$
with low regularity. In this direction, see also
\cite{rabinovich2020boundary}, where general local interactions are
considered, although no explicit reference to the anomalous magnetic
potential is made. 

\appendix
\section{Lemmata on exponential matrices}
\label{sec:Appendix}
Let $\mathbf{t}=(t_1,t_2)$ be a unit vector and $\mathbf{n}:=(t_2,-t_1)$.
Using the shorthand notation \eqref{eq:notation}, for $\eta,\tau,\lambda \in \R$ let us consider the Hermitian
matrix
\begin{equation}\label{defn:B}
  B_{\eta,\tau,\lambda}:= \eta \bI_2 + \tau \sigma_3
  + \lambda (\sigma\cdot\mathbf{t}).
\end{equation}
For $\hat{\eta},\hat{\tau},\hat{\lambda} \in \R$ such that $\hat{d}:=\hat{\eta}^2 -\hat{\tau}^2
-\hat{\lambda}^2 \neq -4$, put 
\begin{equation}\label{def_R_eta_appendix}
  R_{\hat{\eta},\hat{\tau},\hat{\lambda}} :=  \frac{4}{4+\hat{d}}\left(\frac{4-\hat{d}}{4}\,\bI_2+i\hat{\eta}(\sigma\cdot\mathbf{n})+\hat{\tau}
    (\sigma\cdot\mathbf{t}) -\hat{\lambda}\sigma_3 \right).
\end{equation}
In this appendix we address the following questions:
\begin{enumerate}[label=$({\roman*})$]
\item\label{item:existence}  Given $\eta,\tau,\lambda \in \R$, is it possible to find $\hat{\eta},\hat{\tau},\hat{\lambda} \in \R$ such that $\hat d\neq -4$ and
  \begin{equation*}
  \exp[i(\sigma\cdot\mathbf{n})B_{\eta,\tau,\lambda}] =
  R_{\hat{\eta},\hat{\tau},\hat{\lambda}} \, ?
\end{equation*}

\item\label{item:uniqueness} Is this correspondence bijective?
\end{enumerate}
Similar questions were already considered in \cite[Appendix]{tusek19}.

The following lemma gives an answer to question
\ref{item:existence}.
\begin{lemma}\label{lem:exp.isnB}
  Let $\eta,\tau,\lambda \in \R$, $d := \eta^2 - \tau^2 -\lambda^2$,  and 
  $B_{\eta,\tau,\lambda}$ and $R_{\hat{\eta},\hat{\tau},\hat{\lambda}}$ be given by \eqref{defn:B} and \eqref{def_R_eta_appendix}, respectively.
  Let $\hat{\eta},\hat{\tau},\hat{\lambda} \in \R$ be such that  $\hat{d}=
  \hat{\eta}^2 - \hat{\tau}^2 - \hat{\lambda}^2\neq -4$.
  Then $ R_{\hat{\eta},\hat{\tau},\hat{\lambda}} = \exp
  [i(\sigma\cdot \mathbf{n}) B_{\eta,\tau,\lambda}]$ if and only if one of the following holds
  \begin{equation}  \label{eq:going.d>0}
    \left. \begin{aligned} 
     &\bullet\, d>0,\, d \neq (2k+1)^2 \pi^2 \text{ for all }  k \in \N_0,\text{ and }   		      (\hat{\eta},\hat{\tau},\hat{\lambda}) = \dfrac{\tan(\sqrt{d}/2)}{\sqrt{d}/2}      (\eta,\tau,\lambda)
      \\
     &\bullet\, d=0\text{ and }(\hat{\eta},\hat{\tau},\hat{\lambda}) = (\eta,\tau,\lambda)
      \\
     &\bullet\, d<0\text{ and }(\hat{\eta},\hat{\tau},\hat{\lambda})
     = \dfrac{\tanh(\sqrt{-d}/2)}{\sqrt{-d}/2} (\eta,\tau,\lambda).
    \end{aligned} \right\}
  \end{equation} 
  If $d = (2k_0+1)^2\pi^2$, for $k_0 \in \N_0$, there are no
  $\hat{\eta},\hat{\tau},\hat{\lambda} \in \R$  such that
  $ R_{\hat{\eta},\hat{\tau},\hat{\lambda}} = \exp [i(\sigma\cdot
  \mathbf{n}) B_{\eta,\tau,\lambda}]$.
\end{lemma}
\begin{proof}
  Denoting
  $T = t_1 + i t_2$ for $\mathbf{t} = (t_1,t_2)$, we get
  \begin{equation}\label{eq:isn.B}
    i(\sigma\cdot \mathbf{n}) B_{\eta,\tau,\lambda} =
    \begin{pmatrix}
      - \lambda & (\tau - \eta) \overline{T} \\
      (\tau + \eta)T &   \lambda
    \end{pmatrix}.
  \end{equation}
   The exponential of a general $2 \times 2$ matrix $A$ is
\begin{equation}\label{eq:exp.A}
    \exp[A] =
    \exp \left[\frac{\text{Tr}A}{2}\right]
    \left( \cos \nu \, \bI_2 + \frac{\sin \nu}{\nu}\left(A -
        \frac{\text{Tr}A}{2}\bI_2\right)\right),
    \end{equation}
  with
  \begin{equation*}
  \nu = \sqrt{\det A - \left(\frac{\text{Tr}A}{2}\right)^2} \in \C,
\end{equation*}
considering the principal branch of the square root, and
where $\sin\nu/\nu$ is intended to be equal to $1$ when $\nu =0$, see, \eg \cite{bernstein1993some}, \cite[Appendix]{tusek19}.
Plugging \eqref{eq:isn.B} into \eqref{eq:exp.A} we get
\begin{equation*}
  \begin{split}
  \exp[i(\sigma\cdot\mathbf{n})B_{\eta,\tau,\lambda}] = & 
   \cos \sqrt{d} \, \bI_2 + \frac{\sin \sqrt{d}}{\sqrt{d}}
     i (\sigma\cdot\mathbf{n}) B_{\eta,\tau,\lambda}
  \\
  = &
  \cos \sqrt{d} \, \bI_2
  + \frac{\sin \sqrt{d}}{\sqrt{d}}
  \left(
    i(\sigma\cdot\mathbf{n}) \eta
  + \tau (\sigma\cdot \mathbf{t})
  -  \lambda \sigma_3\right).
\end{split} 
\end{equation*}

Since the matrices $\{\bI_2, \sigma_3,
\sigma\cdot\mathbf{t}, \sigma\cdot\mathbf{n}\}$ are a basis of the
Hermitian $2\times 2$ matrices,  we have 
$ R_{\hat{\eta},\hat{\tau},\hat{\lambda}} = \exp[i(\sigma\cdot\mathbf{n})B_{\eta,\tau,\lambda}]$ if and only if the
coefficients with respect to this basis are equal, \ie
\begin{align}
    \cos \sqrt{d} &= \frac{4-\hat{d}}{4+\hat{d}}  \label{eq:conditions.equality.2a}\\
     \frac{\sin \sqrt{d}}{\sqrt{d}} (\eta,\tau,\lambda)
    &= \frac{4}{4+\hat{d}} (\hat{\eta},\hat{\tau},\hat{\lambda}) 			       \label{eq:conditions.equality.2b}.
\end{align}
For $d = (2k_0+1)^2\pi^2$ with $k_0 \in \N_0$, \eqref{eq:conditions.equality.2a} has no solution $\hat{d}\in
\R\setminus\{-4\}$. Consequently, there are no
  $\hat{\eta},\hat{\tau},\hat{\lambda} \in \R$  such that
  $ R_{\hat{\eta},\hat{\tau},\hat{\lambda}} = \exp [i(\sigma\cdot
  \mathbf{n}) B_{\eta,\tau,\lambda}]$.
  We consider now $d\in\R$ such that
$d \neq (2k+1)^2\pi^2$, for all $k\in \N_0$.
Dividing \eqref{eq:conditions.equality.2b} by $1+\cos\sqrt{d}$ and using \eqref{eq:conditions.equality.2a}, we get
\begin{equation*}
  \frac{\sin\sqrt{d}}{1+\cos\sqrt{d}} \frac{1}{\sqrt{d}/2}
  (\eta,\tau,\lambda)
  = (\hat{\eta},\hat{\tau},\hat{\lambda}).
\end{equation*}
We conclude the proof applying the elementary identity
  \begin{equation*}
      \tan \frac{\theta}{2} = \frac{\sin\theta }{1+\cos\theta} \quad
      \text{ for all }\theta \in \C \setminus \{(2k+1)\pi \mid k \in
      \Z \},
    \end{equation*}
    and recalling that, for all  $d<0$, we have
    \begin{equation*}
      \frac{\tan\sqrt{d}/2}{\sqrt{d}/2}  =
      \frac{\tanh \sqrt{-d}/2}{\sqrt{-d}/2}.
      \qedhere
    \end{equation*}
  \end{proof}

  By \Cref{lem:exp.isnB}, the function $d := \eta^2 - \tau^2 -\lambda^2 \mapsto \hat
  d = \hat{\eta}^2 - \hat{\tau}^2 -\hat{\lambda}^2$ maps $d\in [0,+\infty)$ to $\hat{d} \in [0,+\infty)$ and $d\in
  (-\infty,0)$ to $\hat{d}\in(-4,0)$. Consequently,
  \ref{item:uniqueness} has a negative answer: the correspondence
  between $(\eta,\tau,\lambda)$ and
  $(\hat{\eta},\hat{\tau},\hat{\lambda})$ is not surjective 
  since one
  can not find $(\eta,\tau,\lambda) \in \R^3$ such that $
  R_{\hat{\eta},\hat{\tau},\hat{\lambda}} = \exp
  [i(\sigma\cdot \mathbf{n}) B_{\eta,\tau,\lambda}]$
  when $\hat{\eta},\hat{\tau},\hat{\lambda}$ are such that $\hat{d} < -4$.
  Moreover, the correspondence is not injective when
  $d \geq 0$, as the following lemma shows.
  
\begin{lemma}\label{lem:going.back.home}
    Let $\hat{\eta},\hat{\tau},\hat{\lambda} \in \R$ such that
    $\hat{d} := \hat{\eta}^2 - \hat{\tau}^2 -\hat{\lambda}^2 > -4$,
    and let $R_{\hat{\eta},\hat{\tau},\hat{\lambda}}$ 
    as in \eqref{def_R_eta_appendix}. Let $\eta,\tau,\lambda \in \R$,
    $d := \eta^2 - \tau^2 -\lambda^2$ and 
    $B_{\eta,\tau,\lambda}$ as in \eqref{defn:B}.
    Then $ R_{\hat{\eta},\hat{\tau},\hat{\lambda}} = \exp
    [i(\sigma\cdot \mathbf{n}) B_{\eta,\tau,\lambda}]$ or equivalently
    \eqref{eq:going.d>0} holds if and only if one of the following holds:
    \begin{align}
    &\bullet\, \hat{d} > 0 \text{ and }
        (\eta,\tau,\lambda) = \frac{\arctan \sqrt{\hat{d}}/2 +
          k\pi}{\sqrt{\hat{d}}/2}
        (\hat{\eta},\hat{\tau},\hat{\lambda}),
         \text{ for } k \in \Z
        \label{eq:back} \\   
     &\begin{aligned}\,\bullet\,\,  &\hat{d} = 0 \text{ and } (\eta,\tau,\lambda) = (\hat{\eta},\hat{\tau},\hat{\lambda}); \text{ for } 
     \hat{d} = \hat{\eta} = \hat{\tau} = \hat{\lambda} = 0, \text{ also
      any } \eta,\tau,\lambda \in \R\\
       &\text{such that } d =(2 k_0\pi)^2, \text{ for } k_0 \in \N, \text{ are admissible} \end{aligned} \nonumber \\
    &\bullet\, -4 < \hat{d} < 0 \text{ and }
        (\eta,\tau,\lambda) = \frac{\arctanh \sqrt{-\hat{d}}/2}{\sqrt{-\hat{d}}/2}
        (\hat{\eta},\hat{\tau},\hat{\lambda}) \nonumber.
    \end{align}
\end{lemma}

\begin{proof} 
  If $\hat{d} \geq 0$ then, by \eqref{eq:going.d>0}, we have $d \geq 0$
  and $\hat{d} = 4 \tan^2 (\sqrt{d}/2)$, that
  gives
  \begin{equation}\label{eq:d.back>0}
    \tan \frac{\sqrt{d}}{2} = \frac{\sqrt{\hat{d}}}{2},
    \quad
    \frac{\sqrt{d}}{2}= \arctan \tfrac{\sqrt{\hat{d}}}{2} +
    k\pi, \quad \text{ for } k \in \Z.
  \end{equation}
  If $\sqrt{d}/2 = k_0\pi$, for some $k_0 \in \N$, then
  \eqref{eq:going.d>0} is true if and only if
  $\hat{\eta}=\hat{\tau}=\hat{\lambda}=0$.
  If $\sqrt{d}/2 = 0$, then \eqref{eq:going.d>0} gives
  $(\hat{\eta},\hat{\tau},\hat{\lambda})
  = (\eta,\tau,\lambda)$.
  If $\sqrt{d}/2 \neq k\pi$, for all $k \in \N_0$, we can
	divide by $\tan{\sqrt{d}/2}$ in
  \eqref{eq:going.d>0}. Using \eqref{eq:d.back>0}, this yields
  \eqref{eq:back}.
  The proof of the case $-4<\hat{d}<0$ is analogous and even simpler, so it
  will be omitted. 
\end{proof}

\section{Magnetic field}
\label{app:magnetic_field}

The term $\lambda(\sigma\cdot\mathbf{t})\delta_\Sigma$ in  \eqref{eq:Dirac.shells} corresponds to the singular vector potential $\mathbf{A}_\Sigma:= \lambda (t_1 \delta_\Sigma,t_2 \delta_\Sigma)$ supported on $\Sigma$. Note that this is just a formal expression--in fact, this term is reflected in the transmission condition across $\Sigma$. We will introduce the magnetic field by the formula $B_\Sigma=\partial_1 A_2-\partial_2 A_1=\lambda(\partial_1 (t_2 \delta_\Sigma)-\partial_2 (t_1 \delta_\Sigma))$, \ie exactly in the same manner as in a regular case. Here $t_i\delta_\Sigma$ is the simple layer  and the derivatives are understood in the sense of distributions. Since $(n_1,n_2)=(t_2 , -t_1)$, we obtain
\begin{multline*}
  \langle B_\Sigma, \varphi \rangle_{\mathcal{D}'(\R^2),
    \mathcal{D}(\R^2)}=-\lambda\Big(\langle  n_1\delta_\Sigma, \partial_1\varphi \rangle_{\mathcal{D}'(\R^2),\mathcal{D}(\R^2)}+\langle  n_2\delta_\Sigma, \partial_2\varphi \rangle_{\mathcal{D}'(\R^2),\mathcal{D}(\R^2)}\Big)\\
    =-\lambda \int_\Sigma \mathbf{n} \cdot \nabla \varphi \,d\sigma
  =\langle \lambda \partial_{\mathbf{n}}\delta_\Sigma, \varphi \rangle_{\mathcal{D}'(\R^2),\mathcal{D}(\R^2)},
\end{multline*}
for all $\varphi \in \mathcal{D}(\R^2)$,
where $\partial_{\mathbf{n}}\delta_\Sigma$ stands for the double layer distribution, \cf \cite{vladimirov}.

Alternatively, thanks to the divergence theorem, we may write
\begin{equation*}
  \langle B_\Sigma, \varphi \rangle_{\mathcal{D}'(\R^2),
    \mathcal{D}(\R^2)}
  = -\lambda
  \int_\Omega \Delta \varphi \,dx
  =
  \langle -\lambda \Delta(\chi_\Omega), \varphi \rangle_{\mathcal{D}'(\R^2),
    \mathcal{D}(\R^2)}.
\end{equation*}

If, for $\epsilon\in(0,\beta)$, we define the vector potential
\begin{equation}\label{eq:defn.Aepsilon2D}
  \mathbf{A}_\epsilon : \R^2 \to \R^2, \quad
  \mathbf{A}_\epsilon(x):=
  \begin{cases}
    \lambda
    h_\epsilon(p)
    \mathbf{t}(x_\Sigma)
    \quad & \text{ for }x = x_\Sigma + p\mathbf{n}(x_\Sigma)
    \in \Sigma_\epsilon,
    \\
    0 \quad & \text{ for }x \in \R^2\setminus \Sigma_\epsilon.
  \end{cases}
\end{equation}
then the corresponding magnetic field $B_\epsilon$ reads
\begin{equation}\label{eq:defn.Bepsilon2D}
  B_\epsilon(x) =\partial_1 \mathbf{A}_{\epsilon,2}(x) -\partial_2
  \mathbf{A}_{\epsilon,1}(x)=
 \begin{cases}
  \frac{\lambda h_\epsilon(p)\kappa(s)}{1+p\kappa(s)}+\lambda h_\epsilon'(p)
  \quad &\text{ for all }x = \gamma(s) + p\mathbf{n}_\gamma(s) \in
  \Sigma_\epsilon, \\
  0 \quad &\text{ for }x \in \R^2 \setminus \Sigma_\epsilon,
\end{cases}
  \end{equation}
and we have:
\begin{proposition}
  Let ${A}_\epsilon$, ${B}_\epsilon$  be defined as in
  \eqref{eq:defn.Aepsilon2D}, \eqref{eq:defn.Bepsilon2D}.
  Then
  \begin{enumerate}[label=$({\roman*})$]
  \item\label{item:limit.A} $\mathbf{A}_\epsilon \xrightarrow[\epsilon \to 0]{} \lambda \mathbf{t}
    \delta_\Sigma=\mathbf{A}_\Sigma$ in the sense of distributions,
  \item\label{item:limit.B} $ B_\epsilon \xrightarrow[\epsilon \to 0]{}
     \lambda \partial_\mathbf{n} \delta_\Sigma=B_\Sigma$ 
    in the sense of distributions.
  \end{enumerate}
\end{proposition}

\begin{proof}
    Let $\gamma$ be an arc-lenght parametrization of $\Sigma$, as in \Cref{sec:tangent.normal}.
  In order to prove \ref{item:limit.A}, let $\varphi \in
  \mathcal{D}(\R^2)$.
  Thanks to \eqref{eq:change.of.variables},
  \begin{equation*}
    \begin{split}
      &\lim_{\epsilon \to 0} \int_{\R^2}
      A_\epsilon(x)\varphi(x)\,dx
      =
      \lim_{\epsilon \to 0} \int_{\Sigma_\epsilon} A_\epsilon(x)\varphi(x)\,dx
      \\
      & =
      \lim_{\epsilon \to 0} \int_0^{\ell}\int_{-\epsilon}^{\epsilon}
      \lambda
      h_\epsilon(p)
      \mathbf{t}_\gamma(s)
     \varphi(\gamma(s) +
        p\mathbf{n}_\gamma(s))
      (1 + p \kappa_\gamma(s))\,dpds
      \\
      & =
      \lim_{\epsilon \to 0} \int_0^{\ell}\int_{-1}^{1}
      \lambda
      h (q)
      \mathbf{t}_\gamma(s)
      \varphi(\gamma(s) +
        \epsilon q \mathbf{n}_\gamma(s))
      (1 + \epsilon q \kappa_\gamma(s))\,dqds
      \\
      & =
      \int_0^{\ell}
      \lambda \mathbf{t}_\gamma(s)
      \varphi(\gamma(s))
      ds
      = \int_\Sigma  \lambda\mathbf{t}(x_\Sigma) \varphi(x_\Sigma) \,dx_\Sigma
      =
      \langle \lambda \mathbf{t}\delta_\Sigma, \varphi
      \rangle_{\mathcal{D}'(\R^2),\mathcal{D}(\R^2)}.
    \end{split}
  \end{equation*}

The second assertion follows from (i) combined with the continuity of distributional derivatives with respect to the convergence on $\mathcal{D}'(\R^2)$.
\end{proof}

Note that the two-dimensional Dirac operator with the magnetic field associated with the vector potential $\mathbf{A}_\epsilon$ is just $\ccE_{0,0,\lambda;\epsilon}$. By Theorem~\ref{thm:approximation}, $\ccE_{0,0,\lambda;\epsilon}$ converges to $\ccD_{0,0,\hat\lambda}$ in the strong resolvent sense, where $\hat\lambda$ is always (except for the trivial case $\lambda=0$) different from $\lambda$. On the other hand, we have just shown that  the formal limit of $\ccE_{0,0,\lambda;\epsilon}$ is $\ccD_{0,0,\lambda}$.


\begin{thebibliography}{10}

\bibitem{abate2012curves}
{\sc M.~Abate and F.~Tovena}, {\em Curves and surfaces}, Springer Science \&
  Business Media, 2012.

\bibitem{albeverio2012solvable}
{\sc S.~Albeverio, F.~Gesztesy, R.~Hoegh-Krohn, and H.~Holden}, {\em {Solvable
  models in quantum mechanics}}, Springer Science \& Business Media, 2012.

\bibitem{antunes2020variational}
{\sc P.~R. Antunes, R.~D. Benguria, V.~Lotoreichik, and
  T.~Ourmi{\`e}res-Bonafos}, {\em {A} variational formulation for {Dirac}
  operators in bounded domains. {Applications} to spectral geometric
  inequalities}, arXiv preprint arXiv:2003.04061,  (2020).
\newblock To appear in Comm.~Math.~Phys.

\bibitem{amv1}
{\sc N.~Arrizabalaga, A.~Mas, and L.~Vega}, {\em {Shell interactions for Dirac
  operators}}, Journal de Math{\'e}matiques Pures et Appliqu{\'e}es, 102
  (2014), pp.~617--639.

\bibitem{amv2}
\leavevmode\vrule height 2pt depth -1.6pt width 23pt, {\em {Shell interactions
  for Dirac operators: on the point spectrum and the confinement}}, SIAM
  Journal on Mathematical Analysis, 47 (2015), pp.~1044--1069.

\bibitem{amv3}
\leavevmode\vrule height 2pt depth -1.6pt width 23pt, {\em {An
  isoperimetric-type inequality for electrostatic shell interactions for Dirac
  operators}}, Communications in Mathematical Physics, 344 (2016),
  pp.~483--505.

\bibitem{approximation}
{\sc J.~{Behrndt}, P.~{Exner}, M.~{Holzmann}, and V.~{Lotoreichik}}, {\em
  {Approximation of Schr\"odinger operators with $\delta$-interactions
  supported on hypersurfaces}}, {Math. Nachr.}, 290 (2017), pp.~1215--1248.

\bibitem{behrndt2016spectral}
\leavevmode\vrule height 2pt depth -1.6pt width 23pt, {\em {On the spectral
  properties of Dirac operators with electrostatic $\delta$-shell
  interactions}}, {J. Math. Pures Appl. (9)}, 111 (2018), pp.~47--78.

\bibitem{behrndt2019dirac}
\leavevmode\vrule height 2pt depth -1.6pt width 23pt, {\em {On Dirac operators
  in $\mathbb{R}^3$ with electrostatic and Lorentz scalar $\delta$-shell
  interactions}}, {Quantum Stud. Math. Found.}, 6 (2019), pp.~295--314.

\bibitem{behrndtboundary}
{\sc J.~Behrndt, S.~Hassi, and H.~de~Snoo}, {\em {Boundary Value Problems, Weyl
  Functions, and Differential Operators}}, Monographs in Mathematics, Springer,
  2020.

\bibitem{behrndt2020limiting}
{\sc J.~{Behrndt}, M.~{Holzmann}, A.~{Mantile}, and A.~{Posilicano}}, {\em
  {Limiting absorption principle and scattering matrix for Dirac operators with
  $\delta$-shell interactions}}, {J. Math. Phys.}, 61 (2020), pp.~033504, 16.

\bibitem{behrndt2019self}
{\sc J.~{Behrndt}, M.~{Holzmann}, and A.~{Mas}}, {\em {Self-adjoint Dirac
  operators on domains in $\mathbb{R}^3$}}, {Ann. Henri Poincar\'e}, 21 (2020),
  pp.~2681--2735.

\bibitem{behrndt2019two}
{\sc J.~{Behrndt}, M.~{Holzmann}, T.~{Ourmi\`eres-Bonafos}, and
  K.~{Pankrashkin}}, {\em {Two-dimensional Dirac operators with singular
  interactions supported on closed curves}}, {J. Funct. Anal.}, 279 (2020),
  p.~46.
\newblock Id/No 108700.

\bibitem{BLL13}
{\sc J.~{Behrndt}, M.~{Langer}, and V.~{Lotoreichik}}, {\em {Schr\"odinger
  operators with $\delta$ and $\delta^{\prime}$-potentials supported on
  hypersurfaces}}, {Ann. Henri Poincar\'e}, 14 (2013), pp.~385--423.

\bibitem{benguria2017self}
{\sc R.~D. Benguria, S.~Fournais, E.~Stockmeyer, and H.~Van Den~Bosch}, {\em
  {Self-adjointness} of two-dimensional {Dirac} operators on domains}, in
  Annales Henri Poincar{\'e}, vol.~18, Springer, 2017, pp.~1371--1383.

\bibitem{benguria2017spectral}
\leavevmode\vrule height 2pt depth -1.6pt width 23pt, {\em {Spectral} gaps of
  {Dirac} operators describing graphene quantum dots}, Mathematical Physics,
  Analysis and Geometry, 20 (2017), p.~11.

\bibitem{benhellal2021}
{\sc B.~{Benhellal}}, {\em Spectral properties of the {Dirac} operator coupled
  with {$\delta$}--shell interactions}, in preparation,  (2021).

\bibitem{bernstein1993some}
{\sc D.~S. Bernstein and W.~So}, {\em Some explicit formulas for the matrix
  exponential}, IEEE Transactions on Automatic Control, 38 (1993),
  pp.~1228--1232.

\bibitem{borrelli2019overview}
{\sc W.~Borrelli, R.~Carlone, and L.~Tentarelli}, {\em An overview on the
  standing waves of nonlinear schr{\"o}dinger and dirac equations on metric
  graphs with localized nonlinearity}, Symmetry, 11 (2019), p.~169.

\bibitem{boussaid2011virial}
{\sc N.~Boussaid, P.~d'Ancona, and L.~Fanelli}, {\em {Virial} identity and weak
  dispersion for the magnetic {Dirac} equation}, Journal de math{\'e}matiques
  pures et appliqu{\'e}es, 95 (2011), pp.~137--150.

\bibitem{BEKS94}
{\sc J.~F. {Brasche}, P.~{Exner}, Y.~A. {Kuperin}, and P.~{\v{S}eba}}, {\em
  {Schr\"odinger operators with singular interactions}}, {J. Math. Anal.
  Appl.}, 184 (1994), pp.~112--139.

\bibitem{bruning2008spectra}
{\sc J.~Br{\"u}ning, V.~Geyler, and K.~Pankrashkin}, {\em Spectra of
  self-adjoint extensions and applications to solvable {Schr{\"o}dinger}
  operators}, Reviews in Mathematical Physics, 20 (2008), pp.~1--70.

\bibitem{CMP13}
{\sc R.~{Carlone}, M.~{Malamud}, and A.~{Posilicano}}, {\em {On the spectral
  theory of Gesztesy-\v{S}eba realizations of 1-D Dirac operators with point
  interactions on a discrete set}}, {J. Differ. Equations}, 254 (2013),
  pp.~3835--3902.

\bibitem{cassano2019self}
{\sc B.~Cassano and V.~Lotoreichik}, {\em {Self-adjoint} extensions of the
  two-valley {Dirac} operator with discontinuous infinite mass boundary
  conditions}, Operators and Matrices, 14 (2020), pp.~667--678.

\bibitem{cassano2018self}
{\sc B.~Cassano and F.~Pizzichillo}, {\em Self-adjoint extensions for the
  {Dirac} operator with {Coulomb}-type spherically symmetric potentials}, Lett.
  Math. Phys., 108 (2018), pp.~2635--2667.

\bibitem{cassano2019boundary}
\leavevmode\vrule height 2pt depth -1.6pt width 23pt, {\em Boundary triples for
  the {D}irac operator with {C}oulomb-type spherically symmetric
  perturbations}, J. Math. Phys., 60 (2019), pp.~041502, 13.

\bibitem{neto2009electronic}
{\sc A.~Castro~Neto, F.~Guinea, N.~M. Peres, K.~S. Novoselov, and A.~K. Geim},
  {\em {The} electronic properties of graphene}, Reviews of modern physics, 81
  (2009), p.~109.

\bibitem{derkach1991generalized}
{\sc V.~Derkach and M.~Malamud}, {\em {Generalized} resolvents and the boundary
  value problems for {Hermitian} operators with gaps}, Journal of Functional
  Analysis, 95 (1991), pp.~1--95.

\bibitem{DoCa_99}
{\sc N.~Dombey and A.~Calogeracos}, {\em Seventy years of the {Klein} paradox},
  Physics Reports, 315 (1999), pp.~41 -- 58.

\bibitem{evans2015measure}
{\sc L.~C. Evans and R.~F. Gariepy}, {\em Measure theory and fine properties of
  functions}, CRC press, 2015.

\bibitem{exner2007leaky}
{\sc P.~Exner}, {\em {Leaky quantum graphs: a review}}, Analysis on graphs and
  its applications,  (2007).

\bibitem{exner2015quantum}
{\sc P.~Exner and H.~Kova{\v{r}}{\'\i}k}, {\em Quantum waveguides}, Springer,
  2015.

\bibitem{fanelli2009non}
{\sc L.~Fanelli}, {\em {Non-trapping} magnetic fields and {Morrey}-{Campanato}
  estimates for {Schr{\"o}dinger} operators}, Journal of mathematical analysis
  and applications, 357 (2009), pp.~1--14.

\bibitem{fanelli2009magnetic}
{\sc L.~Fanelli and L.~Vega}, {\em {Magnetic} virial identities, weak
  dispersion and {Strichartz} inequalities}, Mathematische Annalen, 344 (2009),
  pp.~249--278.

\bibitem{FiJaTu_18}
{\sc M.~Fialová, V.~Jakubský, and M.~Tušek}, {\em Qualitative analysis of
  magnetic waveguides for two-dimensional {D}irac fermions}, Annals of Physics,
  395 (2018), pp.~219 -- 237.

\bibitem{friedrich2000dirac}
{\sc T.~Friedrich}, {\em {Dirac} operators in {Riemannian} geometry}, vol.~25,
  American Mathematical Soc., 2000.

\bibitem{GS87}
{\sc F.~{Gesztesy} and P.~{\v{S}eba}}, {\em {New analytically solvable models
  of relativistic point interactions}}, {Lett. Math. Phys.}, 13 (1987),
  pp.~345--358.

\bibitem{H20}
{\sc M.~Holzmann}, {\em A note on the three dimensional {Dirac} operator with
  zigzag type boundary conditions}, arXiv preprint arXiv:2006.16739,  (2020).

\bibitem{hughes97}
{\sc R.~J. Hughes}, {\em Relativistic point interactions: Approximation by
  smooth potentials}, Reports on Mathematical Physics, 39 (1997), pp.~425--432.

\bibitem{hughes99}
\leavevmode\vrule height 2pt depth -1.6pt width 23pt, {\em Finite-rank
  perturbations of the {D}irac operator}, Journal of Mathematical Analysis and
  Applications, 238 (1999), pp.~67--81.

\bibitem{jost2008riemannian}
{\sc J.~Jost}, {\em {Riemannian} geometry and geometric analysis}, vol.~42005,
  Springer, 2008.

\bibitem{le2018self}
{\sc L.~Le~Treust and T.~Ourmi{\`e}res-Bonafos}, {\em {Self-adjointness} of
  {Dirac} operators with infinite mass boundary conditions in sectors}, Annales
  Henri Poincar{\'e}, 19 (2018), pp.~1465--1487.

\bibitem{lee2006riemannian}
{\sc J.~M. Lee}, {\em Riemannian manifolds: an introduction to curvature},
  vol.~176, Springer Science \& Business Media, 2006.

\bibitem{mas2017dirac}
{\sc A.~Mas}, {\em Dirac operators, shell interactions, and discontinuous gauge
  functions across the boundary}, Journal of Mathematical Physics, 58 (2017),
  p.~022301.

\bibitem{sphericalnote}
{\sc A.~Mas and F.~Pizzichillo}, {\em {The relativistic spherical
  $\delta$-shell interaction in $\mathbb{R}^3$: Spectrum and approximation}},
  Journal of Mathematical Physics, 58 (2017), p.~082102.

\bibitem{klein}
\leavevmode\vrule height 2pt depth -1.6pt width 23pt, {\em {Klein's Paradox and
  the Relativistic $\delta$-shell Interaction in $\mathbb{R}^3$}}, Analysis \&
  PDE, 11 (2018), pp.~705--744.

\bibitem{MaVaPe_09}
{\sc M.~Masir, P.~Vasilopoulos, and F.~Peeters}, {\em Magnetic
  {K}ronig-{P}enney model for {D}irac electrons in single-layer graphene}, New
  Journal of Physics, 11 (2009).

\bibitem{mclean2000strongly}
{\sc W.~McLean and W.~C.~H. McLean}, {\em {Strongly} elliptic systems and
  boundary integral equations}, Cambridge university press, 2000.

\bibitem{ourmieres2019dirac}
{\sc T.~Ourmi{\`e}res-Bonafos and F.~Pizzichillo}, {\em Dirac operators and
  shell interactions: A survey}, in Mathematical Challenges of Zero-Range
  Physics, A.~Michelangeli, ed., Cham, 2021, Springer International Publishing,
  pp.~105--131.

\bibitem{PANKRASHKIN2006207}
{\sc K.~Pankrashkin}, {\em Resolvents of self-adjoint extensions with mixed
  boundary conditions}, Reports on Mathematical Physics, 58 (2006), pp.~207 --
  221.

\bibitem{PR14}
{\sc K.~{Pankrashkin} and S.~{Richard}}, {\em {One-dimensional Dirac operators
  with zero-range interactions: spectral, scattering, and topological
  results}}, {J. Math. Phys.}, 55 (2014), pp.~062305, 17.

\bibitem{PeNe_09}
{\sc V.~Pereira and A.~Castro~Neto}, {\em Strain engineering of graphene’s
  electronic structure}, Phys. Rev. Lett., 103 (2009).

\bibitem{pizzichillo2019self}
{\sc F.~Pizzichillo and H.~V.~D. Bosch}, {\em {Self-Adjointness} of two
  dimensional {Dirac} operators on corner domains}, arXiv preprint
  arXiv:1902.05010,  (2019).

\bibitem{PSS18a}
{\sc F.~{Portmann}, J.~{Sok}, and J.~P. {Solovej}}, {\em {Analysis of zero
  modes for Dirac operators with magnetic links}}, {J. Funct. Anal.}, 275
  (2018), pp.~604--659.

\bibitem{PSS18b}
\leavevmode\vrule height 2pt depth -1.6pt width 23pt, {\em {Self-adjointness
  and spectral properties of Dirac operators with magnetic links}}, {J. Math.
  Pures Appl. (9)}, 119 (2018), pp.~114--157.

\bibitem{PSS20}
\leavevmode\vrule height 2pt depth -1.6pt width 23pt, {\em {Spectral flow for
  Dirac operators with magnetic links}}, {J. Geom. Anal.}, 30 (2020),
  pp.~1100--1167.

\bibitem{posilicano2007self}
{\sc A.~Posilicano}, {\em {Self-adjoint} extensions of restrictions}, Operators
  and Matrices, 2 (2008), pp.~483--506.

\bibitem{rabinovich2020boundary}
{\sc V.~Rabinovich}, {\em Boundary problems for three-dimensional dirac
  operators and generalized mit bag models for unbounded domains}, Russian
  Journal of Mathematical Physics, 27 (2020), pp.~500--516.

\bibitem{reedsimon1}
{\sc M.~Reed and B.~Simon}, {\em {Methods of modern mathematical physics. vol.
  1. Functional analysis}}, Academic Press, New York, 1980.

\bibitem{saranen2013periodic}
{\sc J.~Saranen and G.~Vainikko}, {\em {Periodic} integral and
  pseudodifferential equations with numerical approximation}, Springer Science
  \& Business Media, 2013.

\bibitem{schmidt1995remark}
{\sc K.~M. Schmidt}, {\em {A} remark on boundary value problems for the {Dirac}
  operator}, The Quarterly Journal of Mathematics, 46 (1995), pp.~509--516.

\bibitem{schmudgen2012unbounded}
{\sc K.~Schm{\"u}dgen}, {\em Unbounded self-adjoint operators on Hilbert
  space}, vol.~265, Springer Science \& Business Media, 2012.

\bibitem{sebaklein}
{\sc P.~{\v{S}}eba}, {\em {Klein's paradox and the relativistic point
  interaction}}, Lett. Math. Phys., 18 (1989), pp.~77--86.

\bibitem{thaller}
{\sc B.~Thaller}, {\em The {D}irac equation}, Texts and Monographs in Physics,
  Springer-Verlag, Berlin, 1992.

\bibitem{tusek19}
{\sc M.~Tušek}, {\em Approximation of one-dimensional relativistic point
  interactions by regular potentials revised}, Lett. Math. Phys., 110 (2020),
  pp.~2585--2601.

\bibitem{vladimirov}
{\sc V.~S. Vladimirov}, {\em Equations of Mathematical Physics}, Marcel Dekker,
  1971.

\bibitem{wilcox1967exponential}
{\sc R.~M. Wilcox}, {\em Exponential operators and parameter differentiation in
  quantum physics}, Journal of Mathematical Physics, 8 (1967), pp.~962--982.

\end{thebibliography}

\end{document}